\documentclass[10pt]{amsart}
\usepackage{amsmath}
\usepackage{amsxtra}
\usepackage{amscd}
\usepackage{amsthm}
\usepackage{amsfonts}
\usepackage{amssymb}
\usepackage{eucal}
\usepackage[all]{xy}
\usepackage{graphicx}
\usepackage{comment}
\usepackage{epsfig}
\usepackage{psfrag}
\usepackage{mathrsfs}
\usepackage{amscd}
\usepackage{rotating}
\usepackage{lscape}
\usepackage{amsbsy}
\usepackage{verbatim}
\usepackage{moreverb}
\usepackage{url}
\usepackage{stmaryrd}

\usepackage[usenames]{color}

\newtheorem{prop}[subsubsection]{Proposition}

\newtheorem{cor}[subsubsection]{Corollary}
\newtheorem{lem}[subsubsection]{Lemma}
\newtheorem{defn}[subsubsection]{Definition}

\newtheorem{thm}[subsubsection]{Theorem}

\numberwithin{equation}{section}

\newcommand{\lemref}[1]{Lemma~\ref{#1}}
\newcommand{\thmref}[1]{Theorem~\ref{#1}}
\newcommand{\secref}[1]{Sect.~\ref{#1}}
\newcommand{\corref}[1]{Corollary~\ref{#1}}
\newcommand{\propref}[1]{Proposition~\ref{#1}}

\theoremstyle{remark}
\newtheorem{rem}[subsubsection]{Remark}

\newcommand{\nc}{\newcommand}

\nc{\ssec}{\subsection}
\nc{\sssec}{\subsubsection}

\nc{\renc}{\renewcommand}
\nc{\on}{\operatorname}
\nc\ol{\overline}
\nc\wt{\widetilde}
\nc{\Loc}{\on{Loc}}
\nc{\Bun}{\on{Bun}}
\nc{\BQ}{{\mathbb{Q}}}
\nc{\BA}{{\mathbb{A}}}
\nc{\BC}{{\mathbb{C}}}
\nc{\BH}{{\mathbb{H}}}
\nc{\BG}{{\mathbb{G}}}
\nc{\BK}{{\mathbb{K}}}
\nc{\BN}{{\mathbb{N}}}
\nc{\BP}{{\mathbb{P}}}
\nc{\BD}{{\mathbb{D}}}
\nc{\BV}{{\mathbb{V}}}
\nc{\CA}{{\mathcal{A}}}
\nc{\CE}{{\mathcal{E}}}
\nc{\CC}{{\mathcal{C}}}
\nc{\CO}{{\mathcal{O}}}
\nc{\CP}{{\mathcal{P}}}
\nc{\CQ}{{\mathcal{Q}}}
\nc{\CR}{{\mathcal{R}}}
\nc{\CK}{{\mathcal{K}}}
\nc{\CM}{{\mathcal{M}}}
\nc{\CN}{{\mathcal{N}}}
\nc{\CX}{{\mathcal{X}}}
\nc{\CW}{{\mathcal{W}}}
\nc{\CL}{{\mathcal{L}}}
\nc{\CF}{{\mathcal{F}}}
\nc{\CS}{{\mathcal{S}}}
\nc{\CT}{{\mathcal{T}}}
\nc{\CY}{{\mathcal{Y}}}
\nc{\D}{{\mathcal{D}}}
\nc{\fg}{{\mathfrak{g}}}
\nc{\fD}{{\mathfrak{D}}}
\nc{\fh}{{\mathfrak{h}}}
\nc{\fn}{{\mathfrak{n}}}
\nc{\sM}{{\mathsf M}}
\nc{\ppart}{(\!(t)\!)}
\nc{\hg}{{\widehat\fg}}
\nc{\sA}{{\mathsf A}}
\nc{\sF}{{\mathsf F}}
\nc{\sG}{{\mathsf G}}

\nc{\bC}{{\mathbf{C}}}
\nc{\bZ}{{\mathbf{Z}}}
\nc{\bD}{{\mathbf{D}}}
\nc{\bO}{{\mathbf{O}}}
\nc{\bc}{{\mathbf{c}}}
\nc{\be}{{\mathbf{e}}}
\nc{\bu}{{\mathbf{u}}}
\nc{\bM}{{\mathbf{M}}}
\nc{\bA}{{\mathbf{A}}}
\nc{\bK}{{\mathbf{K}}}

\nc{\fW}{{\mathfrak{W}}}
\nc{\reg}{{\text{\rm reg}}}
\nc{\nilp}{{\text{\rm nilp}}}
\nc{\cG}{{\check{G}}}
\nc{\cB}{{\check{B}}}
\nc{\cg}{{\check{\fg}}}
\nc{\cb}{{\check{\fb}}}
\nc{\cn}{{\check{\fn}}}
\nc{\mer}{{\on{mer}}}
\nc{\Const}{\mathsf{Const}}
\nc{\Whit}{\on{Whit}}
\nc{\KL}{\on{KL}}
\nc{\FS}{\on{FS}}
\nc{\LocSys}{\on{LocSys}}
\nc{\QCoh}{\on{QCoh}}
\nc{\Coh}{\on{Coh}}
\nc{\Cat}{\on{Cat}}
\nc{\Op}{\on{Op}}
\nc{\Gr}{\on{Gr}}
\nc{\Fl}{\on{Fl}}
\nc{\Rep}{\on{Rep}}
\renc{\mod}{{\on{-mod}}}
\nc{\bimod}{{\on{-bimod}}}
\nc{\Conn}{\on{Conn}}
\nc{\unit}{{\mathbf{1}}}
\nc{\uHom}{\underline{\Hom}}

\nc{\Hom}{\on{Hom}}
\nc{\End}{\on{End}}
\nc{\Vect}{\on{Vect}}
\nc{\Funct}{\on{Funct}}
\nc{\Av}{\on{Av}}
\nc{\Ind}{\on{Ind}}
\nc{\KG}{K\backslash G}
\nc{\comult}{{co\on{-}mult}}
\nc{\counit}{{co\on{-}unit}}
\nc{\arw}{{(0\to 1)}}

\nc{\Groupoids}{Grpd}
\nc{\inftygroup}{\infty\on{-Grpd}}
\nc{\inftyCat}{\infty\on{-}\on{Cat}}
\nc{\StinftyCat}{\on{DGCat}}
\nc{\MoninftyCat}{\infty\on{-}\on{Cat}^{\on{Mon}}}
\nc{\SymMoninftyCat}{\infty\on{-}\on{Cat}^{\on{SymMon}}}
\nc{\SymMonStinftyCat}{\on{DGCat}^{\on{SymMon}}}
\nc{\MonStinftyCat}{\on{DGCat}^{\on{Mon}}}
\nc{\inftystack}{\infty\on{-}Stk}
\nc{\inftystackalg}{\infty\on{-}Stk^{1\text{-}alg}}
\nc{\inftyprestack}{\infty\on{-}preStk}
\nc{\inftydgstack}{\infty\on{-}DGStk}
\nc{\inftydgstackalg}{\infty\on{-}DGStk^{1\text{-}alg}}
\nc{\inftydgprestack}{\infty\on{-}DGpreStk}
\nc{\Dmod}{\on{D-mod}}
\nc{\psId}{\on{Ps-Id}}
\nc{\sFD}{\mathsf{Four}}

\nc{\one}{{\mathbf{1}}}
\nc{\sotimes}{\overset{!}\otimes}
\nc{\Maps}{\on{Maps}}
\nc{\CMaps}{\mathcal Maps}
\nc{\Ran}{\on{Ran}}
\nc{\IndCoh}{\on{IndCoh}}
\nc{\ind}{{\mathbf{ind}}}
\nc{\oblv}{{\mathbf{oblv}}}
\nc{\dr}{{\on{dR}}}

\title{Functors given by kernels, adjunctions and duality}

\author{Dennis Gaitsgory}
\address{Department of Mathematics, Harvard University, 1 Oxford street, Cambridge MA 02138, USA}

\date{\today}

\begin{document}

\maketitle

\tableofcontents

\section*{Introduction}

\ssec{The goals of this paper}

The goal of this paper is two-fold. One is to explain a certain phenomenon pertaining to adjoint functors between
DG categories of D-modules on schemes of finite type. Two is to explain what this phenomenon
generalizes to when instead of schemes we consider Artin stacks. 

\sssec{}

We begin by describing the situation with schemes. 

\medskip

We will be working over a ground field $k$ of characteristic $0$. By a scheme we shall mean a scheme of
finite type over $X$. 

\medskip

For a scheme $X$ we let $\Dmod(X)$ the DG category of D-modules on $X$; we refer
the reader to \cite[Sect. 5]{DrGa1}, where the basic properties of this category are discussed. In particular,
the category $\Dmod(X)$ is compactly generated; the corresponding subcategory $\Dmod(X)^c$ of
compact objects identifies with $\Dmod(X)_{\on{coh}}$ of cohomologically bounded objects with
coherent cohomologies. 

\medskip

Let $X_1$ and $X_2$ be a pair of schemes, and let $\sF$ be an (exact) functor 
$$\Dmod(X_1)\to \Dmod(X_2).$$

Assume that $\sF$ is \emph{continuous}, i.e., commutes with colimits (which is equivalent to 
commuting with infinite direct sums). The (DG) category of such functors is equivalent to the category
$\Dmod(X_1\times X_2)$. Namely, an object $\CQ\in \Dmod(X_1\times X_2)$ gives rise to the functor
$$\sF_{X_1\to X_2,\CQ}:\Dmod(X_1)\to \Dmod(X_2), \quad \CM\mapsto \on{pr_2}_\bullet(\on{pr}_1^!(\CM)\sotimes \CQ).$$

Here for a morphism $f$, we denote by $f_\bullet$ the de Rham direct image functor between the corresponding 
DG categories of D-modules, and $\sotimes$ is the usual tensor product functor on the DG category of D-modules
on a scheme. We refer the reader to \secref{sss:kernels and D-modules} for details. 

\medskip

In what follows we shall say that the functor $\sF=\sF_{X_1\to X_2,\CQ}$ is given by the kernel $\CQ$.

\sssec{}

It is a general theorem in the theory of DG categories that a functor $\sF$ as above admits a right adjoint. However,
this right adjoint need not be continuous. In fact, by \lemref{l:when right}, the right adjoint in question is continuous
if and only if the functor $\sF$ preserves compactness, i.e., maps $\Dmod(X_1)^c$ to $\Dmod(X_2)^c$. 

\medskip

Let us, however, assume that the right adjoint of $\sF$, denoted $\sF^R$, is continuous. Then, by the above, it is 
also given by a kernel
$$\CP\simeq \Dmod(X_2\times X_1)=\Dmod(X_1\times X_2).$$

The question that we would like to address is the following: can we explicitly relate the kernels of $\sF$ and $\sF^R$?

\sssec{}

Before we give the answer in general, we consider the following well-known example (more details on this example are supplied in \secref{ss:FD}). 
For a $k$-vector space $V$ considered as a 
scheme, take $X_1=V$ and $X_2=V^\vee$. We let $\sF$ be the Fourier-Deligne transform functor
$\Dmod(V)\to \Dmod(V^\vee)$. It is given by the kernel that we symbolically denote by 
$$\on{exp}\in \Dmod(V\times V^\vee),$$
equal to the pullback of the exponential D-module on $\BA^1$ under the evaluation map $$V\times V^\vee\to \BA^1.$$
We normalize $\on{exp}$ so that it lives in cohomological degree $-\dim(V)$. 

\medskip

As is well-known, $\sF$ is an equivalence of categories. Its inverse (and hence also the right adjoint) is
the Fourier-Deligne transform   $\Dmod(V^\vee)\to \Dmod(V)$, given by the kernel 
$$-\on{exp}:=\BD^{\on{Ve}}_{V\times V^\vee}(\on{exp})[2\dim(V)],$$
where $\BD^{\on{Ve}}_X$ denotes the functor of Verdier duality on a scheme $X$. 

\sssec{}

The assertion of our main theorem in the case when $X_1$ is smooth and separated scheme is 
that the above phenomenon is not specific to the Fourier-Deligne transform, but holds for any
functor $\sF$ that preserves compactness. In fact, this generalization was one of the main initial
motivations for this paper. 

\medskip

Namely, \thmref{t:schemes} says that the kernel $\CP$ defining $\sF^R$ is related to $\CQ$ by the following
formula:
\begin{equation} \label{e:preview smooth}
\CP=\BD_{X_1\times X_2}^{\on{Ve}}(\CQ)[2n_1],
\end{equation}
where $n_1=\dim(X_1)$, and where we remind that $\BD_{X_1\times X_2}^{\on{Ve}}$ is the Verdier duality functor on 
$\Dmod(X_1\times X_2)$. 

\medskip

As we will remark in \secref{sss:other functor}, from \eqref{e:preview smooth} we obtain the following isomorphism
of functors
\begin{equation} \label{e:preview smooth other}
\on{pr_2}_!(\on{pr}_1^\bullet(\CM)\overset{\bullet}\otimes \CQ)[-2n_1]\overset{\sim}\to 
\on{pr_2}_\bullet(\on{pr}_1^!(\CM)\sotimes \CQ)=\sF(M),
\end{equation}
where the functors $\on{pr_2}_!$ and $\overset{\bullet}\otimes$ a priori take values in the correponding
pro-categories \footnote{As our D-modules are not neccesarily holonomic, for a morphism $f$, only
the functors $f^!$ and $f_\bullet$ are defined, whereas their respective left adjoints $f_!$ and $f^\bullet$
take values in the corresponding pro-category.}
(so, in particular, we obtain that the right-hand side in \eqref{e:preview smooth other}
takes values in $\Dmod(X_2)$). 

\sssec{}

Let us consider several most basic examples of the isomorphisms \eqref{e:preview smooth}
and \eqref{e:preview smooth other}. In all these examples we will be assuming that $X_1$
is smooth of dimension $n_1$ and separated. 

\medskip

\noindent{{\bf (i)} Proper pushforward.}  Let $\sF$ be the functor $f_\bullet$, where $f:X_1\to X_2$ is a map. 
In this case $$\CQ=(\on{id}_{X_1}\times f)_\bullet(\omega_{X_1})\in \Dmod(X_1\times X_2),$$
where $\omega_X$ denotes the dualizing complex on a scheme $X$, and $\on{id}_{X_1}\times f$
is the graph of the map $f$. We have $\BD_{X_1\times X_2}^{\on{Ve}}(\CQ)[2n_1]\simeq \CQ$,
so 
$$\sF_{X_2\to X_1,\BD_{X_1\times X_2}^{\on{Ve}}(\CQ)[2n_1]}\simeq f^!.$$ 
Assume that $f$ is proper. In this case $f_\bullet$ preserves compactness. So
\eqref{e:preview smooth} expresses the fact that for $f$ proper, we have $f_\bullet\simeq f_!$ and hence $f^!$ 
is the right adjoint of $f_\bullet$.

\medskip 

\noindent{{\bf (ii)} Smooth pullback.} Let $\sF$ be the functor $f^!$, where $f:X_2\to X_1$ is a \emph{smooth} map; in particular,
the functor $f^!$ preserves compactness. Note that since $X_1$ and $f$ are smooth, $X_2$ is also smooth. We have 
$$\CQ=(f\times \on{id}_{X_2})_\bullet(\omega_{X_2})\simeq (f\times \on{id}_{X_2})_!(k_{X_2})[2n_2],$$
so 
$$\on{pr_2}_!(\on{pr}_1^\bullet(\CM)\overset{\bullet}\otimes \CQ)[-2n_1]\simeq f^\bullet(\CM)[2(n_2-n_1)].$$
Thus, \eqref{e:preview smooth other} amounts to the isomorphism
$$f^!(\CM)\simeq f^\bullet(\CM)[2(n_2-n_1)],$$
which is valid since $f$ is smooth. 

\medskip 

\noindent{{\bf (iii)} Tensor product by a lisse D-module.} 
Let $X_1=X_2=X$, and let $\sF$ be the functor $\CM\mapsto \CM_0\sotimes \CM$, where $\CM_0$ is a lisse D-module
on $X$. In this case $\CQ=(\Delta_X)_\bullet(\CM_0)$. The right adjoint to $\sF$ is given by tensor product with the 
D-module $\BD_X^{\on{Ve}}(\CM_0)[2n]$, which is the assertion of \eqref{e:preview smooth}.

\sssec{The ULA property} 
The next example may be less familiar. Let $X_1$ be as above, and let $f:X_2\to X_1$
be a smooth map. Let $\CN$ be an object in $\Dmod(X_2)$. We consider the functor
$$\sF:\Dmod(X_1)\to \Dmod(X_2),\quad \sF(\CM)=f^!(\CM)\sotimes \CN.$$

We shall say that $\CN$ is \emph{universally locally acyclic} (ULA) with respect to $f$ if the functor $\sF$
preserves compactness. 

\medskip

If $\CN$ is ULA with respect to $f$,  \eqref{e:preview smooth other} says that there is a canonical isomorphism:
$$f^\bullet(\CM)\overset{\bullet}\otimes \CN[-2n_1] \simeq f^!(\CM)\sotimes \CN.$$

\ssec{The case of Artin stacks}

\sssec{}

We now let $\CX_1$ and $\CX_2$ be a pair of quasi-compact Artin stacks, locally of finite type
over $k$. We shall require that both $\CX_1$ and $\CX_2$ be QCA in the sense of \cite{DrGa1}.
This means that the automorphism group of any field-valued point is affine. 

\medskip

The category $\Dmod(\CX)$ is defined for any \emph{prestack} (see \cite[Sect. 6.1]{DrGa1}), and
in particular for any Artin stack. When $\CX$ is a QCA Artin stack, \cite[Theorem 8.1.1]{DrGa1}
says that the category $\Dmod(\CX)$ is compactly generated. Moreover, it is self-dual
(see \secref{sss:Verdier stacks} for what this means).  

\medskip

This implies that the DG category of continuous functors
$$\sF:\Dmod(\CX_1)\to \Dmod(\CX_2)$$
is equivalent to $\Dmod(\CX_1\times \CX_2)$. Namely, to $\CQ\in \Dmod(\CX_1\times \CX_2)$ we assign
the functor $\sF_{\CX_1\to \CX_2,\CQ}$ given by
$$\CM\mapsto \on{pr_2}_\blacktriangle(\on{pr}_1^!(\CM)\sotimes \CQ).$$

Here for a morphism $f$ between QCA stacks we denote by $f_\blacktriangle$ the functor of 
\emph{renormalized direct image}, introduced in \cite[Sect. 9.3]{DrGa1}, and reviewed in \secref{sss:ren}.
Here we will just say that $f_\blacktriangle$ is a colimit-preserving version of $f_\bullet$. 

\sssec{}

We now ask the same question as in the case of schemes: let 
$\sF:\Dmod(\CX_1)\to \Dmod(\CX_2)$ be a continuous functor. Assume that 
the right adjoint of $\sF$ is also continuous (i.e., $\sF$ preserves compactness). What is the
relationship between the kernel $\sF$ and the kernel $\sF^R$? 

\medskip

The answer turns out much more interesting than in the case of smooth separated schemes. To 
formulate it we introduce a certain endo-functor
$$\psId_\CX:\Dmod(\CX)\to \Dmod(\CX)$$
defined for any QCA stack $\CX$. Namely, $\psId_\CX$ is given by the kernel 
$$(\Delta_\CX)_!(k_\CX)\in \Dmod(\CX\times \CX),$$
where $\Delta_\CX$ is the diagonal morphism for $\CX$, and $k_\CX$ is the ``constant sheaf" on $\CX$, i.e., 
the Verdier dual of the dualizing complex $\omega_\CX$.

\medskip

The main theorem for QCA stacks asserts that there is a canonical isomorphism
\begin{equation} \label{e:preview stacks}
\psId_{\CX_1}\circ  (\sF_{\CX_1\to \CX_2,\CQ})^R\simeq  \sF_{\CX_2\to \CX_1,\BD^{\on{Ve}}_{\CX_1\times \CX_2}(\CQ)}.
\end{equation}

I.e., what for a smooth separated cheme $X$ was the functor of cohomological shift by $[-2\dim(X)]$, for a QCA stack becomes
the functor $\psId_{\CX}$.

\sssec{}

The idea of considering the functor $\psId_\CX$ was suggested by V.~Drinfeld. 

\medskip

We should also point out, that
the nature of $\psId_\CX$ (and the kernel $(\Delta_\CX)_!(k_\CX)$, by which it is defined) is not a special feature
of categories of D-modules. Rather, it comes from a certain general manipilation that makes sense for
an arbitrary compactly generated DG category, see \secref{ss:the functor}.

\sssec{}

We would like to draw the reader's attention to the analogy between the isomorphism \eqref{e:preview stacks} and 
the formalism of Grothendieck-Verdier categories of \cite{BoDr}. 

\medskip

Namely, consider \cite[Example 2.2]{BoDr}, where $X$ scheme (or, more generally, a \emph{safe} algebraic stack)
of finite type, and $\Gamma$ is the groupoid $X\times X$. 

\medskip

Let $\CM:=\Dmod(X\times X)_{\on{hol}} \subset \Dmod(X\times X)$ be the holonomic subcategory.
We consider $\CM$ as a monoidal category, endowed with the convolution product (denoted $\circ$), 
where, in order to be consistent with \cite{BoDr}, we now use !-pushforward and $\bullet$-pullback 
(instead of the $\bullet$-pushforward and !-pullback).  Then $\CM$ is a Grothendieck-Verdier category, 
with the dualizing object being $(\Delta_X)_\bullet(\omega_X)$. 

\medskip

Let $\CQ\in \Dmod(X\times X)_{\on{hol}}$ be such that the corresponding functor $\sF_{X\to X,\CQ}$
admits a right adjoint, given by a holonomic kernel. Denote $\CP:=\BD^{\on{Ve}}_{X\times X}(\CQ)$. 
Then $\CP$, regarded as an object of the monoidal category $\CM$,
admits a \emph{left} rigid dual, denoted in the notation of \cite{BoDr} by $\underline\Hom'(\CP,{\mathbf 1})$.

\medskip 

The isomorphism \eqref{e:preview stacks} is equivalent to an isomorphism in $\Dmod(X\times X)_{\on{hol}}$
$$(\Delta_X)_\bullet(\omega_X)\circ \underline\Hom'(\CP,{\mathbf 1})\simeq \BD^{\on{Ve}}_{X\times X}(\CP),$$
valid for any left-dualizable object in a Grothendieck-Verdier category. 

\medskip

However, unfortunately, we were unable to formally apply the formalism of \cite{BoDr} to deduce our 
\eqref{e:preview stacks} in general.

\sssec{}

Next we consider the case of non-quasi compact Artin stacks. We will require that our stacks
be locally QCA and \emph{truncatable}, see \secref{sss:truncatable} for what this means. 
The main example of a truncatable stack that we have in mind is $\Bun_G$--the moduli
stack of $G$-bundles on $X$, where $G$ is a reductive group and $X$ is a smooth and
complete curve. 

\medskip

For a truncatable stack $\CX$ there are two categories of D-modules that one can attach to it.
One is the usual category $\Dmod(\CX)$, and the other is the category that we denote by
$\Dmod(\CX)_{\on{co}}$, whose definition uses the truncatability of $\CX$, see \secref{sss:co}. 

\medskip

To elucidate the nature of $\Dmod(\CX)$ and $\Dmod(\CX)_{\on{co}}$ let us describe their
respective categories of compact objects. An object of $\Dmod(\CX)$ is compact if
and only if it equals the !-extension from a compact object on a quasi-compact open substack 
of $\CX$. The category of compact objects of $\Dmod(\CX)_{\on{co}}$ also embeds fully
faithfully into $\Dmod(\CX)$, and its essential image consists of objects that are 
*-extensions from compact objects on quasi-compact open substacks of $\CX$. 
\footnote{One can informally think of $\Dmod(\CX)$ and $\Dmod(\CX)_{\on{co}}$ as obtained by imposing 
different ``growth" conditions.} 

\medskip

What for a QCA stack was the Verdier duality self-equivalence of the DG category of D-modules,
for a truncatable stack becomes an equivalence between the dual of $\Dmod(\CX)$ and 
$\Dmod(\CX)_{\on{co}}$. This implies that for a pair of truncatable stacks $\CX_1$ and $\CX_2$, 
the DG category
$$\Dmod(\CX_1\times \CX_2)$$
is equivalent to that of continuous functors
$$\Dmod(\CX_1)_{\on{co}}\to \Dmod(\CX_2).$$

In particular, for a QCA stack $\CX$ there exists a canonically defined functor
$$\psId_\CX:\Dmod(\CX)_{\on{co}}\to \Dmod(\CX),$$
given by the kernel $(\Delta_\CX)_!\in \Dmod(\CX\times \CX)$.

\medskip

The DG category of continuous functors
$$\Dmod(\CX_1)_{\on{co}}\to \Dmod(\CX_2)_{\on{co}},$$
which is the same as that of continuous functors $\Dmod(\CX_2)\to \Dmod(\CX_1)$, 
is equivalent to the tensor product category
$$\Dmod(\CX_1)\otimes \Dmod(\CX_2)_{\on{co}}.$$

\medskip

Finally, we note that for a \emph{coherent} object $\CQ\in \Dmod(\CX_1)\otimes \Dmod(\CX_2)_{\on{co}}$
there is a well-defined Verdier dual $\BD^{\on{Ve}}_{\CX_1\times \CX_2}(\CQ)$, which is an object of
$\Dmod(\CX_1)\otimes \Dmod(\CX_2)$.

\sssec{}

The main theorem for non-quasi compact stacks reads as follows. Let $$\sF:\Dmod(\CX_1)_{\on{co}}\to \Dmod(\CX_2)_{\on{co}}$$
be a continuous functor given by a coherent kernel
$$\CQ\in \Dmod(\CX_1)\otimes \Dmod(\CX_2)_{\on{co}}.$$
Assume that $\sF$ admits a continuous right adjoint (equivalently, $\sF$ preserves compactness). Then we have
a canonical isomorphism of functors $\Dmod(\CX_2)_{\on{co}}\to \Dmod(\CX_1)$:
\begin{equation} \label{e:preview nonqc}
\psId_{\CX_1}\circ (\sF_{\CX_1\to \CX_2,\CQ})^R\simeq \sF_{\CX_2\to \CX_1,\BD^{\on{Ve}}_{\CX_1\times \CX_2}(\CQ)}.
\end{equation}

We note that in the right-hand side, the object $\BD^{\on{Ve}}_{\CX_1\times \CX_2}(\CQ)$ belongs to 
$$\Dmod(\CX_1\times \CX_2)=\Dmod(\CX_2\times \CX_1),$$ and hence the functor 
$\sF_{\CX_2\to \CX_1,\BD^{\on{Ve}}_{\CX_1\times \CX_2}(\CQ)}$ is understood as a functor
$$\Dmod(\CX_2)_{\on{co}}\to  \Dmod(\CX_1).$$

So, the initial kernel and its Verdier dual define functors between different categories, and the connection
is provided by the functor $\psId_{\CX_1}$, which maps $\Dmod(\CX_1)_{\on{co}}\to \Dmod(\CX_1)$. 
 
\ssec{Contents}

We shall now review the contents of this paper section-by-section.

\sssec{} 

In \secref{s:schemes} we state the main theorem pertaining to schemes, \thmref{t:schemes},
which generalizes the isomorphism \eqref{e:preview smooth} to the case when the scheme $X_1$
is not necessarily smooth and separated. We discuss various corollaries and
particular cases of \thmref{t:schemes}. 

\medskip

We intersperse the discussion about functors between
categories of D-modules with a review of some basic facts concerning DG categories. 

\sssec{}

In \secref{s:nat trans} we give a geometric description of a canonical natural transformation 
\begin{equation} \label{e:peview nat trans}
\on{pr_2}_!(\on{pr}_1^\bullet(\CM)\overset{\bullet}\otimes \psId_{X_1}(\CQ))\to \on{pr_2}_\bullet(\on{pr}_1^!(\CM)\sotimes \CQ),
\end{equation}
which is an isomorphism whenever the functor $\sF_{X_1\to X_2,\CQ}$, defined by $\CQ$, preserves compactness. 

\medskip

When the scheme $X_1$ is separated, the basic ingredient of the map \eqref{e:peview nat trans}
is the natural transformation $f_!\to f_\bullet$ for a separated morphism $f$ between schemes,
and the natural transformation
$$g_1^\bullet \circ f_0^! \to f_1^! \circ g_0^\bullet$$
for a Cartesian diagram
$$
\CD
Y_{11} @>{f_1}>>  Y_{1,0} \\
@V{g_1}VV   @VV{g_0}V  \\
Y_{0,1}  @>{f_0}>> Y_{0,0}.
\endCD
$$

\sssec{}

In \secref{s:O} we study the following question. Let $X_i$, $i=1,2$ be derived schemes, and let
$\CQ$ be an object of $\Dmod(X_1\times X_2)$. 

\medskip

Recall that for a derived scheme $X$ there are natural 
forgetful functors
$$\oblv_X^{\on{left}}:\Dmod(X) \to \QCoh(X) \text{ and } \oblv_X:\Dmod(X) \to \IndCoh(X),$$
where $\QCoh(X)$ is the DG category of quasi-coherent sheaves on $X$ and $\IndCoh(X)$
is its modification introduced in \cite{IndCoh}. The functors $\oblv_X^{\on{left}}$ and $\oblv_X$
are the realizations of D-modules on $X$ as ``left" and ``right" D-modules, respectively
(see \cite[Sect. 2.4]{Crys} for more details). 

\medskip

We would like to know how to express the condition that the functor
$$\sF_{X_1\to X_2,\CQ}:\Dmod(X_1)\to \Dmod(X_2),$$
corresponding to $\CQ\in \Dmod(X_1\times X_2)$, preserve compactness, in terms of the 
corresponding objects
$$\oblv^{\on{left}}_{X_1}\otimes \on{Id}_{\Dmod(X_2)}(\CQ)\in \QCoh(X_1)\otimes \Dmod(X_2)$$ and 
$$\oblv_{X_1}\otimes \on{Id}_{\Dmod(X_2)}(\CQ)\in \IndCoh(X_1)\otimes \Dmod(X_2).$$

\medskip

For example, we show that if the support of $\CQ$ is proper over $X_2$, then 
$\sF_{X_1\to X_2,\CQ}$ preserves compactness if and only if $\oblv_{X_1}\otimes \on{Id}_{\Dmod(X_2)}(\CQ)$
is compact. 

\sssec{}

In \secref{s:proof of subq} we prove the following, perhaps a little unexpected, result: 

\medskip

Let $X_1$ be quasi-projective and smooth. Then then the property of an object $$\CQ\in \Dmod(X_1\times X_2)$$
that the corresponding functor
$$\sF_{X_1\to X_2,\CQ}:\Dmod(X_1)\to \Dmod(X_2)$$
preserve compactness is inherited by any subquotient of any cohomology of $\CQ$ (with respect to the
standard t-structure on $\Dmod(X_1\times X_2)$). 

\sssec{}

In \secref{s:delo} we prove our main result pertaining to functors between D-modules
on schemes, namely, \thmref{t:schemes}. In fact, we prove a more general assertion,
in the general context of DG categories, namely, \thmref{t:general}. 

\medskip

In more detail, \thmref{t:general} describes the following situation. We start with
a continuous functor between DG categories
$$\sF:\bC_1\to \bC_2,$$
given by a kernel $\CQ\in \bC_1^\vee\otimes \bC_2$, and we assume that the right adjoint
of $\sF$ is also continuous. We want to relate this right adjoint $\sF^R$ to the kernel
\emph{dual} to $\CQ$, which is an object
$$\CQ^\vee\in  \bC_1\otimes \bC^\vee_2.$$

\medskip

In describing the relation, we will encounter an endo-functor $\psId_{\bC}$, defined for any
DG category $\bC$. When $\bC=\Dmod(\CX)$, where $\CX$ is a QCA stack, the corresponding
endo-functor is $\psId_\CX$ mentioned above.

\medskip

At the suggestion of Drinfeld, we also introduce the notion of Gorenstein category. Namely, 
this is a DG category $\bC$ for which the functor $\psId_\bC$ is an equivalence. The name
Gorenstein is explained by the following result: for a separated derived scheme $X$ almost of
finite type, the category $\QCoh(X)$ is Gorenstein if and only if $X$ is Gorenstein.

\sssec{}

In \secref{s:Artin qc} our goal is to generalize \thmref{t:schemes} to the case of QCA stacks.
The generalization itself, \thmref{t:stacks}, will be easy to carry out: the corresponding theorem follows from
the general result about DG categories, namely, \thmref{t:general}. 

\medskip

However, there are two important technical points that one needs to pay attention to
in the case of Artin stacks (as opposed to schemes or Deligne-Mumford stacks).

\medskip

First, for a scheme $X$ , the subcategory $\Dmod(X)^c$ of compact objects in $\Dmod(X)$
is the same as $\Dmod(X)_{\on{coh}}$, i.e., the subcategory spanned by cohomologically
bounded objects with coherent cohomologies. This is no longer the case for stacks: for
a QCA stack $\CX$ we always
have an inclusion 
$$\Dmod(\CX)^c\subset \Dmod(\CX)_{\on{coh}},$$
which is an equality if and only if $\CX$ is \emph{safe}. 

\medskip

Second, for a non-schematic map $f:\CX_1\to \CX_2$, the usual de Rham direct image functor
$$f_\bullet:\Dmod(\CX_1)\to  \Dmod(\CX_2)$$
may be ill-behaved (e.g., fails to be continuous). This applies in particular to the functor of de Rham
cohomology $(p_\CX)_\bullet$ of a stack $\CX$, where $p_\CX:\CX\to \on{pt}:=\on{Spec}(k)$. 
To remedy this, one replaces $f_\bullet$ by its \emph{renormalized} version,
introduced in \cite{DrGa1}, and denoted $f_\blacktriangle$. 

\medskip

In the remainder of \secref{s:Artin qc} we consider some applications of \thmref{t:stacks}. For example,
we consider the situation of an open embedding of stacks $j:U\hookrightarrow \CX$, for which
the functor $j_\bullet:\Dmod(U)\to \Dmod(\CX)$ preserves compactness (such open embeddings
are in \cite{DrGa2} called \emph{co-truncative}), and see what \thmref{t:stacks} gives in this case.

\medskip

We also consider the class of stacks $\CX$, for which the functor of de Rham cohomology preserves compactness;
we call such stacks mock-proper. For a mock-proper stack $\CX$ we relate the 
functor $(p_\CX)_\blacktriangle$ (i.e., the renormalized version of de Rham cohomology) and $(p_\CX)_!$
(the functor of de Rham cohomology with compact supports). 

\medskip

Finally, we consider a particular example of a QCA stack, namely, $V/\BG_m$, where $V$ is a vector space with $\BG_m$
acting by dilations. We show that $\Dmod(V/\BG_m)$ is Gorenstein. 

\sssec{}  In \secref{s:Artin gen} we state and prove the theorem relating the adjoint functor
to the Verdier dual kernel for locally QCA truncatable stacks.

\medskip

We first review the definition of what it means for a QCA tack $\CX$ to be truncatable, and introduce the
two versions of the category of D-modules, $\Dmod(\CX)$ and $\Dmod(\CX)_{\on{co}}$. 

\medskip

We proceed to stating and proving \thmref{t:stacks nonqc}, which amounts to the isomorphism
\eqref{e:preview nonqc}. 

\medskip

In the remainder of this section we consider applications of \thmref{t:stacks nonqc}, most of which are
straightforward modifications of the corresponding statements for QCA stacks, once we take into 
account the difference between $\Dmod(\CX)$ and $\Dmod(\CX)_{\on{co}}$. 

\medskip

Finally, we consider the notion of a mock-proper truncatable stack and define the \emph{mock-constant sheaf}
on such a stack. We consider the particular case of $\CX=\Bun_G$, and show that its mock-constant sheaf
$k_{\Bun_G,\on{mock}}$ has some peculiar properties; this object would be invisible if one did not distinguish
between $\Dmod(\CX)$ and $\Dmod(\CX)_{\on{co}}$. 

\ssec{Conventions and notation}

\sssec{}   \label{sss:dag}

The word ``scheme" in this paper means ``derived scheme almost of finite type over $k$, which is eventually
coconnective." We refer the reader to \cite[Sect. 3.1.1]{DrGa1}, where this notion is reviewed.  Sometimes
(more as a matter of convenience) we will use the term ``prestack," by which we will always mean a prestack
locally almost of finite type, see \cite[Sect. 3.1]{DrGa1}.  

\medskip

The ``good news" is that derived algebraic geometry is not needed in this paper, except
in Sects. \ref{s:O} and \ref{s:proof of subq}, in which $\CO$-modules are discussed. So, since the material
of these two sections is not used in the rest of the paper, the reader can skip them and 
work with ordinary schemes of finite type over $k$. 

\medskip

For a scheme/prestack $\CX$ we denote by $p_\CX$ the tautological 
map $\CX\to \on{pt}:=\on{Spec}(k)$. We let $\Delta_\CX$ denote the diagonal morphism $\CX\to \CX\times \CX$.  

\sssec{}    

Conventions and notations regarding DG categories adopted in this papar follow those reviewed in \cite[Sect. 1]{DrGa2}. 

\medskip

In particular, we let $\Vect$ denote the category of chain complexes of $k$-vector spaces. 

\medskip

For a DG category $\bC$ and $\bc_1,\bc_2\in \bC$ we let $\CMaps_\bC(\bc_1,\bc_2)\in \Vect$ denote the resulting chain
complex of maps between them. 

\sssec{}

Conventions and notations regarding the category of D-modules on a scheme follow those of \cite[Sect. 5]{DrGa1},
and on algebraic stack those of \cite[Sect. 6]{DrGa1}. See also \cite[Sect. 2]{DrGa2} (for a brief review), and \cite{Crys}
for a systematic treatment of the foundadtions of the theory. 

\medskip

The
only notational difference between the present paper and \cite{DrGa1}
is that the functor of de Rham direct image with respect to a morphism $f$ is denoted here
by $f_\bullet$ instead of $f_{\dr,*}$. 

\medskip

For a morphism $f:\CX_1\to \CX_2$ between prestacks, we have a tautologically defined functor 
$$f^!:\Dmod(\CX_2)\to \Dmod(\CX_1).$$

The symmetric monoidal structure, denoted $\sotimes$, on the category of D-modules on a prestack $\CX$
is defined by
$$\CM_1\sotimes \CM_2:=\Delta_\CX^!(\CM_1\boxtimes \CM_2).$$

\sssec{} \label{sss:bad functors}

The \emph{partially defined} left adjoint of the functor $f^!$ will be denoted by $f_!$. I.e., for $\CM_1\in \Dmod(\CX_1)$,
the object $f_!(\CM_1)\in  \Dmod(\CX_1)$ is defined if and only if the functor
$$\CM_2\mapsto \CMaps_{\Dmod(\CX_1)}(\CM_1,f^!(\CM_2)),\quad \Dmod(\CX_2)\to \Vect$$
is co-representable. 

\medskip

For a general $\CM_1$, we can view $f_!(\CM_1)$ as an object of $\on{Pro}(\Dmod(\CX_2))$, the pro-completion of
$\Dmod(\CX_1)$. \footnote{For a DG category $\bC$, its pro-completion $\on{Pro}(\bC)$ is the category
of all exact covariant functor $\bC\to \Vect$ that commute with $\kappa$-filtered colimits for some sufficiently
large cardinal $\kappa$.} 

\medskip

For a morphism $f: \CX_1\to \CX_2$ between \emph{Artin stacks} we have the 
functor of de Rham direct image $$f_\bullet:\Dmod(\CX_1)\to \Dmod(\CX_2)$$ see \cite[Sect. 7.4.]{DrGa1}. Its partially
defined left adjoint is denoted $f^\bullet$. 

\medskip

We let $\overset{\bullet}\otimes$ denote the partially defined functor
$$\Dmod(\CX)\otimes \Dmod(\CX)\to \Dmod(\CX)$$
equal to 
$$\CM_1,\CM_2\mapsto (\Delta_\CX)^\bullet(\CM_1\boxtimes \CM_2).$$
I.e., it is defined an on object of $\Dmod(\CX)\otimes \Dmod(\CX)$,
whenever $(\Delta_\CX)^\bullet$ is defined on the corresponding object of $\Dmod(\CX\times \CX)$. 

\medskip

As in the case of $f_!$, in general, the functors $f^\bullet$ and $\overset{\bullet}\otimes$ can be viewed 
as taking values in the pro-completion of the target category. 

\medskip

If $\CX_1$ and $\CX_2$ are Artin stacks, the functors $f_!$, $f^\bullet$ and $\overset{\bullet}\otimes$
are defined on 
any \emph{holonomic} object, i.e., one whose pullback to a scheme mapping smoothly to our
stack has holonomic cohmologies. 

\ssec{Acknowledgments} 

The author would like to thank V.~Drinfeld for many helpful discussions and suggestions related to
this paper. The definition of the key player in the case of Artin stacks, namely, the functor $\psId_\CX$, 
is due to him. 

\medskip

The author is supported by NSF grant DMS-1063470.

\section{Functors between categories of D-modules}  \label{s:schemes}

In this section we state our main theorem in the case of schemes (\thmref{t:schemes}) and discuss
its corollaries. 

\ssec{Continuous functors and kernels: recollections}

\sssec{}  \label{sss:functors and kernels}

Let $\bC$ be a dualizable category, and let $\bC^\vee$ denote its dual. We let
$$\bu_\bC\in \bC\otimes \bC^\vee$$
denote the object correspnding to the unit map
$$\Vect\to \bC\otimes \bC^\vee,$$
and we let
$$\on{ev}_\bC:\bC\otimes \bC^\vee\to \Vect$$
denote the counit map.

\medskip

Let $\bC_1$ and $\bC_2$ be two DG categories. Recall that an exact functor $\sF:\bC_1\to \bC_2$
is said to be \emph{continuous} if it commutes with infinite direct sums (equivalently, all colimits). 
We let $\on{Funct}_{\on{cont}}(\bC_1,\bC_2)$ denote the full DG subcategory of the DG category
$\on{Funct}(\bC_1,\bC_2)$ of all DG functors $\bC_1\to \bC_2$, spanned by continuous functors. 

\medskip

Assume that $\bC_1$ dualizable. In this case, the 
DG category $\on{Funct}_{\on{cont}}(\bC_1,\bC_2)$
identifies with 
$$\bC_1^\vee\otimes \bC_2.$$

Explicitly, an object $\CQ\in \bC_1^\vee\otimes \bC_2$ gives rise to the functor $\sF_{\bC_1\to \bC_2,\CQ}$ equal to
$$\bC_1\overset{\on{Id}_{\bC_1}\otimes \CQ}\longrightarrow \bC_1\otimes \bC_1^\vee\otimes \bC_2
\overset{\on{ev}_{\bC_1}\otimes \on{Id}_{\bC_2}}\longrightarrow \bC_2.$$

\medskip

Vice versa, given a continuous functor $\sF:\bC_1\to \bC_2$ we construct the corresponding object
$\CQ_\sF\in \bC_1^\vee\otimes \bC_2$ as
$$(\on{Id}_{\bC_1^\vee}\otimes \sF) (\bu_{\bC_1}).$$

In particular, $\bu_{\bC_1}\in \bC_1^\vee\otimes \bC_1$ corresponds to the identity functor on $\bC_1$.

\medskip

We shall refer to $\CQ_\sF$ as the \emph{kernel} of $\sF$, and to $\sF_{\bC_1\to \bC_2;\CQ}$ as the \emph{functor defined by}
$\CQ$.

\sssec{}

Let $\bC$ be a compactly generated category. Recall that in this case
$$\bC\simeq \Ind(\bC^c),$$
where $\on{Ind}(-)$ denotes the ind-completion of a given small DG category.

\medskip

Recall also that such $\bC$ is dualizable, and we have a canonical equivalence
\begin{equation} \label{e:compact of dual}
(\bC^c)^{\on{op}}\simeq (\bC^\vee)^c, \quad \bc\mapsto \bc^\vee.
\end{equation}

In particular,
\begin{equation} \label{e:dual as Ind}
\bC^\vee\simeq \on{Ind}((\bC^c)^{\on{op}}).
\end{equation}

Under this identification for $\bc\in \bC^c$ and $\xi\in \bC^\vee$ we have
\begin{equation} \label{e:eval dual}
\CMaps_{\bC^\vee}(\bc^\vee,\xi)\simeq \on{ev}_\bC(\bc\otimes \xi).
\end{equation}

\sssec{}

The following simple observation will be used throughout the paper (see, e.g., \cite[Proposition 1.2.4]{DrGa2} for the proof):

\begin{lem} \label{l:when right}
Let $\sF:\bC_1\to \bC_2$ be a continuous functor. If $\sF$ admits a continuous right adjoint,
then it preserves compactness. Vice versa, if  $\sF$ preserves compactness and $\bC_1$ is
compactly generated, then $\sF$ admits a continuous right adjoint.
\end{lem}

\medskip

Let us note the following consequence of \lemref{l:when right}:

\begin{cor} \label{c:preserve compactness tensor}
Let $\bC_1$ be compactly generated, and let $\sF:\bC_1\to \bC_2$ be a continuous functor
that preserves compactness. Then for any DG category $\bC$ the functor 
$$\bC\otimes \bC_1\overset{\on{Id}_\bC\otimes \sF}\longrightarrow \bC\otimes \bC_2$$
also preserves compactness. 
\end{cor}

\begin{proof}
By \lemref{l:when right}, the functor $\sF$ admits a continuous right adjoint; denote it $\sG$.
Hence, the functor $\on{Id}_\bC\otimes \sF$ also admits a continuous right adjoint, namely,
$\on{Id}_\bC\otimes \sG$. Now, apply \lemref{l:when right} again.
\end{proof}

\ssec{Continuous functors and kernels: the case of D-modules}

\sssec{}

Let $X$ be a scheme of finite type over $k$.  (We remind that in the present section, as well
as elsewhere in the paper with the exception of Sects. \ref{s:O} and \ref{s:proof of subq}, we can
work within classical algebraic geometry.)

\medskip

Recall (see e.g., \cite[Sect. 5.3.4]{DrGa1}) that the DG category $\Dmod(X)$ canonically identifies
with its own dual:
$$\bD^{\on{Ve}}_X:\Dmod(X)^\vee \simeq \Dmod(X),$$
where the corresponding equivalence on compact objects
$$(\Dmod(X)^c)^{\on{op}}=(\Dmod(X)^\vee)^c \overset{\bD^{\on{Ve}}_X}\longrightarrow \Dmod(X)^c$$
is the usual Verdier duality functor
$$\BD^{\on{Ve}}_X:(\Dmod(X)^c)^{\on{op}}\to \Dmod(X)^c.$$

\medskip

Let us also recall the corresponding evaluation and unit functors. For this we recall  (see, e.g., \cite[Sect. 5.1.7]{DrGa1}) 
that if $X_1$ and $X_2$ are two schemes of finite type, the operation of external tensor product of D-modules defines
an equivalence
$$\Dmod(X_1)\otimes \Dmod(X_2)\simeq \Dmod(X_1\times X_2).$$

\medskip

Under the equivalence $\Dmod(X)\otimes \Dmod(X)\simeq \Dmod(X\times X)$, 
the evaluation functor
$$\on{ev}_{\Dmod(X)}:\Dmod(X)\otimes \Dmod(X)\to \Vect$$
is $(p_X)_\bullet\circ (\Delta_X)^!$, where
$p_X:X\to \on{pt}$ and $\Delta_X$ is the diagonal map $X\to X\times X$.

\medskip

The unit object $$\bu_{\Dmod(X)}\in \Dmod(X)^\vee \otimes \Dmod(X)\simeq \Dmod(X)\otimes \Dmod(X) \simeq 
\Dmod(X\times X)$$
is $(\Delta_X)_\bullet(\omega_X)$, where
$$\omega_X:=p_X^!(k)$$
is the dualizing complex.

\sssec{}  \label{sss:kernels and D-modules}

Let $X_1$ and $X_2$ be two schemes of finite type. By \secref{sss:functors and kernels},
the DG category $$\on{Funct}_{\on{cont}}(\Dmod(X_1),\Dmod(X_2))$$
of continuous functors $\Dmod(X_1)\to \Dmod(X_2)$ identifies with 
$$\Dmod(X_1)^\vee\otimes \Dmod(X_2),$$
and further, using the equivalence $\bD_{X_1}$, with
$$\Dmod(X_1)\otimes \Dmod(X_2)\simeq \Dmod(X_1\times X_2).$$

I.e., continuous functors $\Dmod(X_1)\to \Dmod(X_2)$ are in bijection with kernels, thought of
objects of $\Dmod(X_1\times X_2)$. 

\medskip

Explicitly, for $\CQ\in \Dmod(X_1\times X_2)$ the corresponding 
functor $\sF_{X_1\to X_2;\CQ}$ sends an object $\CM\in \Dmod(X_1)$ to 
$$(\on{pr}_2)_\bullet(\on{pr}_1^!(\CM)\sotimes \CQ),$$
where $\sotimes$ denotes the tensor product on the category of D-modules (see \cite[Sect. 5.1.7]{DrGa1}),
and $$\on{pr}_i:X_1\times X_2\to X_i,\quad i=1,2$$
are the two projections.

\sssec{}

The question we would like to address in this section is the following: suppose that a functor
$\sF:\Dmod(X_1)\to \Dmod(X_2)$ admits a a \emph{continuous} right adjoint $\sF^R$ or 
a left adjoint $\sF^L$ (the latter is automatically continuous). 

\medskip

We would like to relate the kernels of the functors $\sF^R$ or $\sF^L$ (and also of the \emph{conjugate}
functors, see \secref{sss:conj functors}) to that of $\sF$. 

\medskip

The relationship will be particularly explicit when $X_1$ and $X_2$ are separated and smooth,
see Sects. \ref{sss:thm for smooth} and \ref{sss:thm for smooth conj}. In the case of arbitrary schemes of 
finite type, the corresponding assertion is stated in
Sects. \ref{sss:thm for sch} and  \ref{sss:thm for sch conj}. The situation becomes significantly more 
interesting when instead of schemes, we consider Artin stacks, see Sects. \ref{s:Artin qc} and
\ref{s:Artin gen}. 

\sssec{}

Note that from \lemref{l:when right}, we obtain:

\begin{cor} \label{c:when right D}
Let $\CQ$ be an object of $\Dmod(X_1\times X_2)$. Then 
the functor $$\sF_{X_1\to X_2,\CQ}:\Dmod(X_1)\to \Dmod(X_2)$$ admits 
a \emph{continuous} right adjoint if and only if it preserves
compactness.
\end{cor}

\ssec{Statement of the theorem: the case of schemes}

\sssec{}

Note that for a scheme of finite type $X$, the object $\omega_X\in \Dmod(X)$ is compact. 

\medskip

We let 
$$k_X:=\BD^{\on{Ve}}_X(\omega_X)\in \Dmod(X)^c.$$

By definition, $k_X$ is the D-module incarnation of the constant sheaf on $X$.  As is well-known, if $X$ is smooth
(or rationally smooth) of dimension $n$, we have 
\begin{equation} \label{e:smooth for dualizing}
k_X\simeq \omega_X[-2n].
\end{equation}

\sssec{}

A fundamental role in this paper is played by the object 
$$(\Delta_X)_!(k_X)\in \Dmod(X\times X).$$
The object $(\Delta_X)_!(k_X)$ is well-defined (see \secref{sss:bad functors}) because $k_X$ is holonomic. Note also
that if $X$ is separated, 
\begin{equation} \label{e:separated for dualizing}
(\Delta_X)_!(k_X)\simeq (\Delta_X)_\bullet(k_X).
\end{equation}

\medskip

We let
$$\psId_{X}:\Dmod(X)\to \Dmod(X)$$
denote the functor, given by the kernel $(\Delta_X)_!(k_X)$.

\medskip

Note that when $X$ is separated, we have
$$\psId_{X}(\CM)\simeq \CM\sotimes k_X,$$
and when $X$ is separated and smooth of dimension $n$, we thus have:
\begin{equation} 
\psId_{X}\simeq \on{Id}_{\Dmod(X)}[-2n].
\end{equation}

\sssec{}  \label{sss:thm for sch}

We have the following theorem, which will be proved in \secref{s:delo}, more precisely, in \secref{sss:proof of schemes}:

\begin{thm} \label{t:schemes} Let $\CQ\in \Dmod(X_1\times X_2)$ be an object such that
the corresponding functor $\sF_{X_1\to X_2,\CQ}:\Dmod(X_1)\to \Dmod(X_2)$ admits a continuous right adjoint.
Then:

\medskip

\noindent{\em(a)} The object $\CQ$ is compact.

\medskip

\noindent{\em(b)} The functor 
$$\sF_{X_2\to X_1,\BD^{\on{Ve}}_{X_1\times X_2}(\CQ)}:\Dmod(X_2)\to \Dmod(X_1),$$
identifies canonically with 
$$\Dmod(X_2)\overset{(\sF_{X_1\to X_2,\CQ})^R}\longrightarrow \Dmod(X_1)\overset{\psId_{X_1}}\longrightarrow \Dmod(X_1).$$

\end{thm} 

Thus, informally, the functor $(\sF_{X_1\to X_2,\CQ})^R$ is ``almost" given by the kernel, Verdier dual to that of $\sF$, and the correction to the
``almost" is given by the functor $\psId_{X_1}$.

\sssec{}

We emphasize that by \corref{c:when right D}, the condition in the theorem that the functor $$\sF_{X_1\to X_2,\CQ}:\Dmod(X_1)\to \Dmod(X_2)$$ 
admit a continuous right adjoint is equivalent to the condition that it preserve compactness. 

\medskip

We note that point (a) of \thmref{t:schemes} is very simple:

\begin{proof}

The assertion follows from \corref{c:preserve compactness tensor}, using the fact that 
$$\CQ\simeq \left(\on{Id}_{\Dmod(X_1)}\otimes \sF_{X_1\to X_2,\CQ}\right)(\omega_{X_1}),$$
and $\omega_{X_1}\in \Dmod(X_1\times X_1)$ is compact.

\end{proof}

Let us also note the following:

\begin{prop} \label{p:open embedding}
Let $\CQ$ be as in \thmref{t:schemes}. Let $f:X_1\to X'_1$ (resp., $g:X'_1\to X_1$) be a smooth (resp., proper) morphism. 
Then the objects
$$\left(f\times \on{id}_{X_2}\right)_\bullet(\CQ) \text{ and } \left(g\times \on{id}_{X_2}\right)^!(\CQ)$$
of $\Dmod(X'_1\times X_2)$ also satisfy the assumption of \thmref{t:schemes}; in particular, they are compact.
\end{prop}

\begin{proof}
Apply \thmref{t:schemes}(a) to the functors $\Phi_{X_1\to X_2,\CQ}\circ f^!$ and $\Phi_{X_1\to X_2,\CQ}\circ g_\bullet$, respectively, 
and use the fact that the functors $f^!$ and $g_\bullet$ preserve compactness. 
\end{proof}

\sssec{}

Let us now swap the roles of $X_1$ and $X_2$: 

\begin{cor} \label{c:schemes left}
Let $\sF:\Dmod(X_1)\to \Dmod(X_2)$ be a continuous functor that admits a \emph{left} adjoint, and let $\CQ\in  \Dmod(X_1\times X_2)$ 
denote the kernel of $\sF^L$. Then:

\medskip

\noindent{\em(a)} The object $\CQ$ is compact.

\medskip

\noindent{\em(b)} The functor 
$$\sF_{X_1\to X_2,\BD^{\on{Ve}}_{X_1\times X_2}(\CQ)}:\Dmod(X_1)\to \Dmod(X_2),$$ 
identifies canonically with the composition 
$$\Dmod(X_1)\overset{\sF}\longrightarrow \Dmod(X_2)\overset{\psId_{X_2}}\longrightarrow \Dmod(X_2).$$

\end{cor}

\sssec{}   \label{sss:thm for smooth}

A particular case of \thmref{t:schemes} reads:

\begin{cor} \label{c:smooth schemes}  Let $\sF:\Dmod(X_1)\to \Dmod(X_2)$ be a continuous functor,
given by a kernel $\CQ\in \Dmod(X_1\times X_2)$. 

\smallskip

\noindent{\em(1)} Let $X_1$ be smooth of dimension $n_1$ and separated, and 
suppose that $\sF$ admits a continuous right adjoint. 
Then $\CQ$ is compact and the functor $\sF^R$ is given by the kernel $\BD^{\on{Ve}}_{X_1\times X_2}(\CQ)[2n_1]$.

\smallskip

\noindent{\em(2)}
Let $X_2$ be smooth of dimension $n_2$ and separated, and suppose that $\sF$ admits a left adjoint. Then $\CQ$ is compact and 
the functor $\sF^L$ is given by the kernel $\BD^{\on{Ve}}_{X_1\times X_2}(\CQ)[2n_2]$.

\end{cor}

As a particular case, we obtain:

\begin{cor}  \label{c:smooth schemes same dim}
Let $X_1$ and $X_2$ be both smooth and separated, of dimensions $n_1$ and $n_2$, respectively. 
Let $\sF:\Dmod(X_1)\to \Dmod(X_2)$ be a continuous functor, and 
assume that $\sF$ admits both left and continuous right adjoints. Then
$\sF^L[2(n_1-n_2)]\simeq \sF^R$.
\end{cor}

\sssec{}

Finally, we have the following, perhaps a little unexpected, result that will be proved in \secref{s:proof of subq}:

\begin{thm} \label{t:subquotient}
Assume that $X_1$ is quasi-projective and smooth. Let $\CQ\in \Dmod(X_1\times X_2)$ be such 
that the functor $\sF_{X_1\to X_2,\CQ}$ preserves compactness. Then any subquotient of any of the
cohomologies of $\CQ$ with respect to the standard t-structure on $\Dmod(X_1\times X_2)$, 
has the same property.
\end{thm}

\begin{rem}
We are nearly sure that in \thmref{t:subquotient}, the assumption that $X_1$ be quasi-projective
can be replaced by that of being separated.
\end{rem}

\ssec{Digression: dual functors}  

\sssec{} 

Let $\bC_1$ and $\bC_2$ be dualizable DG categories.  Recall that there is a canonical equivalence
\begin{equation} \label{e:dualization}
\on{Funct}_{\on{cont}}(\bC_1,\bC_2)\simeq \on{Funct}_{\on{cont}}(\bC^\vee_2,\bC^\vee_1),
\end{equation}
given by the passage to the dual functor, 
$$\sF\mapsto \sF^\vee.$$

In terms of the identification
$$\on{Funct}_{\on{cont}}(\bC_1,\bC_2)\simeq \bC_1^\vee \otimes \bC_2 \text{ and }
\on{Funct}_{\on{cont}}(\bC^\vee_2,\bC^\vee_1)\simeq (\bC_2^\vee)^\vee \otimes \bC^\vee_1,$$
the equivalence \eqref{e:dualization} corresponds to
$$\bC_1^\vee \otimes \bC_2\simeq \bC_2\otimes \bC_1^\vee\simeq (\bC_2^\vee)^\vee \otimes \bC^\vee_1.$$

\sssec{}  \label{sss:adjoint dual pair}

Note that if
$$\sF:\bC_1\rightleftarrows \bC_2:\sG$$
is an adjoint pair of functors, then the pair
$$\sG^\vee:\bC_2^\vee\rightleftarrows \bC_1:\sF^\vee$$
is also naturally adjoint.

\medskip

By duality, from \lemref{l:when right}, we obtain:

\begin{cor}  \label{c:left adj via check}
Let $\sF:\bC_1\to \bC_2$ be a continuous functor, and assume that $\bC_2$
is compactly generated. Then $\sF$ admits a \emph{left} adjoint if and only if the functor 
$\sF^\vee:\bC_2^\vee\to \bC_1^\vee$ preserves compactness.
\end{cor}

\sssec{}

Let us apply the above discussion to $\bC_i=\Dmod(X_i)$, $i=1,2$, where $X_1$ and $X_2$ are 
schemes of finite type.

\medskip

Thus, for 
$$\CQ\in \Dmod(X_1\times X_2)\simeq \Dmod(X_2\times X_1),$$ we have the following canonical isomorphism:
$$(\sF_{X_1\to X_2,\CQ})^\vee\simeq \sF_{X_2\to X_1,\CQ},$$
as functors $\Dmod(X_2)\to \Dmod(X_1)$, where we identify $\Dmod(X_i)^\vee\simeq \Dmod(X_i)$ by means
of $\bD^{\on{Ve}}_{X_i}$. 

\medskip 

In particular, for a scheme of finite type $X$, we have a canonical isomorphism
\begin{equation} \label{e:dual of pseud}
(\psId_{!,X})^\vee\simeq \psId_{!,X}.
\end{equation}

It comes from the equivariance structure on $(\Delta_X)_!(k_X)$ with respect to the flip automorphism of $X\times X$.

\medskip

Applying \corref{c:left adj via check} to D-modules, we obtain: 

\begin{cor} \label{c:when left D}
Let $\CQ$ be an object of $\Dmod(X_1\times X_2)$. Then the functor 
$$\sF_{X_1\to X_2,\CQ}:\Dmod(X_1)\to \Dmod(X_2)$$ admits a
left adjoint of and only if the functor 
$$\sF_{X_2\to X_1,\CQ}:\Dmod(X_2)\to \Dmod(X_1)$$ preserves
compactness.
\end{cor}

\ssec{Conjugate functors} \label{ss:conj functors}

\sssec{}  \label{sss:conj functors}

Let $\bC_1$ and $\bC_2$ be compactly generated categories, and let 
$\sF:\bC_1\to \bC_2$ be a functor that preserves compactness.

\medskip

Thus, we obtain a functor 
$$\sF^c:\bC_1^c\to \bC_2^c,$$
and consider the corresponding functor between the opposite categories
$$(\sF^c)^{\on{op}}: (\bC_1^c)^{\on{op}}\to (\bC_2^c)^{\on{op}}.$$

\medskip

Hence, ind-extending $(\sF^c)^{\on{op}}$ and using \eqref{e:dual as Ind}, we obtain a functor 
$$\bC_1^\vee\to \bC_2^\vee.$$
We shall denote it by $\sF^{\on{op}}$ and call it the \emph{conjugate} functor. 

\sssec{}

The following is \cite[Lemma 2.3.3]{DG}:

\begin{lem} \label{l:conjugate}
The functor $\sF^{\on{op}}$ is the left adjoint of $\sF^\vee$.
\end{lem}

Combining this with \secref{sss:adjoint dual pair}, we obtain: 

\begin{cor} \label{c:conjugate}
The functor $\sF^{\on{op}}$ is the dual of $\sF^R$.
\end{cor}

\sssec{Proof of \lemref{l:conjugate}}

Since all the functors in question are continuous and the categories are compactly generated, it suffices
to construct a functorial equivalence
\begin{equation} \label{e:adj isom to construct}
\CMaps_{\bC^\vee_2}(\sF^{\on{op}}(\bc^\vee_1),\bc^\vee_2)\simeq \CMaps_{\bC_1^\vee}(\bc^\vee_1,\sF^\vee(\bc^\vee_2)), \quad 
\bc_i\in \bC_i^c.
\end{equation} 

Recall (see \eqref{e:eval dual}) that for $\xi_i\in \bC_i^\vee$, 
$$\CMaps_{\bC^\vee_i}(\bc_i^\vee,\xi_i)\simeq \on{ev}_{\bC_i}(\bc_i\otimes \xi_i).$$

\medskip

Hence, the left-hand side in \eqref{e:adj isom to construct} can be rewritten as
$$\on{ev}_{\bC_2}(\sF(\bc_1)\otimes \bc^\vee_2),$$
while the right-hand side as
$$\on{ev}_{\bC_1}(\bc_1\otimes \sF^\vee(\bc^\vee_2)).$$

\medskip

Finally,
$$\on{ev}_{\bC_2}(\sF(\bc_1)\otimes \bc^\vee_2)\simeq \on{ev}_{\bC_1}(\bc_1\otimes \sF^\vee(\bc^\vee_2)),$$
by the definition of the dual functor.

\qed

Note that the same argument proves the following generalization of \lemref{l:conjugate}:

\begin{lem} \label{l:conjugate bis} Let $\bC_1$ and $\bC_2$ be two compactly generated 
categories, and let $\sG:\bC_2\to \bC_1$ be a continuous functor; let $\sF$ denote its partially defined 
left adjoint. Let $\bc_1\in \bC_1^c$ be an object such that $\sG^\vee(\bc_1^\vee)\in \bC_2^\vee$ is
compact. Then $\sF(\bc_1)$ is defined and canonically isomorphic to
$$\left(\sG^\vee(\bc_1^\vee)\right)^\vee.$$
\end{lem}

\ssec{Back to D-modules: conjugate functors}

\sssec{}   \label{sss:thm for sch conj}

By combining \thmref{t:schemes} with \corref{c:conjugate} and Equation \eqref{e:dual of pseud}, we obtain:

\begin{cor} \label{c:right conj}
Under the assumptions and in the notations of of \thmref{t:schemes}, the functor
$$\sF_{X_1\to X_2,\BD^{\on{Ve}}_{X_1\times X_2}(\CQ)}:\Dmod(X_1)\to \Dmod(X_2),$$ 
is canonically isomorphic to the composition
$$\Dmod(X_1)\overset{\psId_{X_1}}\longrightarrow \Dmod(X_1) \overset{(\sF_{X_1\to X_2,\CQ})^{\on{op}}}\longrightarrow  \Dmod(X_2).$$
\end{cor}

\medskip

Note that in the circumstances of \corref{c:right conj}, the functor
$$(\sF_{X_1\to X_2,\CQ})^{\on{op}}:\Dmod(X_1)\to \Dmod(X_2)$$
also preserves compactness, by construction. 

\sssec{}

We emphasize that the functor $(\sF_{X_1\to X_2,\CQ})^{\on{op}}$ that appears in \corref{c:right conj smooth} is by 
definition the ind-extension
of the functor, defined on $\Dmod(X_1)^c\subset \Dmod(X_1)$ and given by
$$\CM\mapsto \BD^{\on{Ve}}_{X_2}\circ \sF_{X_1\to X_2,\CQ} \circ  \BD^{\on{Ve}}_{X_1}(\CM),$$
where the right-hand side is defined, because 
$\sF_{X_1\to X_2,\CQ}\circ  \BD^{\on{Ve}}_{X_1}(\CM)\in \Dmod(X_2)^c$. 

\medskip

In other words, $(\sF_{X_1\to X_2,\CQ})^{\on{op}}|_{\Dmod(X_1)^c}$
is obtained from $(\sF_{X_1\to X_2,\CQ})|_{\Dmod(X_1)^c}$ by conjugating by Verdier duality (hence the name ``conjugate"). 

\medskip

See \secref{sss:other functor} for an even more explicit description of $(\sF_{X_1\to X_2,\CQ})^{\on{op}}$.

\sssec{}   \label{sss:thm for smooth conj}

By imposing the smoothness and separatedness hypothesis, from \corref{c:right conj} we obtain: 

\begin{cor} \label{c:right conj smooth}
Let $X_1$ be smooth of dimension $n_1$, and separated. 
Let $$\sF:\Dmod(X_1)\to \Dmod(X_2)$$ be a continuous functor,
given by a kernel $\CQ\in \Dmod(X_1\times X_2)$. Assume that $\sF$ preserves compactness. 
Then the conjugate functor $$\sF^{\on{op}}:\Dmod(X_1)\to \Dmod(X_2)$$
is given by the kernel $\BD^{\on{Ve}}_{X_1\times X_2}(\CQ)[2n_1]$.
\end{cor} 

Further, from \corref{c:right conj smooth}, we deduce:

\begin{cor} \label{c:stable under dual}
Let $\CQ$ be a compact object of $\Dmod(X_1\times X_2)$.

\smallskip

\noindent{\em(1)} Let $X_1$ be smooth and separated, and assume that the functor 
$\Dmod(X_1)\to \Dmod(X_2)$, defined by $\CQ$, admits a \emph{continuous right} adjoint
(i.e., preserves compactness). Then
so does the functor $\Dmod(X_1)\to \Dmod(X_2)$, defined by $\BD^{\on{Ve}}_{X_1\times X_2}(\CQ)$.

\smallskip

\noindent{\em(2)} Let $X_2$ be smooth and separated, and assume that the functor 
$\Dmod(X_1)\to \Dmod(X_2)$, defined by $\CQ$, admits a \emph{left} adjoint. Then
so does the functor $\Dmod(X_1)\to \Dmod(X_2)$, defined by $\BD^{\on{Ve}}_{X_1\times X_2}(\CQ)$.

\end{cor}

\ssec{Another interpretation of conjugate functors} \label{ss:other functors}

\sssec{}

Consider the functors
$$\on{pr}_1^\bullet:\Dmod(X_1)\to \Dmod(X_1\times X_2);$$
$$\overset{\bullet}\otimes=\Delta_{X_1\times X_2}^\bullet: \Dmod((X_1\times X_2)\times (X_1\times X_2))\to 
\on{Pro}(\Dmod(X_1\times X_2)),$$
and
$$(\on{pr}_2)_!:\on{Pro}(\Dmod(X_1\times X_2))\to \on{Pro}(\Dmod(X_2)),$$
see \secref{sss:bad functors}. 

\medskip

For an object $\CP\in \Dmod(X_1\times X_2)$ consider the functor
$$\sF^{\on{op}}_{X_1\to X_2;\CP}:\Dmod(X_1)\to \on{Pro}(\Dmod(X_2)),$$
defined by 
\begin{equation} \label{e:other functor}
\sF^{\on{op}}_{X_1\to X_2;\CP}(\CM):=
(\on{pr}_2)_!\left(\on{pr}_1^\bullet(\CM)\overset{\bullet}\otimes\CP\right).
\end{equation}

\medskip

The assignment
$$\CP\in \Dmod(X_1\times X_2)\rightsquigarrow \sF^{\on{op}}_{X_1\to X_2;\CP}:\Dmod(X_1)\to \on{Pro}(\Dmod(X_2)),$$ 
is another way to construct a functor from an object on the product, using the Verdier conjugate functors, i.e., by replacing 
$$p_1^!\mapsto p_1^\bullet;\,\, \sotimes\mapsto  \overset{\bullet}\otimes,\,\, (\on{pr}_2)_\bullet\mapsto (\on{pr}_2)_!.$$

\begin{rem}
Let $\CM\in \Dmod(X_1)$ be such that the functors $\Delta_{X_1\times X_2}^\bullet$ and $(\on{pr}_2)_!$ are
defined on the objects $\on{pr}_1^\bullet(\CM)\boxtimes \CP$ and $\on{pr}_1^\bullet(\CM)\overset{\bullet}\otimes\CP$, respectively.
(E.g., this is the case when $\CP$ and $\CM$ are both holonomic.) Then the notation 
$$(\on{pr}_2)_!\left(\on{pr}_1^\bullet(\CM)\overset{\bullet}\otimes\CP\right)\in \Dmod(X_2)\subset \on{Pro}(\Dmod(X_2))$$ 
is unambiguous. 
\end{rem}

\sssec{}  \label{sss:other as left}

Assume that $\CP\in \Dmod(X_1\times X_2)^c$. Denote $\CQ:=\BD^{\on{Ve}}_{X_1\times X_2}(\CP)$. 
Then it is easy to see that the functor 
$$\sF^{\on{op}}_{X_1\to X_2;\CP}:\Dmod(X_1)\to \on{Pro}(\Dmod(X_2))$$ is the left adjoint of the functor
$$\sF_{X_2\to X_1,\CQ}:\Dmod(X_2)\to \Dmod(X_1),$$ 
in the sense that for $\CM_i\in \Dmod(X_i)$ we have a canonical isomorphism
$$\CMaps_{\on{Pro}(\Dmod(X_2))}(\sF^{\on{op}}_{X_1\to X_2;\CP}(\CM_1),\CM_2)\simeq
\CMaps_{\Dmod(X_1)}(\CM_1,\sF_{X_2\to X_1,\CQ}(\CM_2)),$$
where the left-hand side can be also interpreted as the evaluation of an object of the pro-completion
of a DG category on an object of that DG category, see \secref{sss:bad functors}.

\sssec{}  \label{sss:other functor}

Take now $\CP:=\BD^{\on{Ve}}_{X_1\times X_2}(\CQ)$, where $\CQ$ is as in \thmref{t:schemes}
(i.e., the functor $\sF_{X_1\to X_2,\CQ}$ preserves compactness).  

\medskip

By \corref{c:when left D}, the functor $\sF_{X_2\to X_1,\CQ}\simeq \sF_{X_1\to X_2,\CQ}^\vee$ admits a left adjoint.
Hence, by \secref{sss:other as left}, the functor $\sF^{\on{op}}_{X_1\to X_2;\BD^{\on{Ve}}_{X_1\times X_2}(\CQ)}$ takes values in 
$$\Dmod(X_2)\subset \on{Pro}(\Dmod(X_2)),$$
and provides a left adjoint to $\sF_{X_2\to X_1,\CQ}$. By \lemref{l:conjugate}, we obtain an isomorphism of functors 
$\Dmod(X_1)\to \Dmod(X_2)$:
$$\sF^{\on{op}}_{X_1\to X_2;\BD^{\on{Ve}}_{X_1\times X_2}}(\CQ)\simeq
(\sF_{X_1\to X_2,\CQ})^{\on{op}}.$$

\medskip

Thus, we can interpret \corref{c:right conj} as follows:

\begin{cor} \label{c:conj other}
For $\CQ$ as in \thmref{t:schemes} we have a canonical isomorphism
$$\sF^{\on{op}}_{X_1\to X_2;\BD^{\on{Ve}}_{X_1\times X_2}(\CQ)}\circ \psId_{X_1}\simeq 
\sF_{X_1\to X_2,\BD^{\on{Ve}}_{X_1\times X_2}(\CQ)},$$
where
$$\sF^{\on{op}}_{X_1\to X_2;\BD^{\on{Ve}}_{X_1\times X_2}(\CQ)}(\CM)=
(\on{pr}_2)_!\left(\on{pr}_1^\bullet(\CM)\overset{\bullet}\otimes(\BD^{\on{Ve}}_{X_1\times X_2}(\CQ))\right).$$
\end{cor}

\sssec{}

Combining \corref{c:conj other} with \corref{c:stable under dual}(1) we obtain:

\begin{cor} \label{c:our functor as other}
Let $X_1$ be smooth of dimension $n_1$ and separated. Let $\CQ\in \Dmod(X_1\times X_2)$
satisfy the assumption of \thmref{t:schemes}. Then there is a canonical isomorphism
$$\sF^{\on{op}}_{X_1\to X_2,\CQ}\simeq \sF_{X_1\to X_2,\CQ}[2n_1],$$
i.e.,
$$(\on{pr}_2)_!\left(\on{pr}_1^\bullet(\CM)\overset{\bullet}\otimes \CQ\right)\simeq
(\on{pr}_2)_\bullet\left(\on{pr}_1^!(\CM)\sotimes \CQ\right)[2n_1],\quad \CM\in \Dmod(X_1);$$
in particular, the left-hand side takes values in $\Dmod(X_2)\subset \on{Pro}(\Dmod(X_2))$. 
\end{cor}

\sssec{}

We shall now deduce a property of the functors $\sF_{X_1\to X_2,\CQ}$ 
satisfying the assumption of \thmref{t:schemes} with respect to the standard
t-structure on the category of D-modules. 

\medskip

In what follows, for a DG category $\bC$, endowed with a t-structure, we let
$\bC^{\leq 0}$ (resp., $\bC^{\geq 0}$) denote the corresponding
subcategory of connective (resp., coconnective) objects. We let 
$\bC^\heartsuit:=\bC^{\leq 0}\cap \bC^{\geq 0}$ denote the heart of the t-structure. 

\begin{cor}  \label{c:exactness}
Let $X_1$ and $\CQ$ be as in \corref{c:our functor as other}. Assume in addition that
the support of $\CQ$ is affine over $X_2$, and that $\CQ\in \Dmod(X_1\times X_2)^\heartsuit$. 
Then the functor $$\sF_{X_1\to X_2,\CQ}[n_1]:\Dmod(X_1)\to \Dmod(X_2)$$
is t-exact.
\end{cor} 

\begin{proof}

The fact that $\sF_{X_1\to X_2,\CQ}[n_1]$ is right t-exact is straightforward from the definition
(no assumption that $\sF_{X_1\to X_2,\CQ}$ preserve compactness is needed). 

\medskip

The fact that $\sF_{X_1\to X_2,\CQ}[n_1]$ is left t-exact follows from the isomorphism
$$\sF_{X_1\to X_2,\CQ}[n_1]\simeq  \sF^{\on{op}}_{X_1\to X_2,\CQ}[-n_1].$$

\end{proof}

\ssec{An example: Fourier-Deligne transform} \label{ss:FD}

Let us consider a familiar example of the situation described in Corollaries \ref{c:our functor as other},
\ref{c:smooth schemes} and \ref{c:smooth schemes same dim}. 

\sssec{}

Namely, let $V$ be a 
finite-dimensional vector space, thought of a scheme over $k$, and let $V^\vee$ be the dual vector space.
We take $X_1=V$ and $X_2=V^\vee$.

\medskip

We take the kernel $\CQ\in \Dmod(V\times V^\vee)$ to be the pullback of exponential D-module on $\BG_a$
under the evaluation map $V\times V^\vee\to \BG_a$. We denote it symbolically by
$$\on{exp}\in  \Dmod(V\times V^\vee),$$
and we normalize it so that it lives in cohomological degree $-\dim(V)$ with respect to the natural t-structure 
on the category $\Dmod(V\times V^\vee)$.

\medskip

The corresponding functor $\Dmod(V)\to \Dmod(V^\vee)$ is by definition the Fourier-Deligne transform 
\begin{equation} \label{e:FD}
\sF_{V\to V^\vee,\on{exp}}=(\on{pr}_2)_\bullet\left(\on{pr}_1^!(\CM)\sotimes \on{exp}\right).
\end{equation}

Since $\sF_{V\to V^\vee,\on{exp}}$ is an equivalence, it admits both left and right adjoints (which are isomorphic). 

\sssec{}  \label{sss:FD}
 
It is well-known that the functor $\sF_{V\to V^\vee,\on{exp}}$ can be rewritten as 
\begin{equation} \label{e:FD !}
\CM\mapsto (\on{pr}_2)_!\left(\on{pr}_1^!(\CM)\sotimes \on{exp}\right),
\end{equation}
(see \secref{sss:bad functors} regarding the meaning of $(\on{pr}_2)_!$).  

\medskip

Now, using the fact that the map $\on{pr}_1$ is smooth and that the D-module $\on{exp}$
on $V\times V^\vee$ is lisse, the expression in \eqref{e:FD !}
can be further rewritten as 
$$(\on{pr}_2)_!\left(\on{pr}_1^\bullet(\CM)\overset{\bullet}\otimes\on{exp}\right)[-2\dim(V)],$$
and the latter functor identifies with the functor 
$$\sF^{\on{op}}_{V\to V^\vee,\on{exp}}[-2\dim(V)].$$

\medskip

Thus, we obtain an isomorphism
$$\sF^{\on{op}}_{V\to V^\vee,\on{exp}}[-2\dim(V)]\simeq \sF_{V\to V^\vee,\on{exp}}.$$

However, we now know that the latter is not a special feature of the Fourier-Deligne trasform,
but rather a particular case of \corref{c:our functor as other} (for $X_1$ smooth and separated). 

\medskip

Note also that the fact that the map 
from \eqref{e:FD !} $\to$ \eqref{e:FD}, coming from the natural transformation 
$(\on{pr}_2)_!\to (\on{pr}_2)_\bullet$ is an isomorphism, follows from the description of the 
isomorphism of \corref{c:right conj} in the separated case; this description will be given in the next section,
specifically, \secref{ss:nat trans sep}.

\sssec{}

The right adjoint of $\sF_{V\to V^\vee,\on{exp}}$, written as $\sF^{\on{op}}_{V\to V^\vee,\on{exp}}[-2\dim(V)]$, identifies with 
\begin{equation} \label{e:FD ! rewrite again}
\CM'\mapsto (\on{pr}_1)_\bullet\left(\on{pr}_2^!(\CM')\sotimes \BD^{\on{Ve}}_{V\times V^\vee}(\on{exp})\right)[2\dim(V)],
\end{equation}
which in turn is the functor $\sF_{V^\vee\to V;\on{-exp}}$, i.e.,  the inverse Fourier-Deligne transform. 

\medskip

The isomorphism
$$(\sF_{V\to V^\vee,\on{exp}})^R\simeq \sF_{V^\vee\to V;\on{-exp}}$$
coincides with the assertion of \corref{c:smooth schemes}(1).

\sssec{}

Finally, we note that the functor $\sF_{V\to V^\vee,\on{exp}}$ admits a \emph{left} adjoint, 
given by
$$\CM'\mapsto (\on{pr}_1)_!\left(\on{pr}_2^\bullet(\CM')\overset{\bullet}\otimes\BD^{\on{Ve}}_{V\times V^\vee}(\on{exp})\right),$$
i.e., $\sF^{\on{op}}_{V^\vee\to V,\on{-exp}}[-2\dim(V)]$, 
which, by \secref{sss:FD} with the roles of $V$ and $V^\vee$ swapped, is well-defined and isomorphic to $\sF_{V^\vee\to V;\on{-exp}}$.

\medskip

The isomorphism
$$(\sF_{V\to V^\vee,\on{exp}})^R\simeq (\sF_{V\to V^\vee,\on{exp}})^L$$
coincides with the assertion of \corref{c:smooth schemes same dim}.

\section{The natural transformations}  \label{s:nat trans}

The goal of this section is to describe geometrically the isomorphisms of \thmref{t:schemes}
and \corref{c:conj other}. This material will not be used elsewhere in the paper. 

\ssec{The adjunction map}  \label{ss:adj map}

\sssec{}

Let $\CQ$ be as in \thmref{t:schemes}. The (iso)morphism
\begin{equation} \label{e:original map}
\psId_{X_1}\circ (\sF_{X_1\to X_2,\CQ})^R\to \sF_{X_2\to X_1,\BD^{\on{Ve}}_{X_1\times X_2}(\BQ)}
\end{equation}
of \thmref{t:schemes}
gives rise to a natural transformation
\begin{equation} \label{e:adj map}
\psId_{X_1}\to \sF_{X_2\to X_1,\BD^{\on{Ve}}_{X_1\times X_2}(\CQ)}\circ \sF_{X_1\to X_2,\CQ}.
\end{equation}

The map \eqref{e:adj map} will be described exlicitly (in the context of general DG categories)
in \secref{sss:adj map abstract}. We will now explain what this abstract description amounts to
in the case of categories of D-modules.

\sssec{}

First, we note that for a scheme $Y$ and $\CM\in \Dmod(Y)^c$ we have a canonical map
$$\CM\boxtimes \BD^{\on{Ve}}_Y(\CM)\to (\Delta_Y)_\bullet(\omega_Y).$$

Applying Verdier duality, we obtain a canonical map
\begin{equation} \label{e:map from k}
(\Delta_Y)_!(k_Y)\to \CM\boxtimes \BD^{\on{Ve}}_Y(\CM).
\end{equation}

\sssec{}   \label{sss:adj map}

The right-hand side in \eqref{e:adj map} is a functor $\Dmod(X_1)\to \Dmod(X_1)$ given by the kernel
\begin{equation} \label{e:kernel of comp}
(\on{id}_{X_1}\times p_{X_2}\times \on{id}_{X_1})_\bullet \circ (\on{id}_{X_1}\times \Delta_{X_2}\times \on{id}_{X_1})^!\circ 
\sigma_{2,3}(\CQ\boxtimes \BD^{\on{Ve}}_{X_1\times X_2}(\CQ)),
\end{equation}
where $\sigma_{2,3}$ is the transposition of the corresponding factors.

\medskip

The datum of a map in \eqref{e:adj map} is equivalent to that of a map from $(\Delta_{X_1})_!(k_{X_1})$ to
\eqref{e:kernel of comp}, and further, by the $((\Delta_{X_1})_!,\Delta_{X_1}^!)$-adjunction, to a map
\begin{equation} \label{e:kernel of comp 1}
k_{X_1}\to \Delta_{X_1}^!\circ 
(\on{id}_{X_1}\times p_{X_2}\times \on{id}_{X_1})_\bullet \circ (\on{id}_{X_1}\times \Delta_{X_2}\times \on{id}_{X_1})^!\circ 
\sigma_{2,3}(\CQ\boxtimes \BD^{\on{Ve}}_{X_1\times X_2}(\CQ)).
\end{equation}

By base change along
$$
\CD
X_1\times X_2  @>{\Delta_{X_1}\times \on{id}_{X_2}}>>   X_1\times X_1\times  X_2  \\ 
@V{\on{id}_{X_1}\times p_{X_2}}VV    @VV{\on{id}_{X_1}\times \on{id}_{X_1}\times p_{X_2}}V   \\
X_1  @>{\Delta_{X_1}}>>   X_1\times X_1,
\endCD
$$
the right-hand side in \eqref{e:kernel of comp 1} identifies with
$$(\on{id}_{X_1}\times p_{X_2})_\bullet(\CQ\sotimes \BD^{\on{Ve}}_{X_1\times X_2}(\CQ)).$$

\sssec{}

Now, the desired map in \eqref{e:kernel of comp 1} comes from
$$k_{X_1}\to (\on{id}_{X_1}\times p_{X_2})_\bullet(k_{X_1\times X_2})\overset{\text{\eqref{e:map from k}}}\longrightarrow
(\on{id}_{X_1}\times p_{X_2})_\bullet(\CQ\sotimes \BD^{\on{Ve}}_{X_1\times X_2}(\CQ)).$$

In the above formula, the first arrow uses the canonical map (defined for any scheme $Y$; in our case $Y=X_2$)
$$k\to (p_Y)_\bullet(k_Y),$$
that arises from the $(p_Y^\bullet,(p_Y)_\bullet)$-adjunction. 

\ssec{The map between two styles of functors}   \label{ss:nat trans}

Let $\CQ$ be again as in \thmref{t:schemes}. We shall now write down explicitly the (iso)morphism 
\begin{equation} \label{e:left to right prel}
(\sF_{X_1\to X_2;\CQ})^{\on{op}}\circ \psId_{X_1}\to \sF_{X_1\to X_2,\BD^{\on{Ve}}_{X_1\times X_2}(\CQ)}
\end{equation} 
of \corref{c:right conj}. 

\sssec{}

By \secref{sss:other functor}, we rewrite $(\sF_{X_1\to X_2;\CQ})^{\on{op}}$ as 
$$\sF^{\on{op}}_{X_1\to X_2;\BD^{\on{Ve}}_{X_1\times X_2}(\CQ)},$$
where $\sF^{\on{op}}_{X_1\to X_2;\BD^{\on{Ve}}_{X_1\times X_2}(\CQ)}$ is as in \eqref{e:other functor}. 

\medskip

Thus, we need to describe the resulting natural transformation
\begin{equation} \label{e:left to right Q}
\sF^{\on{op}}_{X_1\to X_2;\BD^{\on{Ve}}_{X_1\times X_2}(\CQ)}\circ \psId_{X_1}\to \sF_{X_1\to X_2,\BD^{\on{Ve}}_{X_1\times X_2}(\CQ)}.
\end{equation}

More generally, we will write down a natural transformation
\begin{equation} \label{e:left to right}
\sF^{\on{op}}_{X_1\to X_2;\CP}\circ \psId_{X_1}\to \sF_{X_1\to X_2,\CP}
\end{equation}
for any $\CP\in \Dmod(X_1\times X_2)$  (i.e., not necessarily the dual of an object defining a functor
satisfying the assumption of \thmref{t:schemes}).

\medskip

The description of the map \eqref{e:left to right} occupies the rest of this subsection.  The fact that \eqref{e:left to right},
when applied to $\CP:=\BD^{\on{Ve}}_{X_1\times X_2}(\CQ)$, yields \eqref{e:left to right Q} is verified by a diagram chase,
once we interpret \eqref{e:left to right Q} as obtained by passage to the dual functors in \eqref{e:original map},
described explicitly in \secref{sss:adj map}.

\medskip

We note that when the scheme $X_1$ is separated, the map \eqref{e:left to right} can be significantly simplified,
see \secref{ss:nat trans sep}.

\sssec{}

Consider the following diagram of schemes
$$
\CD
& & & & X_1\times X_1\times X_1\times X_2 \\
& & & & @V{\on{id}_{X_1}\times \Delta_{X_1}\times \on{id}_{X_1}\times \on{id}_{X_2}}VV  \\
& & X_1\times X_1\times X_1\times X_2 
@>{\Delta_{X_1}\times\on{id}_{X_1}\times \on{id}_{X_1}\times \on{id}_{X_2}}>> X_1\times X_1\times X_1\times X_1\times X_2 \\
& & @VV{p_{X_1}\times \on{id}_{X_1}\times \on{id}_{X_1}\times \on{id}_{X_2}}V   \\
X_1\times X_2 @>{\Delta_{X_1}\times\on{id}_{X_2}}>> X_1\times X_1\times X_2 \\
@VV{p_{X_1}\times \on{id}_{X_2}}V  \\
X_2
\endCD
$$

For $\CM\in \Dmod(X_1)$ we start with the object
\begin{equation} \label{e:starting object}
\CM\boxtimes k_{X_1}\boxtimes \CP\in \Dmod(X_1\times X_1\times X_1\times X_2).
\end{equation}
The object
$$\sF^{\on{op}}_{X_1\to X_2;\CP}\circ \psId_{X_1}(\CM)\in \on{Pro}(\Dmod(X_2)),$$ 
i.e., the left-hand side of \eqref{e:left to right}, applied to $\CM$, 
equals the result the application to \eqref{e:starting object}
of the following composition of functors 
\begin{multline}  \label{e:stage 1}
(p_{X_1}\times \on{id}_{X_2})_!\circ (\Delta_{X_1}\times\on{id}_{X_2})^\bullet\circ 
(p_{X_1}\times \on{id}_{X_1}\times \on{id}_{X_1}\times \on{id}_{X_2})_\bullet\circ \\
\circ (\Delta_{X_1}\times\on{id}_{X_1}\times \on{id}_{X_1}\times \on{id}_{X_2})^!\circ 
(\on{id}_{X_1}\times \Delta_{X_1}\times \on{id}_{X_1}\times \on{id}_{X_2})_!.
\end{multline}

\sssec{}

Note that for a Cartesian diagram
\begin{equation} \label{e:Cart diagram}
\CD
Y_{11} @>{f_1}>>  Y_{1,0} \\
@V{g_1}VV   @VV{g_0}V  \\
Y_{0,1}  @>{f_0}>> Y_{0,0}
\endCD
\end{equation}
we have a canonically defined natural transformation
$$f_0^\bullet\circ (g_0)_\bullet\to (g_1)_\bullet\circ f_1^\bullet,$$
coming by adjunction from the isomorphism
$$(g_0)_\bullet\circ (f_1)_\bullet\simeq (f_0)_\bullet\circ (g_1)_\bullet.$$

Applying this to the square
$$
\CD
X_1\times X_1\times X_2 @>{\on{id}_{X_1}\times \Delta_{X_1}\times  \on{id}_{X_2}}>>  X_1\times X_1\times X_1\times X_2  \\
@V{p_{X_1}\times  \on{id}_{X_1}\times \on{id}_{X_2}}VV   
@VV{p_{X_1}\times \on{id}_{X_1}\times \on{id}_{X_1}\times \on{id}_{X_2}}V  \\
X_1\times X_2 @>{\Delta_{X_1}\times\on{id}_{X_2}}>> X_1\times X_1\times X_2, 
\endCD
$$

\medskip

we obtain a natural transformation from \eqref{e:stage 1} to
\begin{multline}  \label{e:stage 2}
(p_{X_1}\times \on{id}_{X_2})_!\circ (p_{X_1}\times \on{id}_{X_1}\times \on{id}_{X_2})_\bullet\circ 
(\on{id}_{X_1}\times  \Delta_{X_1} \times \on{id}_{X_2})^\bullet\circ \\
\circ (\Delta_{X_1}\times\on{id}_{X_1}\times \on{id}_{X_1}\times \on{id}_{X_2})^!\circ 
(\on{id}_{X_1}\times \Delta_{X_1}\times \on{id}_{X_1}\times \on{id}_{X_2})_!.
\end{multline}

\medskip

I.e., we are now looking at the diagram

$$
\xy
(5,0)*+{X_1\times X_1\times X_2}="X";
(5,-25)*+{X_1\times X_2}="Y";
(5,-45)*+{X_2.}="Z";
(25,-15)*+{X_1\times X_1\times X_1\times X_2}="W";
(100,-15)*+{\,\,X_1\times X_1\times X_1\times X_1\times X_2}="U";
(100,10)*+{X_1\times X_1\times X_1\times X_2}="V";
{\ar@{->}_{p_{X_1}\times \on{id}_{X_1}\times \on{id}_{X_2}} "X";"Y"}
{\ar@{->}_{p_{X_1}\times \on{id}_{X_2}} "Y";"Z"}
{\ar@{->}^{\on{id}_{X_1}\times  \Delta_{X_1} \times \on{id}_{X_2}} "X";"W"}
{\ar@{->}^{\Delta_{X_1}\times  \on{id}_{X_1}\times  \on{id}_{X_1}\times \on{id}_{X_2}} "W";"U"}
{\ar@{->}_{\on{id}_{X_1}\times \Delta_{X_1}\times \on{id}_{X_1}\times \on{id}_{X_2}} "V";"U"}
\endxy
$$

\sssec{}

Note also that for a Cartesian diagram \eqref{e:Cart diagram} there is a canonical natural transformation
\begin{equation} \label{e:nat trans pullback}
g_1^\bullet\circ f_0^!\to f_1^!\circ g_0^\bullet,
\end{equation}
coming by adjunction from the base change isomorphism
$$f_0^!\circ (g_0)_\bullet\simeq (g_1)_\bullet\circ f_1^!.$$  

\medskip

Applying this to the square
$$
\CD
X_1\times X_1\times X_2   @>{\Delta_{X_1}\times\on{id}_{X_1}\times \on{id}_{X_2}}>>   X_1\times X_1\times X_1\times X_2   \\ 
@V{\on{id}_{X_1}\times  \Delta_{X_1} \times \on{id}_{X_2}}VV    @VV{\on{id}_{X_1}\times \on{id}_{X_1}\times \Delta_{X_1}\times \on{id}_{X_2}}V      \\
X_1\times X_1\times X_1\times X_2 
@>{\Delta_{X_1}\times\on{id}_{X_1}\times \on{id}_{X_1}\times \on{id}_{X_2}}>> X_1\times X_1\times X_1\times X_1\times X_2,
\endCD
$$

\medskip

we obtain a natural transformation from \eqref{e:stage 2} to 
\begin{multline}  \label{e:stage 3}
(p_{X_1}\times \on{id}_{X_2})_!\circ (p_{X_1}\times \on{id}_{X_1}\times \on{id}_{X_2})_\bullet\circ 
(\Delta_{X_1}\times\on{id}_{X_1}\times \on{id}_{X_2})^!\circ \\
\circ  (\on{id}_{X_1}\times \on{id}_{X_1}\times \Delta_{X_1}\times
\on{id}_{X_2})^\bullet\circ (\on{id}_{X_1}\times \Delta_{X_1}\times \on{id}_{X_1}\times \on{id}_{X_2})_!.
\end{multline}

\medskip

I.e., we are now looking at the diagram

\smallskip

$$
\xy
(5,0)*+{X_1\times X_1\times X_2}="X";
(5,-20)*+{X_1\times X_2}="Y";
(5,-40)*+{X_2.}="Z";
(65,0)*+{X_1\times X_1\times X_1\times X_2}="W";
(100,-30)*+{\,\,X_1\times X_1\times X_1\times X_1\times X_2}="U";
(100,10)*+{X_1\times X_1\times X_1\times X_2}="V";
{\ar@{->}_{p_{X_1}\times \on{id}_{X_1}\times \on{id}_{X_2}} "X";"Y"}
{\ar@{->}_{p_{X_1}\times \on{id}_{X_2}} "Y";"Z"}
{\ar@{->}^{\Delta_{X_1} \times \on{id}_{X_1}\times \on{id}_{X_2}} "X";"W"}
{\ar@{->}_{\on{id}_{X_1}\times \on{id}_{X_1}\times   \Delta_{X_1}\times \on{id}_{X_2}} "W";"U"}
{\ar@{->}^{\on{id}_{X_1}\times \Delta_{X_1}\times \on{id}_{X_1}\times \on{id}_{X_2}} "V";"U"}
\endxy
$$

\sssec{}

By base change along
$$
\CD
X_1\times X_1\times X_2 @>{\on{id}_{X_1}\times \Delta_{X_1}\times \on{id}_{X_2}}>>  X_1\times X_1\times X_1\times X_2   \\  
@V{\on{id}_{X_1}\times \Delta_{X_1}\times \on{id}_{X_2}}VV   
@VV{\on{id}_{X_1}\times \Delta_{X_1}\times \on{id}_{X_1}\times \on{id}_{X_2}}V  \\
X_1\times X_1\times X_1\times X_2   
@>>{\on{id}_{X_1}\times \on{id}_{X_1}\times \Delta_{X_1}\times \on{id}_{X_2}}>  X_1\times X_1\times X_1\times X_1\times X_2,
\endCD
$$

\medskip

we rewrite \eqref{e:stage 3} as
\begin{multline}  \label{e:stage 4}
(p_{X_1}\times \on{id}_{X_2})_!\circ (p_{X_1}\times \on{id}_{X_1}\times \on{id}_{X_2})_\bullet\circ \\
\circ (\Delta_{X_1}\times\on{id}_{X_1}\times \on{id}_{X_2})^!\circ (\on{id}_{X_1}\times \Delta_{X_1}\times \on{id}_{X_2})_!\circ 
(\on{id}_{X_1}\times \Delta_{X_1}\times \on{id}_{X_2})^\bullet.
\end{multline}

I.e., our diagram is now

$$
\xy
(5,0)*+{X_1\times X_1\times X_2}="X";
(5,-20)*+{X_1\times X_2}="Y";
(5,-40)*+{X_2.}="Z";
(65,0)*+{X_1\times X_1\times X_1\times X_2}="W";
(65,30)*+{X_1\times X_1\times X_2}="U";
(100,15)*+{X_1\times X_1\times X_1\times X_2}="V";
{\ar@{->}_{p_{X_1}\times \on{id}_{X_1}\times \on{id}_{X_2}} "X";"Y"}
{\ar@{->}_{p_{X_1}\times \on{id}_{X_2}} "Y";"Z"}
{\ar@{->}^{\Delta_{X_1} \times \on{id}_{X_1}\times \on{id}_{X_2}} "X";"W"}
{\ar@{<-}^{\on{id}_{X_1}\times \Delta_{X_1}\times \on{id}_{X_2}} "W";"U"}
{\ar@{<-}_{\,\,\,\,\,\,\,\on{id}_{X_1}\times \Delta_{X_1}\times \on{id}_{X_2}} "V";"U"}
\endxy
$$

\sssec{}

Note now that
$$(\on{id}_{X_1}\times \Delta_{X_1}\times \on{id}_{X_2})^\bullet(\CM\boxtimes k_{X_1}\boxtimes \CP)\simeq \CM\boxtimes \CP.$$

Hence, we are considering the diagram
$$
\xy
(0,0)*+{X_1\times X_1\times X_2}="X";
(0,-15)*+{X_1\times X_2}="Y";
(20,-25)*+{X_2}="Z";
(60,0)*+{\,\,X_1\times X_1\times X_1\times X_2}="W";
(60,15)*+{X_1\times X_1\times X_2}="V"
{\ar@{->}_{p_{X_1}\times \on{id}_{X_1}\times \on{id}_{X_2}} "X";"Y"}
{\ar@{->}^{\,\,\,p_{X_1}\times \on{id}_{X_2}} "Y";"Z"}
{\ar@{->}^{\Delta_{X_1} \times \on{id}_{X_1}\times  \on{id}_{X_2}} "X";"W"}
{\ar@{->}^{\on{id}_{X_1}\times \Delta_{X_1}\times \on{id}_{X_2}} "V";"W"}
\endxy
$$

and we need to calculate the functor
\begin{equation}  \label{e:stage 4 bis}
(p_{X_1}\times \on{id}_{X_2})_!\circ (p_{X_1}\times \on{id}_{X_1}\times \on{id}_{X_2})_\bullet
\circ (\Delta_{X_1}\times\on{id}_{X_1}\times \on{id}_{X_2})^!\circ (\on{id}_{X_1}\times \Delta_{X_1}\times \on{id}_{X_2})_!
\end{equation} 

applied to $\CM\boxtimes \CP\in \Dmod(X_1\times X_1\times X_2)$. 

\sssec{}

Consider again the Cartesian diagarm \eqref{e:Cart diagram}. Note that we have a canonical natural transformation
\begin{equation} \label{e:nat trans dir im}
(f_0)_!\circ (g_1)_\bullet\to (g_0)_\bullet\circ (f_1)_!
\end{equation}
that comes by adjunction from the base change isomorphism
$$(g_1)_\bullet\circ f_1^!\simeq f_0^!\circ (g_0)_\bullet.$$

Applying this to the square
$$
\CD
X_1\times X_1\times X_2   @>{ \on{id}_{X_1}\times  p_{X_1}\times \on{id}_{X_2}}>>   X_1\times X_2  \\
@V{p_{X_1}\times \on{id}_{X_1}\times \on{id}_{X_2}}VV   @VV{p_{X_1}\times \on{id}_{X_2}}V   \\
X_1\times X_2 @>{p_{X_1}\times \on{id}_{X_2}}>>  X_2,
\endCD
$$

we obtain a natural transformation from \eqref{e:stage 4 bis} to the functor
\begin{equation}  \label{e:stage 5}
(p_{X_1}\times \on{id}_{X_2})_\bullet\circ (\on{id}_{X_1}\times p_{X_1}\times \on{id}_{X_2})_!
\circ (\Delta_{X_1}\times\on{id}_{X_1}\times \on{id}_{X_2})^!\circ (\on{id}_{X_1}\times \Delta_{X_1}\times \on{id}_{X_2})_!.
\end{equation} 

I.e., we are now considering the diagram
$$
\xy
(0,0)*+{X_1\times X_1\times X_2}="X";
(30,-15)*+{X_1\times X_2}="Y";
(30,-30)*+{X_2}="Z";
(60,0)*+{\,\,X_1\times X_1\times X_1\times X_2}="W";
(60,15)*+{X_1\times X_1\times X_2}="V"
{\ar@{->}_{\on{id}_{X_1}\times p_{X_1}\times \on{id}_{X_2}\,\,} "X";"Y"}
{\ar@{->}^{\,\,\,p_{X_1}\times \on{id}_{X_2}} "Y";"Z"}
{\ar@{->}^{\Delta_{X_1} \times \on{id}_{X_1}\times  \on{id}_{X_2}} "X";"W"}
{\ar@{->}^{\on{id}_{X_1}\times \Delta_{X_1}\times \on{id}_{X_2}} "V";"W"}
\endxy
$$

\sssec{}

Returing again to \eqref{e:Cart diagram}, we have a natural transformation
$$(g_1)_!\circ f_1^!\to f_0^!\circ (g_0)_!,$$
obtained by adjunction from the isomorphism
$$f_1^!\circ g_0^!\simeq g_1^!\circ f_0^!.$$

\medskip

Applying this to the square
$$
\CD
X_1\times X_1\times X_2 @>{\Delta_{X_1} \times \on{id}_{X_1}\times  \on{id}_{X_2}}>> X_1\times X_1\times X_1\times X_2 \\
@VV{\on{id}_{X_1}\times p_{X_1}\times \on{id}_{X_2}}V  @VV{\on{id}_{X_1}\times \on{id}_{X_1}\times  p_{X_1}\times \on{id}_{X_2}}V   \\
X_1\times X_2 @>{\Delta_{X_1} \times  \on{id}_{X_2}}>>  X_1\times X_1\times X_2, 
\endCD
$$

we obtain a natural transformation from \eqref{e:stage 5} to the functor
\begin{equation}  \label{e:stage 6}
(p_{X_1}\times \on{id}_{X_2})_\bullet\circ (\Delta_{X_1} \times  \on{id}_{X_2})^!\circ (\on{id}_{X_1}\times \on{id}_{X_1}\times  p_{X_1}\times \on{id}_{X_2})_!
\circ (\on{id}_{X_1}\times \Delta_{X_1}\times \on{id}_{X_2})_!.
\end{equation}

I.e., we are now looking at the diagram
$$
\CD
& &  X_1\times X_1\times X_2 \\
& & @VV{\on{id}_{X_1}\times \Delta_{X_1}\times \on{id}_{X_2}}V  \\
& & X_1\times X_1\times X_1\times X_2 \\
& & @VV{\on{id}_{X_1}\times \on{id}_{X_1}\times p_{X_1}\times \on{id}_{X_2}}V  \\
X_1\times X_2 @>{\Delta_{X_1} \times  \on{id}_{X_2}}>> X_1\times X_1\times X_2 \\
@VV{p_{X_1}\times \on{id}_{X_2}}V \\
X_2.
\endCD
$$

However, the composed morphism
$(\on{id}_{X_1}\times \on{id}_{X_1}\times p_{X_1}\times \on{id}_{X_2})\circ (\on{id}_{X_1}\times \Delta_{X_1}\times \on{id}_{X_2})$
equals the identity, and hence, the functor in \eqref{e:stage 6} identifies with 
$$(p_{X_1}\times \on{id}_{X_2})_\bullet\circ (\Delta_{X_1} \times  \on{id}_{X_2})^!.$$

\medskip

When applied to $\CM\boxtimes \CP$, this yields $\sF_{X_1\to X_2,\CP}(\CM)$, i.e., the right-hand side
in \eqref{e:left to right}, applied to $\CM\in \Dmod(X_1)$. 

\sssec{Summary}

Here is the picture of the evolution of diagrams (the highlighted portion is the one to undergo 
base change). 

$$
\CD
& & & & X_1\times X_1\times X_1\times X_2 \\
& & & & @V{\on{id}_{X_1}\times \Delta_{X_1}\times \on{id}_{X_1}\times \on{id}_{X_2}}VV  \\
& & X_1\times X_1\times X_1\times X_2 
@>{\Delta_{X_1}\times\on{id}_{X_1}\times \on{id}_{X_1}\times \on{id}_{X_2}}>> X_1\times X_1\times X_1\times X_1\times X_2 \\
& & @VV{p_{X_1}\times \on{id}_{X_1}\times \on{id}_{X_1}\times \on{id}_{X_2}}V   \\
\mathbf {X_1\times X_2} @>{\Delta_{X_1}\times\on{id}_{X_2}}>> \mathbf{X_1\times X_1\times X_2} \\
@VV{p_{X_1}\times \on{id}_{X_2}}V  \\
\mathbf{X_2}
\endCD
$$

$$
\xy
(5,0)*+{\mathbf{X_1\times X_1\times X_2}}="X";
(5,-25)*+{X_1\times X_2}="Y";
(5,-45)*+{X_2.}="Z";
(25,-15)*+{\mathbf{X_1\times X_1\times X_1\times X_2}}="W";
(100,-15)*+{\,\,\mathbf{X_1\times X_1\times X_1\times X_1\times X_2}}="U";
(100,10)*+{X_1\times X_1\times X_1\times X_2}="V";
{\ar@{->}_{p_{X_1}\times \on{id}_{X_1}\times \on{id}_{X_2}} "X";"Y"}
{\ar@{->}_{p_{X_1}\times \on{id}_{X_2}} "Y";"Z"}
{\ar@{->}^{\on{id}_{X_1}\times  \Delta_{X_1} \times \on{id}_{X_2}} "X";"W"}
{\ar@{->}^{\Delta_{X_1}\times  \on{id}_{X_1}\times  \on{id}_{X_1}\times \on{id}_{X_2}} "W";"U"}
{\ar@{->}_{\on{id}_{X_1}\times \Delta_{X_1}\times \on{id}_{X_1}\times \on{id}_{X_2}} "V";"U"}
\endxy
$$
$$
\xy
(5,0)*+{X_1\times X_1\times X_2}="X";
(5,-20)*+{X_1\times X_2}="Y";
(5,-40)*+{X_2.}="Z";
(65,0)*+{\mathbf{X_1\times X_1\times X_1\times X_2}}="W";
(100,-30)*+{\,\,\mathbf{X_1\times X_1\times X_1\times X_1\times X_2}}="U";
(100,10)*+{\mathbf{X_1\times X_1\times X_1\times X_2}}="V";
{\ar@{->}_{p_{X_1}\times \on{id}_{X_1}\times \on{id}_{X_2}} "X";"Y"}
{\ar@{->}_{p_{X_1}\times \on{id}_{X_2}} "Y";"Z"}
{\ar@{->}^{\Delta_{X_1} \times \on{id}_{X_1}\times \on{id}_{X_2}} "X";"W"}
{\ar@{->}_{\on{id}_{X_1}\times \on{id}_{X_1}\times   \Delta_{X_1}\times \on{id}_{X_2}} "W";"U"}
{\ar@{->}^{\on{id}_{X_1}\times \Delta_{X_1}\times \on{id}_{X_1}\times \on{id}_{X_2}} "V";"U"}
\endxy
$$

$$
\xy
(5,0)*+{X_1\times X_1\times X_2}="X";
(5,-20)*+{X_1\times X_2}="Y";
(5,-40)*+{X_2.}="Z";
(65,0)*+{X_1\times X_1\times X_1\times X_2}="W";
(65,30)*+{X_1\times X_1\times X_2}="U";
(100,15)*+{X_1\times X_1\times X_1\times X_2}="V";
{\ar@{->}_{p_{X_1}\times \on{id}_{X_1}\times \on{id}_{X_2}} "X";"Y"}
{\ar@{->}_{p_{X_1}\times \on{id}_{X_2}} "Y";"Z"}
{\ar@{->}^{\Delta_{X_1} \times \on{id}_{X_1}\times \on{id}_{X_2}} "X";"W"}
{\ar@{<-}^{\on{id}_{X_1}\times \Delta_{X_1}\times \on{id}_{X_2}} "W";"U"}
{\ar@{<-}_{\,\,\,\,\,\,\,\on{id}_{X_1}\times \Delta_{X_1}\times \on{id}_{X_2}} "V";"U"}
\endxy
$$

At this stage we note that the object of interest on $X_1\times X_1\times X_1\times X_2$ 
comes as a pullback under $\on{id}_{X_1}\times p_{X_1}\times \on{id}_{X_1}\times \on{id}_{X_2}$
from $X_1\times X_1\times X_2$. 

\medskip

Hence, we resume with the next diagram:

$$
\xy
(0,0)*+{\mathbf{X_1\times X_1\times X_2}}="X";
(0,-15)*+{\mathbf{X_1\times X_2}}="Y";
(20,-25)*+{\mathbf{X_2}}="Z";
(60,0)*+{\,\,X_1\times X_1\times X_1\times X_2}="W";
(60,15)*+{X_1\times X_1\times X_2}="V"
{\ar@{->}_{p_{X_1}\times \on{id}_{X_1}\times \on{id}_{X_2}} "X";"Y"}
{\ar@{->}^{\,\,\,p_{X_1}\times \on{id}_{X_2}} "Y";"Z"}
{\ar@{->}^{\Delta_{X_1} \times \on{id}_{X_1}\times  \on{id}_{X_2}} "X";"W"}
{\ar@{->}^{\on{id}_{X_1}\times \Delta_{X_1}\times \on{id}_{X_2}} "V";"W"}
\endxy
$$

$$
\xy
(0,0)*+{\mathbf{X_1\times X_1\times X_2}}="X";
(30,-15)*+{\mathbf{X_1\times X_2}}="Y";
(30,-30)*+{X_2}="Z";
(60,0)*+{\,\,\mathbf{X_1\times X_1\times X_1\times X_2}}="W";
(60,15)*+{X_1\times X_1\times X_2}="V"
{\ar@{->}_{\on{id}_{X_1}\times p_{X_1}\times \on{id}_{X_2}\,\,} "X";"Y"}
{\ar@{->}^{\,\,\,p_{X_1}\times \on{id}_{X_2}} "Y";"Z"}
{\ar@{->}^{\Delta_{X_1} \times \on{id}_{X_1}\times  \on{id}_{X_2}} "X";"W"}
{\ar@{->}^{\on{id}_{X_1}\times \Delta_{X_1}\times \on{id}_{X_2}} "V";"W"}
\endxy
$$

$$
\CD
& &  X_1\times X_1\times X_2 \\
& & @VV{\on{id}_{X_1}\times \Delta_{X_1}\times \on{id}_{X_2}}V  \\
& & X_1\times X_1\times X_1\times X_2 \\
& & @VV{\on{id}_{X_1}\times \on{id}_{X_1}\times p_{X_1}\times \on{id}_{X_2}}V  \\
X_1\times X_2 @>{\Delta_{X_1} \times  \on{id}_{X_2}}>> X_1\times X_1\times X_2 \\
@VV{p_{X_1}\times \on{id}_{X_2}}V \\
X_2,
\endCD
$$

while the latter diagram is equivalent to
$$
\CD
X_1\times X_2 @>{\Delta_{X_1} \times  \on{id}_{X_2}}>> X_1\times X_1\times X_2 \\
@VV{p_{X_1}\times \on{id}_{X_2}}V \\
X_2.
\endCD
$$

\ssec{Specializing to the separated case}  \label{ss:nat trans sep}

Assume now that the scheme $X_1$ is separated. In this case the natural transformation \eqref{e:left to right}
can be significantly simplified.

\sssec{}

First, we note if $f:Y\to Z$ is a separated morphism, there is a canonically defined natural transformation
$$f_!\to f_\bullet,$$
described as follows.

\medskip

Consider the Cartesian diagram
$$
\CD
Y\underset{Z}\times Y  @>{\on{pr}_2}>>  Y  \\
@V{\on{pr}_1}VV  @VV{f}V  \\
Y @>{f}>>  Z.
\endCD
$$

By \eqref{e:nat trans dir im}, we have a natural transformation
$$f_!\circ (\on{pr}_1)_\bullet\to f_\bullet\circ (\on{pr}_2)_!.$$

Pre-composing with $(\Delta_{Y/Z})_\bullet\simeq (\Delta_{Y/Z})_!$, where 
$$\Delta_{Y/Z}:Y\to Y\underset{Z}\times Y$$
(it is here that we use the assumption that $\Delta_{Y/Z}$ is a closed embedding), 
we obtain the desired natural transformation
$$f_!\simeq f_!\circ (\on{pr}_1\circ \Delta_{Y/Z})_\bullet\simeq 
f_!\circ (\on{pr}_1)_\bullet\circ (\Delta_{Y/Z})_\bullet\to f_\bullet\circ (\on{pr}_1)_! \circ (\Delta_{Y/Z})_!\simeq
f_\bullet\circ (\on{pr}_1 \circ \Delta_{Y/Z})\simeq f_\bullet.$$

\sssec{}

For $X_1$ separated, the morphism $\on{pr}_2:X_1\times X_2\to X_2$ is separated, and so the
functor $\sF^{\on{op}}_{X_1\to X_2;\CP}$ admits a natural transformation to the functor
$$\CM\mapsto (\on{pr}_2)_\bullet(\on{pr}_1^\bullet(\CM)\overset{\bullet}\otimes\CP).$$

\medskip 

In this case, the natural transformation \eqref{e:left to right} is the composition of the above map
$$(\on{pr}_2)_!(\on{pr}_1^\bullet(\psId_{X_1}(\CM))\overset{\bullet}\otimes\CP)\to 
(\on{pr}_2)_\bullet(\on{pr}_1^\bullet(\psId_{X_1}(\CM))\overset{\bullet}\otimes\CP),$$
and a natural transformation induced by a canonically defined map 
\begin{equation} \label{e:left to right sep}
\on{pr}_1^\bullet(\psId_{X_1}(\CM))\overset{\bullet}\otimes\CP \to 
\on{pr}_1^!(\CM)\sotimes \CP,
\end{equation}
described below. 

\sssec{}

Recall that for $X_1$ separated, 
$$\psId_{X_1}(\CM)\simeq \CM\sotimes k_{X_1}.$$

The map in \eqref{e:left to right sep} comes from \eqref{e:nat trans pullback} applied to the 
Cartesian diagram
$$
\CD
X_1\times X_2 @>{\Delta_{X_1}\times \on{id}_{X_2}}>>  X_1\times X_1\times X_2 \\
@V{\Delta_{X_1}\times \on{id}_{X_2}}VV   @VV{\on{id}_{X_1}\times \Delta_{X_1}\times \on{id}_{X_2}}V   \\
X_1\times X_1\times X_2 @>{\Delta_{X_1}\times \on{id}_{X_1}\times \on{id}_{X_2}}>> X_1\times X_1\times X_1\times X_2,
\endCD
$$
and the object 
$$\CM\boxtimes k_{X_1}\boxtimes \CP\in \Dmod(X_1\times X_1\times X_1\times X_2).$$

\section{Relation to $\CO$-modules}  \label{s:O}

The goal of this section is to express the condition on an object $\CQ\in \Dmod(X_1\times X_2)$
that the corresponding functor $\sF_{X_1\to X_2,\CQ}$ preserve compactness, in terms of the underlying
$\CO$-modules. The material of this section will be used in \secref{s:proof of subq}, but not elsewhere
in the paper. 

\ssec{Recollections} \label{ss:recall IndCoh}

As we will be considering the forgetful functor from D-modules to $\CO$-modules, derived algebraic geometry
comes into play. Henceforth in this and the next section, by a ``scheme" we will understand an 
\emph{eventually coconnective DG scheme almost of finite type}, see \secref{sss:dag}.

\sssec{}

For a scheme $X$ understood as above, we will consider the categories $\IndCoh(X)$ and $\QCoh(X)$
(see \cite[Sect. 1]{IndCoh} for the definition of the former and \cite[Sect. 1]{QCoh} of the latter category). 

\medskip

The category $\IndCoh(X)$
is compactly generated, and $\IndCoh(X)^c=\Coh(X)$, the latter being the full (but not cocomplete)
subcategory of $\QCoh(X)$ consisting of bounded complexes with coherent cohomology sheaves. 

\medskip

By a theorem of Thomason-Trobaugh, the category $\QCoh(X)$ is also compactly generated by the subcategory
$\QCoh(X)^{\on{perf}}$ of perfect complexes. 

\medskip

The categories $\IndCoh(X)$ and $\QCoh(X)$ are connected by a pair of adjoint functors 
$$\Psi_X:\IndCoh(X)\to \QCoh(X):\Xi_X,$$
where $\Psi_X$ is obtained by ind-extending the tautological embedding $\Coh(X)\hookrightarrow \QCoh(X)$,
and $\Xi_X$ by ind-extending the tautological embedding $\QCoh(X)^{\on{perf}}\hookrightarrow \Coh(X)\hookrightarrow \IndCoh(X)$.

\medskip

The functor $\Xi_X$ is fully faithfull by construction. 
%
%
The functors $\Psi_X$ and $\Xi_X$ are mutually inverse equivalences if and only if $X$ is a smooth classical scheme.

\sssec{}

For a pair of schemes $X_1$ and $X_2$, external tensor product defines a functor 
\begin{equation} \label{e:external tensor}
\IndCoh(X_1)\otimes \IndCoh(X_2)\to \IndCoh(X_1\times X_2),
\end{equation}
which is an equivalence by \cite[Proposition 6.4.2]{IndCoh}.





\medskip

For a morphism $f:X_1\to X_2$, we shall denote by $f^{\IndCoh}_*$ and $f^!$
the corresponding functors $\IndCoh(X_1)\to \IndCoh(X_2)$ and $\IndCoh(X_2)\to \IndCoh(X_1)$,
respectively, see \cite[Sects. 3.1 and 5.2.3]{IndCoh}. 

\medskip

In particular, for a scheme $X$ we have the functor 
$$\overset{!}\otimes: \IndCoh(X)\otimes \IndCoh(X)\to \IndCoh(X),$$
that identifies, under the equivalence $\IndCoh(X)\otimes \IndCoh(X)\simeq \IndCoh(X\times X)$, with 
the functor $\Delta_X^!:\IndCoh(X\times X)\to \IndCoh(X)$. 

\medskip

We note that for $f=p_X$, the corresponding functor $(p_X)^{\IndCoh}_*$ is canonically
isomorphic to
$$\IndCoh(X)\overset{\Psi_X}\longrightarrow \QCoh(X) \overset{\Gamma_X}\longrightarrow \Vect,$$
where 
$$\Gamma_X:\QCoh(X)\to \Vect$$
is the usual functor of global sections. 

\sssec{}  \label{sss:duality on IndCoh}

We recall (see \cite[Sect. 9.2.1]{IndCoh}) that Serre duality defines a canonical equvalence
$$\bD_X^{\on{Se}}:\IndCoh(X)^\vee\simeq \IndCoh(X).$$

The corresponding functor
$$(\IndCoh(X)^c)^{\on{op}}=(\IndCoh(X)^\vee)^c \overset{\bD^{\on{Se}}_X}\longrightarrow \IndCoh(X)^c$$
is the usual Serre duality functor
$$\BD^{\on{Se}}_X:\Coh(X)^{\on{op}}\to \Coh(X),$$
see \cite[Sect. 9.5]{IndCoh}.

\medskip

Under this equivalence, the unit object 
$$\bu_{\IndCoh(X)}\in 
\IndCoh(X)^\vee\otimes  \IndCoh(X) 
\simeq \IndCoh(X)\otimes  \IndCoh(X) 
\simeq \IndCoh(X\times X)$$
identifies with $(\Delta_X)^{\IndCoh}_*(\omega_X)$, where $\omega_X=p_X^!(k)$. 

\medskip

We note (see \cite[Proposition 9.6.12]{IndCoh}) that due to the assumption that $X$ is \emph{eventually coconnective}, we have
$\omega_X\in \Coh(X)$. In particular, if $X$ is separated, the object 
$$\bu_{\IndCoh(X)}\in \IndCoh(X)^\vee\otimes  \IndCoh(X)$$
is compact.

\sssec{}  \label{sss:dualities and QCoh}

The category $\QCoh(X)$ is also canonically self dual: the equivalence
$$\bD_{X}^{\on{nv}}:\QCoh(X)^\vee\to \QCoh(X)$$
is uniquely determined by the fact that the corresponding equivalence
$$(\QCoh(X)^c)^{\on{op}}=(\QCoh(X)^\vee)^c \overset{\bD^{\on{nv}}_X}\longrightarrow \QCoh(X)^c$$
is the usual duality functor
$$\BD^{\on{nv}}_X:(\QCoh(X)^{\on{perf}})^{\on{op}}\to \QCoh(X)^{\on{perf}},\quad \CE\mapsto \uHom_X(\CE,\CO_X).$$

\medskip

The corresponding evaluation functor
$$\on{ev}_{\QCoh_X}:\QCoh(X)\otimes \QCoh(X)\to \Vect$$
is 
$$\QCoh(X)\otimes \QCoh(X)\simeq \QCoh(X\times X)\overset{\Delta_X^*}\longrightarrow \QCoh(X)\overset{\Gamma_X}\to \Vect,$$
and the object 
$$\bu_{\QCoh(X)}\in \QCoh(X)\otimes \QCoh(X)$$ identifies with
$$(\Delta_X)_*(\CO_X)\in \QCoh(X\times X)\simeq \QCoh(X)\otimes \QCoh(X).$$

\medskip

We recall (see \cite[Proposition 9.3.3]{IndCoh}) that with respect to the self-dualities $\bD_{X}^{\on{nv}}$ and $\bD_{X}^{\on{Se}}$,
the dual of the functor
$$\Psi_X:\IndCoh(X)\to \QCoh(X)$$
is the functor
$$\Upsilon_X:\QCoh(X)\to \IndCoh(X),\quad \CE\mapsto \CE\underset{\CO_X}\otimes \omega_X,$$
where $\underset{\CO_X}\otimes$ is the functor
$$\QCoh(X)\otimes \IndCoh(X)\to \IndCoh(X)$$
equal to the ind-extension of the action of $\QCoh(X)^{\on{perf}}$ on $\Coh(X)$ by tensor products. 

\sssec{}

We will consider the adjoint
pair of (continuous) functors
$$\ind_X:\IndCoh(X)\rightleftarrows \Dmod(X):\oblv_X,$$
see \cite[Sect. 5.1.5]{DrGa1}.

\medskip

The functor $\oblv_X$ is conservative, which implies that the essential image of $\ind_X$ generates
$\IndCoh(X)$. The latter, in turn, implies that the essential image of $\IndCoh(X)^c\simeq \Coh(X)$
Karoubi-generates $\Dmod(X)^c$.

\medskip

Consider now the functor
$$\ind_X^{\on{left}}:=\ind_X\circ \Upsilon_X:\QCoh(X)\to \Dmod(X).$$

It is shown in \cite[Lemma 3.4.7]{Crys} that $\ind_X^{\on{left}}$ also admits a right adjoint, denoted $\oblv^{\on{left}}_X$, and we have
$$\oblv_X\simeq  \Upsilon_X\circ \oblv^{\on{left}}_X.$$

In particular, $\oblv^{\on{left}}_X$ is also conservative. Hence, the essential image of $\QCoh(X)^{\on{perf}}$ 
under the functor $\ind^{\on{left}}_X$ also Karoubi-generates $\Dmod(X)^c$. 




\sssec{}

For a morphism $f:X_1\to X_2$ we have canonical isomorphisms
$$\ind_{X_2}\circ f^{\IndCoh}_*\simeq f_\bullet\circ \ind_{X_1},\quad \IndCoh(X_1)\to \Dmod(X_2)$$
and
$$\oblv_{X_1}\circ f^!\simeq f^!\circ \oblv_{X_2}, \quad \Dmod(X_2)\to \IndCoh(X_1).$$

\medskip

Finally, we recall (see \cite[Sect. 5.3.4]{DrGa1}) that with respect to the equivalences $\bD_X^{\on{Se}}$ and $\bD_X^{\on{Ve}}$, 
the functors $\ind_X$ and $\oblv_X$ satisfy
\begin{equation} \label{e:oblv and ind}
\ind_X^\vee\simeq \oblv_X.
\end{equation}

\ssec{Criteria for preservation of compactness}

In this subsection we will give more explicit criteria 
for an object $\CQ\in \Dmod(X_1\times X_2)^c$ to 
satisfy the assumption of \thmref{t:schemes}, i.e., for the functor 
$$\sF_{X_1\to X_2,\CQ}:\Dmod(X_1)\to \Dmod(X_2)$$
to preserve compactness (or, equivalently, to admit a continuous right adjoint).  

\begin{rem}
By \corref{c:when left D},
the same criterion,  \emph{with the roles of $X_1$ and $X_2$ swapped}, will tell us when $\sF_{X_1\to X_2,\CQ}$ 
admits a \emph{left} adjoint.
\end{rem}

\sssec{}

For $\sF:\Dmod(X_1)\to \Dmod(X_2)$ consider the functors 
\begin{equation} \label{e:precomp with ind}
\sF\circ \ind_{X_1}:\IndCoh(X_1)\to \Dmod(X_2).
\end{equation} 
\begin{equation} \label{e:precomp with ind Ups}
\sF\circ \ind^{\on{left}}_{X_1}:\QCoh(X_1)\to \Dmod(X_2).
\end{equation} 

\medskip

We claim:

\begin{lem} \label{l:criter}
For a functor $\sF:\Dmod(X_1)\to \Dmod(X_2)$ the following conditions are equivalent:

\smallskip

\noindent{\em(a)} $\sF$ preserves compactness.

\smallskip

\noindent{\em(b)} $\sF\circ \ind_{X_1}$ preserves compactness.

\smallskip

\noindent{\em(c)} $\sF\circ \ind^{\on{left}}_{X_1}$ preserves compactness. 

\end{lem}

\begin{proof}

The implication (a) $\Rightarrow$ (b) (resp., (c)) follows from the fact that the functor $\ind_{X_1}$ 
(resp., $\ind^{\on{left}}_{X_1}$) preserves compactness, since its right adjoint, i.e., $\oblv_{X_1}$
(resp., $\oblv^{\on{left}}_{X_1}$),  is continuous.

\medskip

The implication (b) (resp., (c)) $\Rightarrow$ (a) follows from the fact that the image of 
$\Coh(X_1)$ under $\ind_{X_1}$ (resp., $\QCoh(X_1)^{\on{perf}}$ under $\ind^{\on{left}}_{X_1}$)
Karoubi-generates $\Dmod(X_1)^c$.

\end{proof}

\sssec{}  \label{sss:crit explained}

The usefulness of \lemref{l:criter} lies in the fact that for $\CQ\in \Dmod(X_1\times X_2)$, 
the functors $\sF_{X_1\to X_2,\CQ}\circ \ind_{X_1}$ and $\sF_{X_1\to X_2,\CQ}\circ \ind^{\on{left}}_{X_1}$ 
are more explicit 
than the original functor $\sF_{X_1\to X_2,\CQ}$.

\bigskip

Namely, for $\CF\in \IndCoh(X)$, the object $\sF_{X_1\to X_2,\CQ}\circ \ind_{X_1}(\CF)\in \Dmod(X)$ is calculated
as follows:

\medskip

Consider the functor
\begin{equation} \label{e:functor on X_1}
\Dmod(X_1)\overset{\oblv_{X_1}}\longrightarrow \IndCoh(X_1) \overset{\CF\sotimes -}\longrightarrow 
\IndCoh(X_1) \overset{(p_{X_1})^{\IndCoh}_*}\to \Vect,
\end{equation}
which is the dual of the functor $\Vect\to \IndCoh(X_1)$, corresponding to the object $\ind_{X_1}(\CF)$. 

\medskip

Then 
$$\sF_{X_1\to X_2,\CQ}\circ \ind_{X_1}(\CF)\simeq 
\left(\text{\eqref{e:functor on X_1}}\otimes \on{Id}_{\Dmod(X_2)}\right)(\CQ).$$

\medskip

Similarly, for $\CE\in \QCoh(X)$, consider the functor
\begin{equation} \label{e:functor on X_1 QCoh}
\Dmod(X_1)\overset{\oblv_{X_1}}\longrightarrow \IndCoh(X_1)\overset{\CE\otimes -}\longrightarrow
\IndCoh(X_1)\overset{(p_{X_1})^{\IndCoh}_*}\longrightarrow \Vect,
\end{equation}
or which is the same
$$\Dmod(X_1)\overset{\oblv_{X_1}}\longrightarrow \IndCoh(X_1)\overset{\Psi_{X_1}}\longrightarrow  \QCoh(X_1)
\overset{\CE\otimes -}\longrightarrow \QCoh(X_1) \overset{\Gamma_{X_1}} \longrightarrow \Vect.$$

Then 
$$\sF_{X_1\to X_2,\CQ}\circ \ind^{\on{left}}_{X_1}(\CE)\simeq  
\left(\text{\eqref{e:functor on X_1 QCoh}}\otimes \on{Id}_{\Dmod(X_2)}\right)(\CQ).$$

\medskip

In other words, the point is that the functors \eqref{e:precomp with ind} and \eqref{e:precomp with ind Ups}
only involve the operation of direct image
$$(p_{X_1})^{\IndCoh}_*:\IndCoh(X_1)\to \Vect \text{ and } \Gamma_{X_1}:\QCoh(X_1)\to \Vect,$$
rather than the more complicated functor of de Rham cohomology
$$(p_{X_1})_\bullet:\Dmod(X_1)\to  \Vect.$$

\sssec{} 

From \lemref{l:criter} we obtain: 

\begin{cor}  \label{c:criter with line bundle}
Assume that $X_1$ is quasi-projective with an ample line bundle $\CL$. 
Let $\CQ$ be an object $\Dmod(X_1\times X_2)$. Then the
functor $\sF_{X_1\to X_2,\CQ}$ preserves compactness 
if and only if the following equivalent conditions hold:

\medskip

\noindent{\em(i)} For any $i\in {\mathbb Z}$, the object
$$\sF_{X_1\to X_2,\CQ}\circ \ind^{\on{left}}_{X_1}(\CL^{\otimes i})\in \Dmod(X_2)$$
is compact. 

\medskip

\noindent{\em(ii)} There exists an integer $i_0$ such that the objects
$$\sF_{X_1\to X_2,\CQ}\circ \ind^{\on{left}}_{X_1}(\CL^{\otimes i})\in \Dmod(X_2)$$
are compact for all $i\geq i_0$. 

\medskip

\noindent{\em(iii)} There exists an integer $i_0$ such that the objects
$$\sF_{X_1\to X_2,\CQ}\circ \ind^{\on{left}}_{X_1}(\CL^{\otimes i})\in \Dmod(X_2)$$
are compact for all $i\leq i_0$. 

\medskip

\noindent{\em(iv)} For some specific interval $[i_1,i_2]$ that only depends on $X_1$, the objects 
$$\sF_{X_1\to X_2,\CQ}\circ \ind^{\on{left}}_{X_1}(\CL^{\otimes i})\in \Dmod(X_2)$$
are compact for all $i_1\leq i\leq i_2$. 

\end{cor}

\begin{proof}

By \lemref{l:criter}, we need to check when the functor $\sF_{X_1\to X_2,\CQ}\circ \ind^{\on{left}}_{X_1}$ 
preserves compactness. The statement of the corollary follows from the fact that
the objects $\CL^{\otimes i}$ in all of the four cases Karoubi-generate
$\QCoh(X)^{\on{perf}}$. 

\end{proof}

In particular, we obtain:

\begin{cor}  \label{c:affine}
Assume that $X_1$ is affine. Then Let $\CQ$ be an object $\Dmod(X_1\times X_2)$. Then the
functor $\sF_{X_1\to X_2,\CQ}$ preserves compactness if and only if
\begin{equation} \label{e:forget down affine}
\left((\Gamma_{X_1}\circ \Psi_{X_1}\circ \oblv_{X_1})\otimes \on{Id}_{\Dmod(X_2)}\right)(\CQ)\in \Dmod(X_2)
\end{equation} 
is compact.
\end{cor}

\sssec{} \label{sss:subq affine}

Let us note that \corref{c:affine} implies the assertion of \thmref{t:subquotient} in the particular case when $X_1$ is affine:

\medskip

Let us recall that for a scheme $X$ the category $\IndCoh(X)$ carries a canonical t-structure,
see \cite[Sect. 1.2]{IndCoh}. Its basic feature is that the functor $\Psi_X:\IndCoh(X)\to \QCoh(X)$
is t-exact.

\medskip

Note that since $X_1$ is smooth, the functor $\oblv_{X_1}$ is t-exact (see \cite[Proposition 4.2.11(a)]{Crys}).
Since $X_1$ is affine, we obtain that the composed functor 
$$\Gamma_{X_1}\circ \Psi_{X_1}\circ \oblv_{X_1}:\Dmod(X_1)\to \Vect$$
is t-exact. 

\medskip

Hence, the same is true for the functor \eqref{e:forget down affine} (see \secref{ss:ten prod t-structure}, 
where the general statement along these lines is explained). 

\medskip

Now, the assertion of the theorem follows from the fact that if an object
of $\Dmod(X_2)$ is compact, then the same is true for
any subquotient of any of its cohomologies.

\qed

\ssec{Preservation of compactness and compactness of the kernel}

\sssec{}

Consider the category 
$$\IndCoh(X_1)\otimes \Dmod(X_2),$$
which is endowed with a forgetful functor 
\begin{multline} \label{e:oblv mixed}
\Dmod(X_1\times X_2)\simeq 
\Dmod(X_1)\otimes \Dmod(X_2)\overset{\oblv_{X_1}\otimes \on{Id}_{\Dmod(X_2)}}\longrightarrow \\
\to \IndCoh(X_1)\otimes \Dmod(X_2).
\end{multline}

We claim: 

\begin{prop} \label{p:O criter} 
Assume that $X_1$ is separated.  
Let $\CQ$ be an object of $\Dmod(X_1\times X_2)$, such that the functor 
$$\sF_{X_1\to X_2,\CQ}:\Dmod(X_1)\to \Dmod(X_2)$$
preserves compactness. Then the image of $\CQ$ under the
functor \eqref{e:oblv mixed} is compact in 
$$\IndCoh(X_1)\otimes \Dmod(X_2).$$
\end{prop}

\begin{proof} 

If $\sF_{X_1\to X_2,\CQ}$ preserves compactness, then so does the functor $\sF_{X_1\to X_2,\CQ}\circ \ind_{X_1}$. Hence,
by \corref{c:preserve compactness tensor}, the same is true for the functor 
\begin{equation} \label{e:ind Phi ten}
\on{Id}_{\bC}\otimes (\sF_{X_1\to X_2,\CQ}\circ \ind_{X_1}):\bC\otimes \IndCoh(X_1)\to \bC\otimes \Dmod(X_2)
\end{equation}
for any DG category $\bC$. 

\medskip

Note that the functor $\sF_{\CQ,X_1\to X_2}\circ \ind_{X_1}$ is defined by the kernel
\begin{equation} \label{e:ind Phi ten IndCoh}
\left(\on{Id}_{\IndCoh(X_1)^\vee}\otimes (\sF_{X_1\to X_2,\CQ}\circ \ind_{X_1})\right)(\bu_{\IndCoh(X_1)})\in
\IndCoh(X_1)^\vee\otimes  \Dmod(X_2).
\end{equation}

By \secref{sss:duality on IndCoh},  the assumption that $X_1$ be separated 
implies that the object 
$$\bu_{\IndCoh(X_1)}\in \IndCoh(X_1)^\vee\otimes \IndCoh(X_1)\simeq \IndCoh(X_1)\otimes \IndCoh(X_1)\simeq 
\IndCoh(X_1\times X_1)$$ 
is compact. Hence, taking in \eqref{e:ind Phi ten} $\bC:=\IndCoh(X_1)^\vee$,
we obtain that the object in \eqref{e:ind Phi ten IndCoh} is compact. 

\medskip

Finally, we observe that in terms of the identification 
$$\IndCoh(X_1)^\vee\otimes  \Dmod(X_2)\overset{\bD_{X_1}^{\on{Se}}\otimes \on{Id}_{\Dmod(X_2)}}\longrightarrow
\IndCoh(X_1)\otimes  \Dmod(X_2),$$
and using \eqref{e:oblv and ind}, the kernel of the functor $\sF_{\CQ,X_1\to X_2}\circ \ind_{X_1}$ identifies 
with
$$(\oblv_{X_1}\otimes \on{Id}_{\Dmod(X_2)})(\CQ).$$

\end{proof}

\sssec{}

We shall now prove:

\begin{thm}
Assume that the support of $\CQ$ is \emph{proper} over $X_2$. Then the assertion of \propref{p:O criter} is ``if and only if." 
\end{thm}

\begin{proof}

Set
$$\CK:=\left(\oblv_{X_1}\otimes \on{Id}_{\Dmod(X_2)}\right)(\CQ)\in \IndCoh(X_1)\otimes \Dmod(X_2).$$

Let $X_1\overset{j}\hookrightarrow \ol{X}_1$ be a compactification of $X_1$. Consider the object
$$\ol\CK:=\left(j^{\IndCoh}_*\otimes \on{Id}_{\Dmod(X_2)}\right)(\CK)
\in \IndCoh(\ol{X}_1)\otimes \Dmod(X_2).$$

\medskip

We claim that $\ol\CK$ is compact.  Let us assume this and finish the proof of the theorem. 

\medskip

By \lemref{l:criter} and \secref{sss:crit explained}, it suffices to show that for any $\CE\in \QCoh(\ol{X}_1)^{\on{perf}}$, 
we have
$$((p_{X_1})^{\IndCoh}_*\otimes \on{Id}_{\Dmod(X_2)})(\CE|_{X_1}\underset{\CO_{X_1}}\otimes \CK)\in \Dmod(X_2)^c.$$

However,
$$((p_{X_1})^{\IndCoh}_*\otimes \on{Id}_{\Dmod(X_2)})(\CE|_{X_1}\underset{\CO_{X_1}}\otimes \CK)\simeq
((p_{\ol{X}_1})^{\IndCoh}_*\otimes \on{Id}_{\Dmod(X_2)})(\CE\underset{\CO_{\ol{X}_1}}\otimes \ol\CK).$$

Note that the functor
$$\CE\underset{\CO_{\ol{X}_1}}\otimes -:\IndCoh(\ol{X}_1)\otimes \Dmod(X_2)\to \IndCoh(\ol{X}_1)\otimes \Dmod(X_2)$$
preserves compactness. Indeed, it admits a continuous right adjoint, given by $\CE^\vee\underset{\CO_{\ol{X}_1}}\otimes -$.

\medskip

Now, the required assertion follows from the fact that the functor
$$(p_{\ol{X}_1})^{\IndCoh}_*\otimes \on{Id}_{\Dmod(X_2)}: \IndCoh(\ol{X}_1)\otimes \Dmod(X_2)\to \Dmod(X_2)$$
preserves compactness, which follows from the corresponding fact (Serre's theorem) for
$$(p_{\ol{X}_1})_*: \IndCoh(\ol{X}_1)\to \Vect.$$

\medskip

To prove that $\ol\CK$ is compact we proceed as follows. 

\medskip

By \cite[Corollary 10.3.6]{IndCoh}, we interpret the category
$\IndCoh(\ol{X}_1)\otimes \Dmod(X_2)$ as the category $\IndCoh$ of the prestack $\ol{X}_1\times (X_2)_\dr$
(see \cite[Sect. 1.1.1]{Crys} for the definition of the de Rham prestack). 
Recall also that the assignment 
$$\CX\rightsquigarrow \IndCoh(\CX), \quad \CX\in \on{PreStk}_{\on{laft}}$$
satisfies Zariski descent (see \cite[Sect. 10.4.2]{IndCoh}).

\medskip

Note that the Zariski site of $\ol{X}_1\times (X_2)_\dr$ is in bijection with that of 
$\ol{X}_1\times X_2$. Set $$U:=X_1\times X_2 \text{ and } V:=\ol{X}_1\times X_2-S,$$ where $S$ is the support of $\CQ$, which
is closed in $\ol{X}_1\times X_2$, by assumption. By Zariski descent,  the category
$\IndCoh(\ol{X}_1)\otimes \Dmod(X_2)$ identifies with
$$\left(\IndCoh(\ol{X}_1)\otimes \Dmod(X_2)\right)_U\underset{\left(\IndCoh(\ol{X}_1)\otimes \Dmod(X_2)\right)_{U\cap V}}
\times \left(\IndCoh(\ol{X}_1)\otimes \Dmod(X_2)\right)_V.$$
Hence, it suffices to show that the restriction of $\ol\CK$ to both $U$ and $V$ is compact. However, the former yields $\CK$, 
and the latter zero.

\end{proof}

\ssec{The ULA property}  \label{ss:ULA}

\sssec{}

Let $X_1$ and $X_2$ be smooth classical schemes, and let $f:X_2\to X_1$ be a smooth morphism.

\begin{defn}
We say that $\CM\in \Dmod(X_2)$ is ULA with respect to $f$ if the functor 
\begin{equation} \label{e:ULA}
\CN\mapsto \CM\sotimes f^!(\CN),\quad \Dmod(X_1)\to \Dmod(X_2)
\end{equation}
preserves compactness. 
\end{defn}

Note that the question of being ULA is Zariski-local on $X_2$, and hence also on $X_1$. So, 
with no restriction of generailty we can assume that $X_1$ and $X_2$ are affine. 

\sssec{}

For $\CM$ as above take 
$$\CQ:=(f\times \on{id}_{X_2})_\bullet(\CM)\in \Dmod(X_1\times X_2),$$
where by a slight abuse of notation we denote by $f\times \on{id}_{X_2}:X_2\to X_1\times X_2$ the graph of the map 
$f$. 

\medskip

Then the functor \eqref{e:ULA} is the same as the corresponding functor $\sF_{X_1\to X_2,\CQ}$, so the above
analysis applies. 

\medskip 

In particular, we obtain:

\begin{cor}
If $\CM$ is ULA with respect to $f$, then the same is true for any subquotient of any of its cohomologies.
\end{cor}

This follows immediately from \thmref{t:subquotient} in the affine case, established in \secref{sss:subq affine}.
Another proof follows from \propref{p:vertical} below. 

\sssec{}

Applying \corref{c:stable under dual}, we obtain:

\begin{cor}
If $\CM\in \Dmod(X_2)$ is ULA with respect to $f$, then it is compact, and $\BD^{\on{Ve}}_{X_2}(\CM)$ is also ULA.
\end{cor}

Finally, from \corref{c:our functor as other} and \secref{ss:nat trans sep} we obtain:

\begin{cor}
Let $\CM\in \Dmod(X_2)$ be ULA with respect to $f$. Then the functor
$$\CN\in \Dmod(X_1)\, \rightsquigarrow \,  \CM\overset{\bullet}\otimes f^\bullet(\CN)$$ 
takes values in $\Dmod(X_2)\subset \on{Pro}(\Dmod(X_2))$, and the natural map
$$\CM\overset{\bullet}\otimes f^\bullet(\CN)\to \CM\sotimes f^!(\CN)[2\dim(X_1)]$$
coming from \eqref{e:nat trans pullback} for the commutative diagram
$$
\CD
X_2 @>{f\times \on{id}_{X_2}}>>  X_1\times X_2  \\
@V{f\times \on{id}_{X_2}}VV     @V{\Delta_{X_1}\times \on{id}_{X_2}}VV    \\
X_1\times X_2  @>{\on{id}_{X_1}\times (f\times \on{id}_{X_2})}>>  X_1\times X_1\times X_2
\endCD
$$ 
and the object $$\CN\boxtimes k_{X_1}\boxtimes \CM\in \Dmod(X_1\times X_1\times X_2),$$
is an isomorphism.
\end{cor}

\sssec{}

Let $\on{D}_{X_2/X_1}$ be the sheaf of vertical differential operators on $X_2$ with respect to $f$.
I.e., this is the subsheaf of rings in $\on{D}_{X_2}$ generated by all functions and $T_{X_2/X_1}\subset T_{X_2}$.
Still equivalently, $\on{D}_{X_2/X_1}$ is the centralizer of $f^\cdot(\CO_{X_1})$ in $\on{D}_{X_2}$.

\medskip

We consider the corresponding DG category $\Dmod_{\on{rel}}(X_2)$ (see, e.g., \cite[Sect. 6.3]{DrGa1}). 
By definition, in the affine situation, the category $\Dmod_{\on{rel}}(X_2)$ is 
compactly generated by the object $\on{D}_{X_2/X_1}$. The category $\Dmod_{\on{rel}}(X_2)$ is endowed with 
continuous conservative functors
$$\Dmod(X_2)\overset{\oblv_{\on{abs}\to\on{rel},X_2}}\longrightarrow \Dmod_{\on{rel}}(X_2) \overset{\oblv_{\on{rel},X_2}}
\longrightarrow \IndCoh(X_2),$$
whose composition is the functor $\oblv_{X_2}$. The functors $\oblv_{\on{abs}\to\on{rel},X_2}$ and $\oblv_{\on{rel},X_2}$
admit left adjoints, 
denoted $\ind_{\on{rel}\to\on{abs},X_2}$ and $\ind_{\on{rel},X_2}$, respectively.

\medskip

In addition, the category $\Dmod_{\on{rel}}(X_2)$ carries a t-structure in which the functor 
$$\oblv_{\on{rel},X_2}:\Dmod_{\on{rel}}(X_2)\to \IndCoh(X_2)$$
is t-exact. This property determines the above t-structure uniquely.

\medskip

Finally, it is easy to see that an object of $\Dmod_{\on{rel}}(X_2)$ is compact if and only if it is cohomologically
bounded, and its cohomologies are finitely generated as $\on{D}_{X_2/X_1}$-modules.

\sssec{}

We claim:

\begin{prop} \label{p:vertical}
An object $\CM\in \Dmod(X_2)$ is ULA with respect to $f$ if and only if its image under the forgetful functor
$$\oblv_{\on{abs}\to\on{rel},X_2}:\Dmod(X_2)\to \Dmod_{\on{rel}}(X_2)$$
is compact.
\end{prop}

\begin{proof}

With no restriction of generality, we can assume that $X_1$ and $X_2$ are affine. Then the functor
\eqref{e:ULA} preserves compactness if and only if it sends $\on{D}_{X_1}$ to a compact object
of $\Dmod(X_2)$. Since $X_1$ is smooth, instead of $\on{D}_{X_1}$ we can take 
$\ind_{X_1}(\omega_{X_1})$; it will still be a generator of $\Dmod(X_1)$. 

\medskip

Thus, we need to show that the object
$$\CM\sotimes f^!(\ind_{X_1}(\omega_{X_1}))\in \Dmod(X_2)$$
is compact if and only if $\oblv_{\on{abs}\to\on{rel},X_2}(\CM)$
is compact.

\medskip

Now, recall (see, e.g., \cite[Sect. 6.3.4]{DrGa1}) that for $\CF\in \IndCoh(X_1)$, the object
$$f^!(\CF)\in \IndCoh(X_2)$$ has a natural structure of object of $\Dmod_{\on{rel}}(X_2)$,
i.e., is the image under $\oblv_{\on{abs}\to\on{rel},X_2}$ of the same-named object of $\Dmod_{\on{rel}}(X_2)$.
Furthermore, by \cite[Lemma 6.3.15]{DrGa1}
$$f^!(\ind_{X_1}(\CF))\simeq \ind_{\on{rel}\to\on{abs},X_2}(f^!(\CF)).$$

\medskip

Combining this with the projection formula of 
\cite[Proposition 6.3.12(b')]{DrGa1}, for $\CM\in \Dmod(X_2)$ we obtain a canonical isomorphism 
$$\CM\sotimes f^!(\ind_{X_1}(\CF))\simeq \ind_{\on{rel}\to\on{abs},X_2}(\oblv_{\on{abs}\to\on{rel},X_2}(\CM)\sotimes f^!(\CF)).$$

\medskip

Hence, we obtain that $\CM$ is ULA if and only if the object
$$\ind_{\on{rel}\to\on{abs},X_2}(\oblv_{\on{abs}\to\on{rel},X_2}(\CM)\sotimes f^!(\omega_{X_1}))\simeq
\ind_{\on{rel}\to\on{abs},X_2}(\oblv_{\on{abs}\to\on{rel},X_2}(\CM)) \in \Dmod(X_2)$$
is compact. 

\medskip

However, it is
easy to see that an object $\CM'\in \Dmod_{\on{rel}}(X_2)$ is compact if and only if 
$$\ind_{\on{rel}\to\on{abs},X_2}(\CM')\in \Dmod(X_2)$$
is compact. 

\end{proof}

\section{Proof of the subquotient theorem}  \label{s:proof of subq}

The goal of this section is to prove \thmref{t:subquotient}. The results of this section will not be used elsewhere
in the paper. 

\ssec{The tensor product t-structure}   \label{ss:ten prod t-structure}

Let $\bC_1$ and $\bC_2$ be two DG categories, each endowed with a t-structure.
Consider the DG category $\bC_1\otimes \bC_2$. It inherits a t-structure where we set
$(\bC_1\otimes \bC_2)^{> 0}$ to be the full subcategory spanned by objects $\bc$ that
satisfy
$$\Maps(\bc_1\otimes \bc_2,\bc)=0, \, \forall \bc_1\in \bC_1^{\leq 0},\bc_2\in \bC_2^{\leq 0}.$$

Equivalently, the subcategory $(\bC_1\otimes \bC_2)^{\leq 0}$ is generated under colimits
by objects of the form $\bc_1\otimes \bc_2$ with $\bc_i\in \bC_i^{\leq 0}$.

\sssec{}  

Let us recall that a t-structure on a DG category $\bC$
is said to be \emph{compactly generated} if the category $\bC^{\leq 0}$ is generated under colimits
by the subcategory $\bC^{\leq 0}\cap \bC^c$. Equivalently, if 
$$\bc\in \bC^{>0}\, \Leftrightarrow \, \Maps(\bc',\bc)=0, \, \forall \,\bc'\in \bC^{\leq 0}\cap \bC^c.$$
E.g., this is the case for the standard t-structures on $\QCoh(X)$,
$\IndCoh(X)$ and $\Dmod(X)$ for a scheme $X$. 

\sssec{}

Let $\bC_1$ and $\bC_2$ be DG categories, both equipped with t-structures. 
Note that, by construction, if the t-structures on $\bC_i$ are compactly generated, the same
will be true for one on $\bC_1\otimes \bC_2$.

\medskip

We will use the following assertion:

\begin{lem} \label{l:ten prod t-structure} 
Let $\bC_1,\bC_2,\wt\bC_2$ be DG categories, each endowed with a t-structure, and let
$\sF:\bC_2\to \wt\bC_2$ be a continuous functor. Consider the functor
$$(\on{Id}_{\bC_1}\otimes \sF):\bC_1\otimes \bC_2\to \bC_1\otimes \wt\bC_2.$$

\smallskip

\noindent{\em(i)} If the functor $\sF$ is right t-exact, then so is $\on{Id}_{\bC_1}\otimes \sF$.

\smallskip

\noindent{\em(ii)} If the functor $\sF$ is left t-exact, and the t-structure on $\bC_1$ is compactly 
generated, then the functor $\on{Id}_{\bC_1}\otimes \sF$ is also left t-exact.
\end{lem}

\begin{rem}
We do not know whether in point (ii) one can get rid of the assumption that the t-structure on $\bC_1$
be compactly generated.
\end{rem}

\begin{proof}

Point (i) is tautological. For point (ii), by the assumption on $\bC_1$, it suffices to show that
for $\bc\in (\bC_1\otimes \bC_2)^{> 0}$ and for $\bc_1\in \bC_1^{\leq 0}\cap \bC_1^c$ and
$\wt\bc_2\in \wt\bC_2^{\leq 0}$, the object 
\begin{equation} \label{e:Hom on prod}
\CMaps_{\bC_1\otimes \wt\bC_2}\Bigl(\bc_1\otimes \wt\bc_2,(\on{Id}_{\bC_1}\otimes \sF)(\bc)\Bigr)\in \Vect
\end{equation}
belongs to $\Vect^{>0}$.

\medskip

Note that for a pair of DG categories $\bC_1$ and $\wt\bC_2$, and objects $\bc_1\in \bC^c_1$, 
$\wt\bc_2\in \wt\bC_2$ and $\bc'\in \bC_1\otimes \wt\bC_2$,
we have a canonical isomorphism
$$\CMaps_{\bC_1\otimes \wt\bC_2}(\bc_1\otimes \wt\bc_2,\bc')\simeq 
\CMaps_{\wt\bC_2}\left(\wt\bc_2,(\on{ev}_{\bC_1}\otimes \on{Id}_{\wt\bC_2})(\bc_1^\vee \otimes \bc')\right),$$
where $\bc_1^\vee$ is the object of $(\bC_1^\vee)^c\simeq (\bC_1^c)^{\on{op}}$ corresponding to $\bc_1\in \bC_1^c$ and 
where 
$$\on{ev}_{\bC_1}:\bC_1^\vee\otimes \bC_1\to \Vect$$
is the canonical evaluation functor. 

\medskip

Hence, we can rewrite \eqref{e:Hom on prod} as
\begin{equation} \label{e:Hom on prod again}
\CMaps_{\wt\bC_2}\left(\wt\bc_2, (\on{ev}_{\bC_1}\otimes \on{Id}_{\wt\bC_2})\left(\bc_1^\vee\otimes 
(\on{Id}_{\bC_1}\otimes \sF)(\bc)\right)\right).
\end{equation}

\medskip

We have
$$(\on{ev}_{\bC_1}\otimes \on{Id}_{\wt\bC_2})\left(\bc_1^\vee\otimes (\on{Id}_{\bC_1}\otimes \sF)(\bc)\right)\simeq
\sF\circ (\on{ev}_{\bC_1}\otimes \on{Id}_{\bC_2})(\bc_1^\vee\otimes \bc).$$

Now, since $\bc_1\in \bC_1^{\leq 0}$ and $\bc\in (\bC_1\otimes \bC_2)^{> 0}$, we have
$$(\on{ev}_{\bC_1}\otimes \on{Id}_{\bC_2})(\bc_1^\vee\otimes \bc)\in \bC_2^{>0}.$$

Hence, $\sF\circ (\on{ev}_{\bC_1}\otimes \on{Id}_{\bC_2})(\bc_1^\vee\otimes \bc)\in \wt\bC_2^{> 0}$, since
$\sF$ is left t-exact. Hence, the expression in \eqref{e:Hom on prod again} belongs to $\Vect^{>0}$ since
$\wt\bc_2\in \wt\bC_2^{\leq 0}$.

\end{proof}

\ssec{The t-structure on $(\CO,\on{D})$-bimodules}

\sssec{}

For a pair of schemes $X_1$ and $X_2$ consider the DG category  
$$\IndCoh(X_1)\otimes \Dmod(X_2),$$
endowed with the t-structure, induced by the t-structures on $\IndCoh(X_1)$ and $\Dmod(X_2)$, respectively. 

\medskip

The goal of this subsection is to prove the following assertion:

\begin{prop}  \label{p:subq compact}
Let $\CK$ be a compact object in $\IndCoh(X_1)\otimes \Dmod(X_2)$. Then any subquotient
of any of its cohomologies (with respect to the above t-structure) is compact. 
\end{prop}

The rest of this subsection is devoted to the proof of this proposition; so, the reader, who is willing
to take the assertion of \propref{p:subq compact} on faith, can skip it. 

\sssec{}

Consider the DG category 
$$\QCoh(X_1)\otimes \Dmod(X_2),$$ 
endowed with the t-structure induced by the t-structures on  $\QCoh(X_1)$ and $\Dmod(X_2)$,
respectively. 

\medskip

By \lemref{l:ten prod t-structure}, the functor 
$$\Psi_{X_1}\otimes \on{Id}_{\Dmod(X_2)}:\IndCoh(X_1)\otimes \Dmod(X_2)\to \QCoh(X_1)\otimes \Dmod(X_2)$$
is t-exact. 

\begin{lem}  \label{l:pre-ren}
The functor $\Psi_{X_1}\otimes \on{Id}_{\Dmod(X_2)}$ induces an equivalence 
$$\left(\IndCoh(X_1)\otimes \Dmod(X_2)\right)^{\geq 0}\to \left(\QCoh(X_1)\otimes \Dmod(X_2)\right)^{\geq 0}.$$
\end{lem}

\begin{proof}

The functor in question admits a left adjoint, which is also a right inverse, given by 
$$\CM\mapsto \tau^{\geq 0}(\Xi_{X_1}\otimes \on{Id}_{\Dmod(X_2)}(\CM)).$$

Hence, it remains to check that $\Psi_{X_1}\otimes \on{Id}_{\Dmod(X_2)}$ is conservative
when restricted to the subcategory $\left(\IndCoh(X_1)\otimes \Dmod(X_2)\right)^{\geq 0}$. Note that
$$\on{ker}(\Psi_{X_1}\otimes \on{Id}_{\Dmod(X_2)})=\on{ker}(\Psi_{X_1})\otimes \Dmod(X_2).$$

So, we need to show that the essential image of the fully faithful embedding
$$\on{ker}(\Psi_{X_1})\otimes \Dmod(X_2)\hookrightarrow \IndCoh(X_1)\otimes \Dmod(X_2)$$
has a zero intersection with $\left(\IndCoh(X_1)\otimes \Dmod(X_2)\right)^{\geq 0}$. 

\medskip

Note (see \cite[Sect. 1.2.7]{IndCoh})
that the essential image of $\on{ker}(\Psi_{X_1})$ in $\IndCoh(X_1)$ belongs to $\IndCoh(X_1)^{<0}$
(in fact, to $\IndCoh(X_1)^{<-n}$ for any $n$).  Hence, the desired assertion follows from \lemref{l:ten prod t-structure}(i).

\end{proof}

\begin{cor} The functor $\Psi_{X_1}\otimes \on{Id}_{\Dmod(X_2)}$ has the following properties:

\smallskip 

\noindent{\em(i)} It is fully faithful when restricted
to $(\IndCoh_{X_1}\otimes \on{Id}_{\Dmod(X_2)})^c$.

\smallskip 

\noindent{\em(ii)} It induces an equivalence
$$\left(\IndCoh(X_1)\otimes \Dmod(X_2)\right)^\heartsuit\to \left(\QCoh(X_1)\otimes \Dmod(X_2)\right)^\heartsuit.$$
\end{cor}

\sssec{}

We claim that the abelian category 
$$\left(\QCoh(X_1)\otimes \Dmod(X_2)\right)^\heartsuit$$
is the usual category of quasi-coherent sheaves of $(\CO_{X_1},\on{D}_{X_2})$-modules on $X_1\times X_2$.

\medskip

Indeed, it is easy to see that the assertion is local, so we can assume that $X_1$ and $X_2$ are affine. In this
case $\left(\IndCoh(X_1)\otimes \Dmod(X_2)\right)^\heartsuit$ admits a projective generator, namely 
$\tau^{\geq 0}(\CO_{X_1})\boxtimes \on{D}_{X_2}$, where $\tau^{\geq 0}$ is the truncation functor.

\sssec{}

Let $$\left(\QCoh(X_1)\otimes \Dmod(X_2)\right)^{\on{f.g.}}\subset \QCoh(X_1)\otimes \Dmod(X_2)$$ be the full
subcategory spanned by cohomologically bounded objects with finitely generated cohomologies. 
As in 
\cite[Proposition 1.2.4]{IndCoh}
one shows that the category $$\left(\QCoh(X_1)\otimes \Dmod(X_2)\right)^{\on{f.g.}}$$ has a unique t-structure such
that the functor
\begin{equation} \label{e:ren}
\Ind\left(\left(\QCoh(X_1)\otimes \Dmod(X_2)\right)^{\on{f.g.}}\right)\to \QCoh(X_1)\otimes \Dmod(X_2),
\end{equation}
obtained by ind-extending the tautological embedding, is t-exact and induces an equivalence
\begin{equation} \label{e:renorm}
\Ind\left(\left(\QCoh(X_1)\otimes \Dmod(X_2)\right)^{\on{f.g.}}\right)^{\geq 0}\simeq 
\left(\QCoh(X_1)\otimes \Dmod(X_2)\right)^{\geq 0}.
\end{equation}

\medskip

The Noetherianness of the sheaf of rings $\CO_{X_1}\otimes \on{D}_{X_2}$ implies that the functor
$\Psi_{X_1}\otimes \on{Id}_{\Dmod(X_2)}$ sends the compact generators of $\IndCoh(X_1)\otimes \Dmod(X_2)$
to $\left(\QCoh(X_1)\otimes \Dmod(X_2)\right)^{\on{f.g.}}$. Hence, we obtain that 
the functor $\Psi_{X_1}\otimes \on{Id}_{\Dmod(X_2)}$ factors as a composition
of a canonically defined functor
\begin{equation} \label{e:to ren}
\IndCoh(X_1)\otimes \Dmod(X_2)\to \on{Ind}(\left(\QCoh(X_1)\otimes \Dmod(X_2)\right)^{\on{f.g.}}),
\end{equation}
followed by \eqref{e:ren}. 

\begin{lem} \label{l:comp ren}
The functor \eqref{e:to ren} is an equivalence and is t-exact. 
\end{lem}

\begin{proof}

The functor  \eqref{e:to ren} is right t-exact by construction. We construct a functor
\begin{equation} \label{e:from ren}
\on{Ind}(\left(\QCoh(X_1)\otimes \Dmod(X_2)\right)^{\on{f.g.}})\to \IndCoh(X_1)\otimes \Dmod(X_2),
\end{equation}
right adjoint to \eqref{e:to ren} by ind-extending
\begin{multline*} 
\left(\QCoh(X_1)\otimes \Dmod(X_2)\right)^{\on{f.g.}}\hookrightarrow \left(\QCoh(X_1)\otimes \Dmod(X_2)\right)^+\simeq \\
\simeq \left(\IndCoh(X_1)\otimes \Dmod(X_2)\right)^+,
\end{multline*}
where the last equivalence is given by \lemref{l:pre-ren}. 

\medskip

Being the right adjoint to a right t-exact functor, the functor \eqref{e:from ren} is left t-exact. Consider the
composition
\begin{multline*} 
\on{Ind}(\left(\QCoh(X_1)\otimes \Dmod(X_2)\right)^{\on{f.g.}})^+\overset{\text{\eqref{e:from ren}}}\longrightarrow \\
\to \left(\IndCoh(X_1)\otimes \Dmod(X_2)\right)^+\to \left(\QCoh(X_1)\otimes \Dmod(X_2)\right)^+.
\end{multline*}

By \lemref{l:pre-ren} and \eqref{e:renorm}, we obtain that the functor \eqref{e:from ren} is t-exact and induces an equivalence
$$\on{Ind}(\left(\QCoh(X_1)\otimes \Dmod(X_2)\right)^{\on{f.g.}})^+\to 
\left(\IndCoh(X_1)\otimes \Dmod(X_2)\right)^+.$$

Since the compact objects of both 
$$\on{Ind}(\left(\QCoh(X_1)\otimes \Dmod(X_2)\right)^{\on{f.g.}}) \text{ and }
\IndCoh(X_1)\otimes \Dmod(X_2)$$
are contained in their eventually coconnective parts,
we obtain that \eqref{e:from ren} is an equivalence.

\end{proof}

\sssec{Proof of \propref{p:subq compact}}

Follows from the fact that the sheaf of rings $\CO_{X_1}\otimes \on{D}_{X_2}$ is Noetherian, combined with \lemref{l:comp ren}.

\qed

\ssec{Proof of \thmref{t:subquotient}}  

\sssec{Step 1} Set
$$\CK:=\oblv_{X_1}\otimes \on{Id}_{\Dmod(X_2)}(\CQ)\in \IndCoh(X_1)\otimes \Dmod(X_2).$$

The functor $\oblv_{X_1}$ is t-exact because $X_1$ is smooth. Hence, the
functor $\oblv_{X_1}\otimes \on{Id}_{\Dmod(X_2)}$ is also t-exact by \lemref{l:ten prod t-structure}.

\medskip

Therefore, if 
$\CQ'$ is a subquotient of the $n$-th cohomology of $\CQ$, we obtain that 
$$\CK':=\oblv_{X_1}\otimes \on{Id}_{\Dmod(X_2)}(\CQ')\in (\IndCoh(X_1)\otimes \Dmod(X_2))^\heartsuit$$
is a subquotient of the $n$-th cohomology of $\CK$. 

\sssec{Step 2}

Choose an affine open embedding $X_1\overset{j}\hookrightarrow \ol{X}_1$, where $\ol{X}_1$ is projective, but
not necessarily smooth (for aesthetic reasons we do not want to use desingularization; the latter allows to choose $\ol{X}_1$
smooth as well).  

\medskip

Set $\ol\CQ:=(j\times \on{id}_{X_2})_\bullet(\CQ)$, and
$$\ol\CK:=\oblv_{\ol{X}_1}\otimes \on{Id}_{\Dmod(X_2)}(\ol\CQ)\simeq
(j^{\IndCoh}_*\otimes \on{Id}_{\Dmod(X_2)})(\CK).$$

\medskip

Since $j$ is affine, the functor $j^{\IndCoh}_*:\IndCoh(X_1)\to \IndCoh(\ol{X}_1)$ is t-exact. Hence, by 
\lemref{l:ten prod t-structure}, so is the functor
$$(j^{\IndCoh}_*\otimes \on{Id}_{\Dmod(X_2)}):
\IndCoh(X_1)\otimes \Dmod(X_2)\to \IndCoh(\ol{X}_1)\otimes \Dmod(X_2).$$

\medskip

Hence, if $\CK'$ is a subquotient of the $n$-th cohomology of $\CK$, we obtain that
$$\ol\CK':=(j^{\IndCoh}_*\otimes \on{Id}_{\Dmod(X_2)})(\CK')\in (\IndCoh(\ol{X}_1)\otimes \Dmod(X_2))^\heartsuit$$
is a subquotient of the $n$-th cohomology of $\ol\CK$.

\medskip

By \propref{p:open embedding}, the object $\ol\CQ$ is compact. Hence, $\ol\CK$ is compact by
\propref{p:O criter}. By \propref{p:subq compact}, we obtain that the object $\ol\CK'$ is compact as well. 

\sssec{Step 3}

We have have the following assertion, proved in \secref{sss:proof of estimate} below:

\begin{lem}  \label{l:estimate}
Let $Y_1$ be projective with ample line bundle $\CL$. Then for 
$$\CT\in \left(\IndCoh(Y_1)\otimes \Dmod(Y_2)\right)^c\cap \left(\IndCoh(Y_1)\otimes \Dmod(Y_2)\right)^\heartsuit,$$
there exists an integer $i_0$ such that for all $i\geq i_0$, the non-zero cohomologies of
$$\left((p_{Y_1})^{\IndCoh}_*\otimes \on{Id}_{\Dmod(Y_2)}\right)(\CL^{i}\underset{\CO_{Y_1}}\otimes \CT)$$
vanish.
\end{lem} 

\medskip

Let $\ol\CK$ and $\ol\CK'$ be as in Step 2. By \lemref{l:estimate}, we obtain that there exists an integer $i_0$ such that
for $i\geq i_0$, the object 
$$\left(\Gamma_{\ol{X}_1}\otimes \on{Id}_{\Dmod(X_2)}\right)(\CL^{i}\underset{\CO_{\ol{X}_1}}\otimes \ol\CK')$$
is acyclic off cohomological degree $n$, and appears as a subquotient of the $n$-th cohomology of
$$\left(\Gamma_{\ol{X}_1}\otimes \on{Id}_{\Dmod(X_2)}\right)(\CL^{i}\underset{\CO_{\ol{X}_1}}\otimes \ol\CK).$$

\medskip

Hence, we obtain that the assertion of the theorem follows from \corref{c:criter with line bundle}(ii), combined with the fact
for a scheme $X$, if an object $\CM\in \Dmod(X)$ is compact,
then the same is true for any subquotient of any cohomology of $\CM$.

\qed

\sssec{Proof of \lemref{l:estimate}}  \label{sss:proof of estimate}

The functor 
$$(p_{Y_1})_*^{\IndCoh}\otimes \on{Id}_{\Dmod(Y_2)}:\IndCoh(Y_1)\otimes \Dmod(Y_2)\to \Dmod(Y_2)$$
is left t-exact by \lemref{l:ten prod t-structure}(ii).

\medskip

Hence, for any 
$$\CT\in \left(\IndCoh(Y_1)\otimes \Dmod(Y_2)\right)^\heartsuit$$
the object 
\begin{equation} \label{e:estimate}
\left((p_{Y_1})_*^{\IndCoh}\otimes \on{Id}_{\Dmod(Y_2)}\right)(\CK)\in \Dmod(Y_2)
\end{equation}
lives in $\Dmod(Y_2)^{\geq 0}$. 

\medskip

Now, any $\CT$ as above admits a \emph{left} resolution $\CT_\bullet$ whose terms $\CT_n$ are of the form
$$\CF_n\boxtimes \CM_n,\quad \CF_n\in \Coh(Y_1)^\heartsuit,\,\, \CM_n\in \Dmod(Y_2)^c\cap \Dmod(Y_2)^\heartsuit.$$

\medskip

Note that the functor $(p_{Y_1})^{\IndCoh}_*\otimes \on{Id}_{\Dmod(Y_2)}$ has cohomological amplitude bounded
\emph{on the right} by $\dim(Y_1)$, because this is trie for $(p_{Y_1})^{\IndCoh}_*$. Hence, it is enough to show that
for $n=0,...,\dim(Y_1)$ and $i\gg 0$, the higher cohomologies of
\begin{equation} \label{e:to vanish}
\left((p_{Y_1})^{\IndCoh}_*\otimes \on{Id}_{\Dmod(Y_2)}\right)(\CL^{i}\underset{\CO_{Y_1}}\otimes \CT_n)
\end{equation} 
vanish. However, the expression in \eqref{e:to vanish} is isomorphic to
$$\Gamma(Y_1,\CL^{\otimes i}\otimes \CF_n)\otimes \CM_n,$$
and the assertion follows from the correponding fact for $\CF_n$. 

\qed

\section{Proof of the main theorem for schemes, and generalizations}  \label{s:delo}

In this section we will prove \thmref{t:schemes} by establishing a general result along the same lines 
for arbitrary DG categories. 

\ssec{Duality in a compactly generated category}  \label{ss:abstract Verdier}

\sssec{}

Let $\bC$ be a compactly generated category. Recall that we have a natural equivalence
$$(\bC^c)^{\on{op}} \simeq (\bC^\vee)^c, \quad \bc\mapsto \bc^\vee.$$

We shall now extend the above assignment to a (non-continuous) functor 
\begin{equation} \label{e:abs duality}
\bC^{\on{op}}\to \bC^\vee.
\end{equation}
Namely, for $\bc\in \bC$ we let $\bc^\vee$ be the object of $\bC^\vee$ characterized by the
property that
$$\Hom_{\bC^\vee}(\bc^\vee_1,\bc^\vee):=\Hom_\bC(\bc,\bc_1) \text { for }\bc_1\in \bC^c.$$

\sssec{}

Explicitly, if $\bc=\underset{i}{colim}\, \bc_i$ with $\bc_i\in \bC^c$,
then 
\begin{equation} \label{e:duality as limit}
\bc^\vee=\underset{i}{lim}\, \bc^\vee_i.
\end{equation}

By construction, the assignment $\bc\mapsto \bc^\vee$ 
sends colimits to limits. In general, it is very ill-behaved.

\sssec{}

From \eqref{e:duality as limit} we obtain:

\begin{lem}  \label{l:RKE}
The functor \eqref{e:abs duality} is the \emph{right Kan extension} of its restriction to $(\bC^c)^{\on{op}}$.
\end{lem}

\begin{proof}
This is the property of any functor from $\bC^{\on{op}}$ that commutes with limits.
\end{proof}

\sssec{}

Let $\bc_1$ and $\bc_2$ be two objects of $\bC$. We claim that there is a canonical map
\begin{equation} \label{e:duality and pairing}
\on{ev}_\bC(\bc^\vee_1\otimes \bc_2)\to \CMaps_{\bC}(\bc_1,\bc_2).
\end{equation}

\medskip

Indeed, for $\bc_2$ compact, the map \eqref{e:duality and pairing} is the isomorphism resulting
from the tautological isomorphism
$$\on{ev}_\bC(-\otimes \bc_2)\simeq \CMaps_{\bC^\vee}(\bc_2^\vee,-).$$

\medskip

In general, the map \eqref{e:duality and pairing} results from the fact that the left-hand side is
continuous as a functor of $\bc_2$, and hence is the left Kan extension from its restriction to $\bC^c$.

\sssec{}   \label{sss:d-compact}

Note that for $\bc\in \bC$ we have a canonical map
\begin{equation} \label{duality map}
\bc \to (\bc^\vee)^\vee.
\end{equation}

We shall say that $\bc$ is \emph{reflexive} if the map
\eqref{duality map} is an isomorphism.  

\medskip

It is clear that every compact object is reflexive. But the converse is obviously false.

\sssec{Interaction with functors}

Let $\sF:\bC_1\to \bC_2$ be a functor that sends compact objects to compact ones.
Consider the conjugate functor $\sF^{\on{op}}:\bC^\vee_1\to \bC^\vee_2$, see \secref{sss:conj functors}.

\medskip

We claim that for $\bc_1\in \bC_1$ we have a canonical map
\begin{equation} \label{e:functor and duality}
\sF^{\on{op}}(\bc^\vee_1)\to (\sF(\bc_1))^\vee,
\end{equation}
that extends the tautological isomorphism for $\bc_1\in \bC_1^c$. 

\medskip

The natural transformation \eqref{e:functor and duality} follows by adjunction from the fact that the functor 
$$\bc_1\mapsto (\sF(\bc_1))^\vee,\quad (\bC_1)^{\on{op}}\to \bC_2$$
is the right Kan extension of its restriction to $(\bC_1^c)^{\on{op}}$ (i.e., takes colimits in $\bC_1$
to limits in $\bC_2^\vee$). 









\ssec{A general framework for \thmref{t:schemes}}   \label{ss:the functor}

\sssec{}

Let $\bC$ be a compactly generated DG category. Recall that $\bu_{\bC}\in \bC\otimes \bC^\vee$
denotes the object that defines the identity functor. 

\medskip

We consider the object 
$$(\bu_\bC)^\vee\in (\bC\otimes \bC^\vee)^\vee=\bC^\vee\otimes \bC.$$

\medskip

We let $\psId_\bC$ be the functor $\bC\to \bC$ defined by the kernel $(\bu_\bC)^\vee$. I.e., in the
notations of \secref{sss:functors and kernels},
$$\psId_\bC:=\sF_{\bC\to \bC,(\bu_\bC)^\vee}.$$

Note that by construction
\begin{equation} \label{e:dual of ps}
(\psId_\bC)^\vee\simeq \psId_{\bC^\vee},\quad \bC^\vee\to \bC^\vee.
\end{equation}

\sssec{}

Let $\bC_1$ and $\bC_2$ be two compactly generated categories, and let $\sF:\bC_1\to \bC_2$ be a functor 
between them that preserves compactness. 

\medskip

Let $\CQ\in \bC^\vee_1\otimes \bC_2$ be
kernel of $\sF$ i.e.,
$$\CQ=(\on{Id}_{\bC_1^\vee}\otimes \sF)(\bu_{\bC_1}).$$

Consider the functor
$$(\on{Id}_{\bC^\vee_1}\otimes \sF):\bC_1^\vee\otimes \bC_1\to \bC_1^\vee\otimes \bC_2.$$
By \corref{c:preserve compactness tensor}, this functor still preserves compactness. Applying 
\eqref{e:functor and duality} to this functor and the object $\bu_{\bC_1}\in \bC_1^\vee\otimes \bC_1$,
we obtain a map 
\begin{equation} \label{e:general}
(\on{Id}_{\bC_1}\otimes \sF^{\on{op}})((\bu_{\bC_1})^\vee) \to ((\on{Id}_{\bC^\vee_1}\otimes \sF)(\bu_{\bC_1}))^\vee,
\end{equation}
where we note that 
\begin{equation} \label{e:general ident}
((\on{Id}_{\bC^\vee_1}\otimes \sF)(\bu_{\bC_1}))^\vee\simeq \CQ^\vee. 
\end{equation}

\medskip

\begin{thm} \label{t:general}
Assume that the map \eqref{e:general} is an isomorphism. Then the composed functor
$$\bC_2\overset{\sF^R}\longrightarrow \bC_1\overset{\psId_{\bC_1}}\longrightarrow \bC_1$$
is given by the kernel 
$$\CQ^\vee \in \bC_1\otimes \bC_2^\vee\simeq \bC^\vee_2\otimes \bC_1.$$
\end{thm}

\begin{proof}
This is a tautology: 

\medskip

The kernel of the composition 
$$\bC_2\overset{\sF^R}\longrightarrow \bC_1\overset{\psId_{\bC_1}}\longrightarrow \bC_1,$$
viewed as an object of $\bC_2^\vee\otimes \bC_1$, is obtained from the kernel of $\psId_{\bC_1}$,
viewed as an object of $\bC_1^\vee\otimes \bC_1$, by applying the functor
$$(\sF^R)^\vee\otimes \on{Id}_{\bC_1}:\bC_1^\vee\otimes \bC_1\to \bC_2^\vee\otimes \bC_1.$$

By \lemref{l:conjugate}, the latter is the same as
$$(\sF^{\on{op}}\otimes \on{Id}_{\bC_1})((\bu_{\bC_1})^\vee),$$
which identifies with $\CQ$ by \eqref{e:general ident} and the assumption of the theorem. 

\end{proof}

\ssec{The smooth case}  \label{ss:diag compact}

\sssec{}

Recall that a DG category $\bC$ is called \emph{smooth} if the object 
$$\bu_\bC\in \bC\otimes \bC^\vee$$
is compact. 

\begin{rem}
The terminology ``smooth" originates in the fact that for a separated scheme $X$, the DG category $\QCoh(X)$
is smooth if and only if $X$ is a smooth classical scheme (see \secref{ss:recall IndCoh} for our conventions
regarding schemes).
\end{rem}

\sssec{}  \label{sss:diag compact}

Note that the assumption of \thmref{t:general} is trivially satisfied when $\bC_1$ is \emph{smooth}.
Indeed, the map \eqref{e:functor and duality} is by definition an isomorphism for $\bc_1$ compact. 

\sssec{Proof of \thmref{t:schemes}}  \label{sss:proof of schemes}

This is follows immediately from \thmref{t:general} and \secref{sss:diag compact}, using the fact that
$$\bu_{\Dmod(X)}=(\Delta_X)_\bullet(\omega_X)\in \Dmod(X\times X),$$
being a bounded holonomic complex, is compact.

\qed

\sssec{The natural transformation}  \label{sss:adj map abstract} 

Let us continue to assume that $\bC_1$ is smooth, and let us be in the situation of 
\thmref{t:general}. 

\medskip

The (iso)morphism of functors
$$\psId_{\bC_1}\circ (\sF_{\bC_1\to \bC_2,\CQ})^R\to \sF_{\bC_2\to \bC_1,\CQ^\vee}$$
gives rise to (and is determined by) the natural transformation 
\begin{equation} \label{e:nat trans abs}
\psId_{\bC_1}\to \sF_{\bC_2\to \bC_1,\CQ^\vee}\circ \sF_{\bC_1\to \bC_2,\CQ}. 
\end{equation}

Let us write down the natural transformation \eqref{e:nat trans abs} explicitly:

\medskip

The datum of a map
$$\psId_{\bC_1}\to \sF_{\bC_2\to \bC_1,\CQ^\vee}\circ \sF_{\bC_1\to \bC_2,\CQ}$$
is equivalent to that of a map between the corresponding kernels, i.e., 
$$(\bu_{\bC_1})^\vee \to (\on{Id}_{\bC_1^\vee}\otimes \on{ev}_{\bC_2}\otimes \on{Id}_{\bC_1})(\CQ\otimes \CQ^\vee),$$
and the latter is the same as a datum of a vector in
$$\on{ev}_{\bC^\vee_1}\left((\on{Id}_{\bC_1^\vee}\otimes \on{ev}_{\bC_2}\otimes \on{Id}_{\bC_1})(\CQ\otimes \CQ^\vee)\right)
\simeq \on{ev}_{\bC^\vee_1\otimes \bC_2}(\CQ\otimes \CQ^\vee).$$

Now vector corresponding to \eqref{e:nat trans abs} is the canonical vector in
$$\on{ev}_\bC(\bc\otimes \bc^\vee)\simeq \CMaps_\bC(\bc,\bc),$$
defined for any DG category $\bC$ and $\bc\in \bC^c$ (where we take $\bC=\bC_1^\vee\otimes \bC_2$ and $\bc=\CQ$). 

\ssec{Gorenstein categories}  \label{ss:Gorenstein}

\sssec{}

Following Drinfeld, we shall say that a compactly generated category $\bC$ is \emph{Gorenstein}
if the functor
$$\psId_\bC:\bC\to \bC$$
is an equivalence. 

\medskip

For example, the category $\Dmod(X)$ on a smooth separated scheme $X$ is Gorenstein.

\sssec{}

The origin of the name is explained by the following assertion. Recall (see \cite[Sect. 7.3.3]{IndCoh}) that a scheme
$X$ is said to be Gorenstein if $\omega_X$, regarded as an object of $\Coh(X)$, is invertible 
(i.e., a cohomologically shifted line bundle). 

\begin{prop} \label{p:Gore}
For a separated scheme $X$ the following assertions are equivalent:

\smallskip

\noindent{\em(a)} The scheme $X$ is Gorenstein;

\smallskip

\noindent{\em(b)} The category $\QCoh(X)$ is Gorenstein;

\smallskip

\noindent{\em(c)} The category $\IndCoh(X)$ is Gorenstein.

\end{prop}

\begin{proof}

First we note that for a separated scheme $X$ the object 
$$(\bu_{\QCoh(X)})^\vee\in \QCoh(X)\otimes\QCoh(X)\simeq \QCoh(X\times X)$$
identifies with
\begin{equation} \label{e:psidqcoh}
\uHom_{X\times X}((\Delta_X)_*(\CO_X),\CO_{X\times X})\simeq 
(\Delta_X)_*(\Delta_X^!(\CO_{X\times X})).
\end{equation}
and the object 
$$(\bu_{\IndCoh(X)})^\vee\in \IndCoh(X)\otimes\IndCoh(X)\simeq \IndCoh(X\times X)$$
identifies with
\begin{equation} \label{e:psidindcoh}
\BD_{X\times X}^{\on{Se}}((\Delta_X)_*^{\IndCoh}(\omega_X))\simeq
(\Delta_X)^{\IndCoh}_*(\CO_X).
\end{equation}

\medskip

Assume first that $X$ is Gorenstein, i.e., $\omega_X\simeq \CL$, where $\CL$ is a cohomologically
shifted line.  

\medskip

In this case $\Delta_X^!(\CO_{X\times X})$ identifies with $\omega_X\otimes \CL^{\otimes -2}\simeq \CL^{\otimes -1}$.
I.e., the functor $\psId_{\QCoh(X)}$ is given by tensor product by $\CL^{\otimes -1}$, and thus is an equivalence.

\medskip

Similarly, $\psId_{\IndCoh(X)}$ is also given by the action of $\CL^{\otimes -1}$, in the sense of the monoidal action
of $\QCoh(X)$ on $\IndCoh(X)$, and hence is also an equivalence.

\medskip

Vice versa, assume that $\psId_{\QCoh(X)}$ is an equivalence.  It suffices to show that for every $k$-point $i_x:\on{pt}\to X$, the object
$$i_x^*(\omega_X)\in \Vect$$
is invertible. By duality (in $\Vect$) it suffices to show that $(i_x^*(\omega_X))^\vee$ is invertible. However,
$$(i_x^*(\omega_X))^\vee=\CMaps_{\Vect}(i_x^*(\omega_X),k)\simeq \CMaps_{\Coh(X)}(\omega_X,(i_x)_*(k)),$$
which by Serre duality identifies with $(i_x)^!(\CO_X)$. 

\medskip

By \eqref{e:psidqcoh}, the assumption that $\psId_{\QCoh(X)}$ is an equivalence means that that the object
$$\Delta_X^!(\CO_{X\times X})\in \QCoh(X)$$ has the property that the functor of tensoring by it
is an equivalence. Hence, $\Delta_X^!(\CO_{X\times X})$
is a cohomologically shifted line bundle; denote it by $\CL'$. Hence, for $i_x$ as above, 
$$i_x^!(\CO_X)\otimes i_x^!(\CO_X)\simeq (i_x\times i_x)^!(\CO_{X\times X})\simeq i_x^!\circ \Delta_X^!(\CO_{X\times X})
\simeq i_x^!(\CL')\simeq i_x^!(\CO_X)\otimes i_x^*(\CL'),$$
from which it follows that $i_x^!(\CO_X)$ is invertible, as required. 

\medskip

Assume now that $\psId_{\IndCoh(X)}$ is an equivalence. By \eqref{e:psidindcoh}, 
this implies that $\CO_X$, regarded as object of $\IndCoh(X)$,
is invertible with respect to the $\sotimes$ symmetric monoidal structure on $\IndCoh(X)$. In particular, for every
$i_x:\on{pt}\to X$ as above, $(i_x)^!(\CO_X)$ is invertible in $\Vect$. By the above, this implies that $i_x^*(\omega_X)$
is invertible, as required.

\end{proof}

\begin{rem}
The following observation is due to A.~Arinkin: the same proof as above shows that the functor
$$\CE\mapsto \Upsilon(\CE):=\CE\otimes \omega_X$$
establishes an equivalence between $\QCoh(X)^{\on{perf}}$ and the category of \emph{dualizable}
objects of $\IndCoh$ with respect to the $\sotimes$ symmetric monoidal structure.
\end{rem}

\sssec{}

We shall now give a criterion for a compactly generated DG category $\bC$ to be Gorenstein.

\begin{prop}  \label{p:when Gor}
Suppose that the functors $\psId_\bC:\bC\to \bC$ and $\psId_{\bC^\vee}:\bC^\vee\to \bC^\vee$ both
satisfy the assumption of \thmref{t:general}.
Suppose also that $\bu\in \bC\otimes \bC^\vee$ is reflexive. Then $\bC$ is Gorenstein. 
\end{prop}

\begin{rem}
Note that \propref{p:when Gor} has the following flavor: \emph{certain finiteness properties of a functor
imply that this functor is an equivalence}.  
\end{rem}

\begin{proof}

We apply \thmref{t:general} to $\sF=\psId_\bC$. Combining with the assumtion that
$$((\bu_\bC)^\vee)^\vee\simeq \bu_\bC,$$
we obtain
\begin{equation} \label{e:left inverse}
\psId_\bC\circ \sF^R\simeq \on{Id}_\bC.
\end{equation}

I.e., we obtain that $\psId_\bC$ admits a right inverse. Passing to the dual functors in \eqref{e:left inverse} for $\bC^\vee$,
and using the fact that $(\psId_\bC)^\vee\simeq \psId_{\bC^\vee}$, we obtain that $\psId_\bC$ also has a left
inverse. Hence, it is an equivalence.

\end{proof} 

\section{Generalization to Artin stacks: quasi-compact case}  \label{s:Artin qc}

\ssec{QCA stacks: recollections}

\sssec{}

In this section all algebraic stacks will be assumed QCA. Recall (see \cite[Definition 1.1.8]{DrGa1}) that an algebraic
stack $\CX$ is said to be QCA if it is quasi-compact and the automorphism group of every field-valued
point is \emph{affine}.

\medskip

We recall (see \cite[Theorem 8.1.1]{DrGa1}) that for a QCA stack the category $\Dmod(X)$ is compactly generated.
Furthermore, by \cite[Corollary 8.3.4]{DrGa1}, for \emph{any} prestack $\CX'$, the operation of external tensor product defines 
an equivalence
$$\Dmod(\CX)\otimes \Dmod(\CX')\to \Dmod(\CX\times \CX').$$

\medskip

We let $\Dmod(\CX)_{\on{coh}}\subset \Dmod(\CX)$ be the full subcategory of \emph{coherent} D-modules.
We remind that an object of $\Dmod(\CX)$ is called coherent if its pullback to any scheme, mapping
smoothly to $\CX$, is compact (see \cite[Sect. 7.3.1]{DrGa1}). 

\medskip

We always have 
$$\Dmod(\CX)^c\subset \Dmod(\CX)_{\on{coh}},$$
and the containment is an equality if and only if $\CX$ is \emph{safe}, which means that the automorphism 
group of every field-valued point is such that its neutral connected component is unipotent (\cite[Corollary 10.2.7]{DrGa1}). For example, 
any Deligne-Mumford stack (and, in particular, any algebraic space) is safe. 

\medskip

The category $\Dmod(\CX)_{\on{coh}}$ carries a canonical Verdier duality anti-involution
$$\BD_\CX^{\on{Ve}}:(\Dmod(\CX)_{\on{coh}})^{\on{op}}\to \Dmod(\CX)_{\on{coh}}.$$

\sssec{}  \label{sss:Verdier stacks}

The basic property of the functor $\BD_\CX^{\on{Ve}}$ is that it preserves the subcategory $\Dmod(\CX)^c$,
thereby inducing an equivalence
$$\BD_\CX^{\on{Ve}}:(\Dmod(\CX)^c)^{\on{op}}\to \Dmod(\CX)^c$$
(see \cite[Corollary 8.4.2]{DrGa1}).

\medskip

Hence, it induces an equivalence $\Dmod(\CX)^\vee\to \Dmod(\CX)$ that we denote by 
$\bD_\CX^{\on{Ve}}$. The unit and counit corresponding to the identification $\bD_\CX^{\on{Ve}}$ are described below,
see \secref{sss:unit and counit stacks}. 

\sssec{}  \label{sss:ren}

For a morphism $f:\CX_1\to \CX_2$ we have the functor $f^!:\Dmod(\CX_2)\to \Dmod(\CX_1)$.
The usual de Rham direct image functor (defined as in \cite[Sect. 7.4.1]{DrGa1})
$$f_\bullet:\Dmod(\CX_1)\to \Dmod(\CX_2)$$
is in general non-continuous. 

\medskip

In fact, $f_\bullet$ is continuous if and only if $f$ is \emph{safe} (i.e., its geometric fibers are safe algebraic stack). 
E.g., any schematic or representable morphism is safe. 

\medskip

In \cite[Sect. 9.3]{DrGa1} another functor 
$$f_\blacktriangle:\Dmod(\CX_1)\to \Dmod(\CX_2)$$
is introduced, which is by definition the ind-extension of the restriction of the functor $f_\bullet$ to $\Dmod(\CX_1)^c$. 
I.e., $f_\blacktriangle$ is the unique continuous functor which equals $f_\bullet$ when restricted to $\Dmod(\CX_1)^c$. 

\medskip

We have a natural transformation
\begin{equation} \label{e:tr to bullet}
f_\blacktriangle\to f_\bullet,
\end{equation}
which is an isomorphism if $f$ is safe. For any $f$, \eqref{e:tr to bullet} is an isomorphism when evaluated
on compact objects. 

\sssec{}  \label{sss:unit and counit stacks}

We can now describe explicitly the unit and the counit of the identification
$\bD_{\CX}^{\on{Ve}}$. Namely,
the unit is given by the object
$$(\Delta_\CX)_\bullet(\omega_\CX)\in \Dmod(\CX\times \CX)\simeq \Dmod(\CX)\otimes \Dmod(\CX)
\simeq \Dmod(\CX)^\vee\otimes \Dmod(\CX),$$
where $(\Delta_\CX)_\bullet\simeq (\Delta_\CX)_\blacktriangle$ since the morphism $\Delta_\CX$ is representable
and hence safe. The object $\omega_\CX$ is, as in the case of scheme, $(p_\CX)^!(k)$, where 
$p_\CX:\CX\to \on{pt}$.

\medskip

The counit corresponds to the functor
$$\Dmod(\CX\times \CX)\overset{\Delta_\CX^!}\longrightarrow \Dmod(\CX)\overset{(p_\CX)_\blacktriangle}\longrightarrow \Vect.$$

\sssec{}

For a morphism $f:\CX_1\to \CX_2$, with respect to the equivalences $$\bD_{\CX_i}^{\on{Ve}}:\Dmod(\CX_i)^\vee\to \Dmod(\CX_i),$$ we have
$$(f_\blacktriangle)^\vee\simeq f^!.$$

\medskip

For a pair of QCA algebraic stacks, we have an equivalence of DG categories
$$\Dmod(\CX_1\times \CX_2)\simeq \on{Funct}_{\on{cont}}(\Dmod(\CX_1),\Dmod(\CX_2)),$$
$$\CQ\mapsto \sF_{\CX_1\to\CX_2,\CQ},\quad \sF_{\CX_1\to\CX_2,\CQ}(\CM)=
(\on{pr}_2)_\blacktriangle(\on{pr}_1^!(\CM)\sotimes \CQ).$$
and 
$$\sF\mapsto \CQ_\sF:=(\on{Id}_{\Dmod(\CX_1)}\otimes \sF)((\Delta_{\CX_1})_\bullet(\omega_{\CX_1})).$$

\sssec{}

We have the following assertion to be used in the sequel:

\begin{lem}  \hfill  \label{l:Verdier and check}
The restriction of the functor $\CF\mapsto \CF^\vee$
to $\Dmod(\CX)_{\on{coh}}$ identifies canonically with $\BD_\CX^{\on{Ve}}$.
\end{lem}

\begin{proof}

We need to show that for $\CF\in \Dmod(\CX)_{\on{coh}}$ and $\CF_1\in \Dmod(\CX)^c$ there exists
a canonical isomorphism
$$\CMaps(\CF_1,\BD_\CX^{\on{Ve}}(\CF))\simeq \CMaps(\CF,\CF_1^\vee).$$

However, this follows from the fact that $\CF_1^\vee=\BD_\CX^{\on{Ve}}(\CF_1)$ and the fact that
$\BD_\CX^{\on{Ve}}$ is an anti-self equivalence on $\Dmod(\CX)_{\on{coh}}$.

\end{proof}

\begin{cor}  \label{c:Verdier and check}
Every object of $\Dmod(\CX)_{\on{coh}}\subset \Dmod(\CX)$ is reflexive.
\end{cor}

We will need the following generalization of \corref{c:Verdier and check}:

\begin{prop} \label{p:Verdier and check}
Every object of $\Dmod(\CX)$ with coherent cohomologies is reflexive.
The functor $\CF\mapsto \CF^\vee$, restricted to the full subcategory of $\Dmod(\CX)$ spanned
by objects with coherent cohomologies, is an involutive anti-self equivalence and is of bounded
cohomological amplitude. 
\end{prop}

\begin{proof}

Note that the functor $\BD_\CX^{\on{Ve}}$ has a bounded cohomological amplitude,
say by $k$. We claim that for $\CF\in \Dmod(\CX)$ with coherent cohomologies we have
\begin{equation} \label{e:stimate on dual}
\tau^{\geq -n,\leq n}(\CF^\vee)\simeq \tau^{\geq -n,\leq n}\left(\BD_\CX^{\on{Ve}}(\tau^{\geq -n-k,\leq n+k}(\CF))\right), \quad \forall n\geq 0.
\end{equation}
This would prove the assertion of the proposition. 

\medskip

To prove \eqref{e:stimate on dual}, we note that since the t-structure on $\Dmod(\CX)$ is left and right complete,
for $\CF$ with coherent cohomologies there is a canonically defined object $\wt\CF\in \Dmod(\CX)$ such that
$$\tau^{\geq -n,\leq n}(\wt\CF)= \tau^{\geq -n,\leq n}\left(\BD_\CX^{\on{Ve}}(\tau^{\geq -n-k,\leq n+k}(\CF))\right), \quad \forall n\geq 0.$$

\medskip

We have to show that for $\CF_1\in \Dmod(\CX)^c$, there is a canonical isomorphism
$$\CMaps(\CF_1,\wt\CF)\simeq \CMaps(\CF,\BD_\CX^{\on{Ve}}(\CF_1)).$$

We shall do it separately in the cases $\CF\in \Dmod(\CX)^-$ and $\CF\in \Dmod(\CX)^+$ in such a way that the two
isomorphisms coincide for $\CF\in \Dmod(\CX)^-\cap \Dmod(\CX)^+$, i.e., when $\CF$ belongs $\Dmod(\CX)_{\on{coh}}$.

\medskip

For $\CF\in \Dmod(\CX)^-$, we have
$$\wt\CF\simeq  \underset{n}{colim}\, \BD_{\CX}^{\on{Ve}}(\tau^{\geq -n}(\CF)).$$
Hence, since $\CF_1$ is compact, 
$$\CMaps(\CF_1,\wt\CF)\simeq  \underset{n}{colim}\, \CMaps\left(\CF_1, \BD_{\CX}^{\on{Ve}}(\tau^{\geq -n}(\CF))\right)
\simeq \underset{n}{colim}\, \CMaps(\tau^{\geq -n}(\CF),\BD_{\CX}^{\on{Ve}}(\CF_1)).$$
However, since $\BD_{\CX}^{\on{Ve}}(\CF_1)$ is in $\Dmod(\CX)^+$, the last colimit stabilizes to 
$$\CMaps(\CF,\BD_{\CX}^{\on{Ve}}(\CF_1)),$$
as required.

\medskip

For $\CF\in \Dmod(\CX)^+$, we have 
$$\wt\CF\simeq  \underset{n}{lim}\, \BD_{\CX}^{\on{Ve}}(\tau^{\leq n}(\CF)).$$
Hence,
\begin{multline*}
\CMaps(\CF_1,\wt\CF)\simeq  \underset{n}{lim}\, \CMaps\left(\CF_1, \BD_{\CX}^{\on{Ve}}(\tau^{\leq n}(\CF))\right)\simeq \\ 
\simeq \underset{n}{lim}\, \CMaps(\tau^{\leq n}(\CF),\BD_{\CX}^{\on{Ve}}(\CF_1))\simeq
\CMaps(\CF,\BD_{\CX}^{\on{Ve}}(\CF_1)).
\end{multline*}

\end{proof}

In what follows, for $\CF\in \Dmod(\CX)$ with coherent cohomologies we shall denote
$$\BD_{\CX}^{\on{Ve}}(\CF):=\CF^\vee.$$

\ssec{Direct image with compact supports}

\sssec{}

Let $f:\CX_1\to \CX_2$ be a morphism between QCA stacks. Let 
$f_!$ denote the partially defined left adjoint to the functor $f^!:\Dmod(\CX_2)\to \Dmod(\CX_1)$.

\medskip

The following is a particular case of \lemref{l:conjugate bis}: 

\begin{lem} \label{l:!}
Let $\CF_1$ be an object of $\Dmod(\CX_1)^c$ for which the object
$$f_\blacktriangle(\BD_{\CX_1}^{\on{Ve}}(\CF_1))\in \Dmod(\CX_2)$$
belongs to $\Dmod(\CX_2)^c$.  Then the functor $f_!$ is defined on
$\CF_1$ and we have a canonical isomorphism
$$f_!(\CF_1)\simeq \BD^{\on{Ve}}_{\CX_2}\left(f_\blacktriangle(\BD_{\CX_1}^{\on{Ve}}(\CF_1))\right).$$
\end{lem}

We shall now prove its generalization where instead of compact objects we consider coherent ones:

\begin{prop} \label{p:!}
Let $\CF_1$ be an object of $\Dmod(\CX_1)_{\on{coh}}$ for which 
the object $$f_\bullet(\BD_{\CX_1}^{\on{Ve}}(\CF_1))\in \Dmod(\CX_2)$$
belongs to $\Dmod(\CX_2)_{\on{coh}}$.  Then $f_!(\CF_1)$ is well-defined and we have a canonical
isomorphism
$$f_!(\CF_1)\simeq \BD^{\on{Ve}}_{\CX_2}\left(f_\bullet(\BD_{\CX_1}^{\on{Ve}}(\CF_1))\right).$$
\end{prop}

\begin{rem}
Note that in \propref{p:!} we use the functor $f_\bullet$ rather than $f_\blacktriangle$. This does not
contradict \lemref{l:!} since the two functors coincide on compact objects. We also remind that the two functors
coincide when $f$ is safe (e.g., schematic or representable). 
\end{rem}

\begin{proof}

We need to establish a functorial isomorphism
$$\CMaps_{\Dmod(\CX_2)}\left(\BD^{\on{Ve}}_{\CX_2}\left(f_\bullet(\BD_{\CX_1}^{\on{Ve}}(\CF_1))\right),\CF_2\right)\simeq 
\CMaps_{\Dmod(\CX_1)}(\CF_1,f^!(\CF_2)), \quad \CF_2\in \Dmod(\CX_2).$$

Since both $\BD^{\on{Ve}}_{\CX_2}\left(f_\bullet(\BD_{\CX_1}^{\on{Ve}}(\CF_1))\right)$ and $\CF_1$
are coherent, and the functor $f^!$ has a bounded cohomological amplitude, we can assume that $\CF_2\in \Dmod(\CX_2)^-$.
Furthermore, since both $\Dmod(\CX_1)$ and $\Dmod(\CX_2)$ are left complete in their respective t-structures, and
we can moreover assume that $\CF_2\in \Dmod(\CX_2)^b$.

\medskip

Note that for a QCA stack $\CX$ and $\CF\in \Dmod(\CX)_{\on{coh}}$, the functor $\CMaps_{\Dmod(\CX)}(\CF,-)$
commutes with colimits \emph{taken in $\Dmod(\CX)^{\geq -n}$}, for any fixed $n$. 

\medskip

This allows to assume that $\CF_2\in \Dmod(\CX_2)_{\on{coh}}$. Hence, we need to establish an isomorphism
\begin{equation} \label{e:!2}
\CMaps_{\Dmod(\CX_2)}\left(\BD^{\on{Ve}}_{\CX_2}(\CF_2), f_\bullet(\BD_{\CX_1}^{\on{Ve}}(\CF_1))\right)
\simeq \CMaps_{\Dmod(\CX_1)}(\CF_1,f^!(\CF_2)), \quad  \Dmod(\CX_2)_{\on{coh}}.
\end{equation}

We claim that the latter isomorphism holds  for any $\CF_i\in \Dmod(\CX_i)_{\on{coh}}$, $i=1,2$. 

\medskip

Indeed, the definition of $f_\bullet$ (see \cite[Sect. 7.4.1]{DrGa1}) allows to reduce the proof of \eqref{e:!2} 
to the case when $\CX_1$ is a scheme. 
Thus, we can assume that $\CF_1\in \Dmod(\CX_1)^c$ and that $f$ is safe, so $f_\bullet=f_\blacktriangle$. In this case,
the right-hand side of \eqref{e:!2} identifies with
$$\on{ev}_{\Dmod(\CX_1)}\left(\BD_{\CX_1}^{\on{Ve}}(\CF_1)\otimes f^!(\CF_2)\right)\simeq 
\on{ev}_{\Dmod(\CX_2)}\left(\CF_2\otimes f_\blacktriangle(\BD_{\CX_1}^{\on{Ve}}(\CF_1))\right).$$
Moreover, by \cite[Lemma 10.4.2(a)]{DrGa1}, the object $f_\blacktriangle(\BD_{\CX_1}^{\on{Ve}}(\CF_1))\in \Dmod(\CX_2)$ is \emph{safe}.

\medskip 

Let $\CX$ be any QCA stack, and $\CF\in \Dmod(\CX)_{\on{coh}}$, $\CF'\in \Dmod(\CX)$. The morphism
\eqref{e:duality and pairing} gives rise to a map
\begin{equation} \label{e:!3}
\on{ev}_{\Dmod(\CX)}\left(\CF\otimes \CF'\right)\to 
\CMaps_{\Dmod(\CX)}\left(\BD^{\on{Ve}}_{\CX}(\CF), \CF'\right).
\end{equation}

The map $\leftarrow$ in \eqref{e:!2} will be the map \eqref{e:!3} for $\CX:=\CX_2$, $\CF:=\CF_2$, $\CF'= f_\blacktriangle(\BD_{\CX_1}^{\on{Ve}}(\CF_1))$.
Hence, it remains to show that the map \eqref{e:!3} is an isomorphism whenever $\CF'$ is safe.

\medskip

We have:
$$\on{ev}_{\Dmod(\CX)}\left(\CF\otimes \CF'\right)\simeq (p_\CX)_\blacktriangle(\CF\sotimes \CF'),$$
and by \cite[Lemma 7.3.5]{DrGa1},
$$\CMaps_{\Dmod(\CX)}\left(\BD^{\on{Ve}}_{\CX}(\CF), \CF'\right)=(p_\CX)_\bullet(\CF\sotimes \CF'),$$
and the map \eqref{e:!3} comes from the natural transformation $(p_\CX)_\blacktriangle\to (p_\CX)_\bullet$.

\medskip

Finally, if $\CF'$ is safe, then so is $\CF\sotimes \CF'$, and hence the map 
$$(p_\CX)_\blacktriangle(\CF\sotimes \CF')\to (p_\CX)_\bullet(\CF\sotimes \CF')$$
is an isomorphism by \cite[Proposition 9.2.9]{DrGa1}.

\end{proof}

\sssec{}

For a QCA algebraic stack, we consider the object 
$$k_\CX:=\BD_\CX^{\on{Ve}}(\omega_\CX) \in \Dmod(\CX)_{\on{coh}}.$$

By \propref{p:!}, the object
$$(\Delta_\CX)_!(k_\CX)\in \Dmod(\CX\times \CX)_{\on{coh}}$$
is well-defined and is isomorphic to
$$\BD^{\on{Ve}}_{\CX\times \CX}\left((\Delta_\CX)_\blacktriangle(\omega_\CX)\right),$$
where we recall that $(\Delta_\CX)_\blacktriangle\simeq (\Delta_\CX)_\bullet$, since $\Delta_\CX$
is representable and hence safe. 

\medskip

Note, however, that neither $(\Delta_\CX)_\blacktriangle(\omega_\CX)$ nor $(\Delta_\CX)_!(k_\CX)$
are in general compact. 

\medskip

We define the functor
$$\psId_\CX:\Dmod(\CX)\to \Dmod(\CX)$$
to be given by the kernel $(\Delta_\CX)_!(k_\CX)$ in the sense of \secref{sss:unit and counit stacks}

\ssec{The theorem for stacks}

\sssec{}

We have the following analog of \thmref{t:schemes} for QCA stacks. 

\begin{thm}  \label{t:stacks}
Let $\CQ$ be an object of $\Dmod(\CX_1\times \CX_2)_{\on{coh}}$. Assume that the corresponding
functor
$$\sF_{\CX_1\to \CX_2,\CQ}:\Dmod(X_1)\to \Dmod(X_2)$$
admits a continuous right adjoint.  Then the functor
$$\sF_{\CX_2\to \CX_1,\BD^{\on{Ve}}_{\CX_1\times \CX_2}(\CQ)}:\Dmod(X_2)\to \Dmod(X_1)$$
identifies canonically with
$$\Dmod(\CX_2)\overset{(\sF_{\CX_1\to \CX_2,\CQ})^R}\longrightarrow \Dmod(\CX_1)
\overset{\psId_{\CX_1}}\longrightarrow \Dmod(\CX_1).$$
\end{thm}

\sssec{}

Using \lemref{l:conjugate}, from \thmref{t:stacks} we obtain: 

\begin{cor} \label{c:stacks}
Let $\CQ$ be an object of $\Dmod(\CX_1\times \CX_2)_{\on{coh}}$. Assume that the corresponding
functor
$$\sF_{\CX_1\to \CX_2,\CQ}:\Dmod(X_1)\to \Dmod(X_2)$$
admits a continuous right adjoint.  Then the functor
$$\sF_{\CX_1\to \CX_2,\BD^{\on{Ve}}_{\CX_1\times \CX_2}(\CQ)}:\Dmod(X_1)\to \Dmod(X_2)$$
identifies canonically with
$$\Dmod(\CX_1)\overset{\psId_{\CX_1}}\longrightarrow \Dmod(\CX_1)
\overset{(\sF_{\CX_1\to \CX_2,\CQ})^{\on{op}}}\longrightarrow \Dmod(\CX_2).$$
\end{cor}


\ssec{Proof of \thmref{t:stacks}}

\sssec{}

Let $\CY_1$ and $\CY_2$ be QCA stacks, let $\CM$ be an object of $\Dmod(\CY_1)_{\on{coh}}$, 
and let
$$\sG:\Dmod(\CY_1)\to \Dmod(\CY_2),$$
given by a kernel $\CP\in \Dmod(\CY_1\times \CY_2)_{\on{coh}}$. Assume that $\sG$ preserves compactness. 

\medskip

We wish to know when the map 
\begin{equation} \label{e:functor and duality corr}
\sG^{\on{op}}(\BD_{\CY_1}^{\on{Ve}}(\CM))=\sG^{\on{op}}(\CM^\vee)
\to \sG(\CM)^\vee
\end{equation}
of \eqref{e:functor and duality} is an isomorphism.

\medskip

Consider the map
\begin{equation} \label{e:tr to bull corr}
\sG(\CM)=(\on{pr}_2)_\blacktriangle(\on{pr}_1^!(\CM)\sotimes \CP)\to (\on{pr}_2)_\bullet(\on{pr}_1^!(\CM)\sotimes \CP)
\end{equation}
of \eqref{e:tr to bullet}. 

\medskip

\begin{lem} \label{l:bdd estimate}
If \eqref{e:tr to bull corr} is an isomorphism, then so is \eqref{e:functor and duality corr}.
\end{lem}

\begin{rem}
The proof of \lemref{l:bdd estimate} will show that if \eqref{e:tr to bull corr} is an isomorphism, then 
$\sG(\CM)$ has coherent cohomologies and hence $\sG(\CM)^\vee$ is the same as 
$\BD_{\CY_2}^{\on{Ve}}(\sG(\CM))$.
\end{rem}

\sssec{}

Let us assume \lemref{l:bdd estimate} and finish the proof of the theorem. 
We need to show that the functor $\sF_{\CX_1\to \CX_2,\CQ}$ satisfies the condition of \thmref{t:general}. 

\medskip

We will apply \lemref{l:bdd estimate} in the following situation. We take
$$\CY_1=\CX_1\times \CX_1,\,\, \CY_2=\CX_1\times \CX_2,\,\, \CM=(\Delta_{\CX_1})_\blacktriangle(\omega_{\CX_1}),\,\,
\sG=\on{Id}_{\Dmod(\CX_1)}\otimes \sF_{\CX_1\to \CX_2,\CQ},$$
so that 
$$\CP\in \Dmod(\CX_1\times \CX_1\times \CX_1\times \CX_2)$$
is 
$$\sigma_{2,3}\left((\Delta_{\CX_1})_\blacktriangle(\omega_{\CX_1})\boxtimes \CQ\right),$$
where $\sigma_{2,3}$ is the transposition of the corresponding factors.  

\medskip

Base change for the $\blacktriangle$-pushforward and $!$-pullback for the Cartesian diagram
$$
\CD
\CX_1\times \CX_1\times \CX_1 \times \CX_2  
@>{\sigma_{4,5}(\Delta_{\CX_1\times \CX_1}\times \on{id}_{\CX_1\times \CX_2})}>>  
\CX_1\times \CX_1\times \CX_1\times \CX_1\times \CX_1 \times \CX_2 \\
@A{\Delta^2_{\CX_1}\times \on{Id}_{\CX_2}}AA   
@AA{\sigma_{2,3}(\Delta_{\CX_1\times \CX_1}\times \on{id}_{\CX_1\times \CX_2})}A   \\
\CX_1\times \CX_2  @>{\Delta^2_{\CX_1}\times \on{Id}_{\CX_2}}>> \CX_1\times \CX_1\times \CX_1 \times \CX_2 
\endCD
$$
(here $\Delta^2_{\CX_2}$ denotes the diagonal morphism $\CX_1\to \CX_1\times \CX_1\times \CX_1$) 
implies that in our case
the left-hand side in \eqref{e:tr to bull corr} is canonically isomorphic to
$$\CQ\in \Dmod(\CX_1\times \CX_2).$$

Now, the base change morphism for the $\bullet$-pushforward and $!$-pullback is not always an isomorphism,
but by \cite[Proposition 7.6.8]{DrGa1} it is an isomorphism for eventually coconnective objects. Hence, the
right-hand side in \eqref{e:tr to bull corr} identifies with
\begin{equation} \label{e:re-obtain Q}
(p_{\CX_1\times \CX_1}\times \on{id}_{\CX_1\times \CX_2})_\bullet\circ
(\Delta^2_{\CX_1}\times \on{Id}_{\CX_2})_\bullet(\CQ).
\end{equation}

\medskip

Again, the $\bullet$-pushforward is not always functorial with respect to compositions of morphisms
(see \cite[Sect. 7.8.7]{DrGa1}), but it is functorial when evaluated on eventually coconnective 
objects by \cite[Sect. 7.8.6(iii)]{DrGa1}. Hence, \eqref{e:re-obtain Q} is isomorphic to $\CQ$,
as required. 

\qed[\thmref{t:stacks}]

\sssec{Proof of \lemref{l:bdd estimate}}



By \cite[Lemma 9.4.7(b)]{DrGa1}, we can find an inverse system of objects $\CM_n\in \Dmod(\CY_1)^c$, equipped
with a compatible system of maps 
$$\CM\to \CM_n,$$
such that $\on{Cone}(\CM\to \CM_n)\in \Dmod(\CY_1)^{\geq n}$. 
Then 
$$\BD_{\CY_1}^{\on{Ve}}(\CM) \simeq \underset{n}{colim}\, \BD_{\CY_1}^{\on{Ve}}(\CM_n),$$
since the functor $\BD_{\CY_1}^{\on{Ve}}$ has a bounded cohomological
amplitude.

\medskip

Hence, the left-hand side in \eqref{e:functor and duality corr} is given by
$$\underset{n}{colim}\,  \sG^{\on{op}}(\BD_{\CY_1}^{\on{Ve}}(\CM_n))\simeq
\underset{n}{colim}\, \BD_{\CY_2}^{\on{Ve}}\left(\sG(\CM_n)\right).$$

\medskip

By \propref{p:Verdier and check}, in order to prove that \eqref{e:functor and duality corr} is an
isomorphism, it suffices to show that for every integer $k$ there exists $n_0$ such that 
the map
$$\sG(\CM)\to \sG(\CM_n)$$
induces an isomorphism in cohomological degrees $\leq k$ for $n\geq n_0$. 

\medskip

Consider the commutative diagram
\begin{equation} \label{e:comp corr}
\CD
\sG(\CM_n)  @>{=}>> (\on{pr}_2)_\blacktriangle(\on{pr}_1^!(\CM_n)\sotimes \CP) @>>>  (\on{pr}_2)_\bullet(\on{pr}_1^!(\CM_n)\sotimes \CP) \\
@AAA   @AAA  @AAA   \\ 
\sG(\CM) @>{=}>> (\on{pr}_2)_\blacktriangle(\on{pr}_1^!(\CM)\sotimes \CP) @>>>   (\on{pr}_2)_\bullet(\on{pr}_1^!(\CM)\sotimes \CP).
\endCD
\end{equation}

By assumption, the bottom horizontal arrows in \eqref{e:comp corr} are isomorphisms. We have the following assertion, proved below:

\begin{lem} \label{l:rel safety}
For any $\CN\in \Dmod(\CY_1)^c$ and $\CP\in \Dmod(\CY_1\times \CY_2)$, the map
$$(\on{pr}_2)_\blacktriangle(\on{pr}_1^!(\CN)\sotimes \CP) \to  (\on{pr}_2)_\bullet(\on{pr}_1^!(\CN)\sotimes \CP)$$
is an isomorphism.
\end{lem}

Assuming \lemref{l:rel safety}, we obtain that the top horizontal arrows in \eqref{e:comp corr} are also isomorphisms. Hence,
it is sufficient to show that for every integer $k$ there exists $n_0$ such that 
the map
$$(\on{pr}_2)_\bullet(\on{pr}_1^!(\CM)\sotimes \CP)\to (\on{pr}_2)_\bullet(\on{pr}_1^!(\CM_n)\sotimes \CP)$$
induces an isomorphism in cohomological degrees $\leq k$ for $n\geq n_0$. 

\medskip

However, this follows from the fact that the functor $\sotimes$ has a bounded cohomological amplitude,
and the functor of $\bullet$-direct image is left t-exact up to a cohomological shift.

\qed[\lemref{l:bdd estimate}]

\sssec{Proof of \lemref{l:rel safety}}

First, by \cite[Proposition 9.3.7]{DrGa1}
the map
$$(\on{pr}_2)_\blacktriangle(\on{pr}_1^!(\CN')\sotimes \CP') \to  (\on{pr}_2)_\bullet(\on{pr}_1^!(\CN')\sotimes \CP')$$
is an isomorphism for any $\CN'\in \Dmod(\CY_1)$ and $\CP'\in \Dmod(\CY_1\times \CY_2)^c$. 
Hence, it suffices to show that for $\CN\in \Dmod(\CY_1)^c$, the functor
$$\CP\mapsto  (\on{pr}_2)_\bullet(\on{pr}_1^!(\CN)\sotimes \CP)$$
is continuous.

\medskip

This is equivalent to showing that for any fixed $\CM\in \Dmod(\CY_2)^c$, the functor
$$\CP\mapsto \CMaps_{\Dmod(\CY_2)}\left(\CM,(\on{pr}_2)_\bullet(\on{pr}_1^!(\CN)\sotimes \CP)\right)$$
is continuous. We have:
$$\CMaps_{\Dmod(\CY_2)}\left(\CM,(\on{pr}_2)_\bullet(\on{pr}_1^!(\CN)\sotimes \CP)\right)\simeq
\CMaps_{\Dmod(\CY_1\times \CY_2)}(k_{\CY_1}\boxtimes \CM,\on{pr}_1^!(\CN)\sotimes \CP),$$
which by \cite[Lemma 7.3.5]{DrGa1} can be rewritten as
$$(p_{\CY_1\times \CY_2})_\bullet\left(\BD_{\CY_1\times \CY_2}^{\on{Ve}}(k_{\CY_1}\boxtimes \CM)\sotimes
\on{pr}_1^!(\CN)\sotimes \CP\right)\simeq 
(p_{\CY_1\times \CY_2})_\bullet\left( (\CN\boxtimes \BD_{\CY_2}^{\on{Ve}}(\CM))\sotimes \CP\right).$$

Now, the object $$\CN\boxtimes \BD_{\CY_2}^{\on{Ve}}(\CM)\in \Dmod(\CY_1\times \CY_2)$$ is compact,
and hence, by \cite[Proposition 9.2.3]{DrGa1}, safe. This implies the assertion of the lemma, by the definition
of safety.

\qed[\lemref{l:rel safety}]

\ssec{Mock-proper stacks}

We shall now discuss some applications of \thmref{t:stacks}.

\sssec{}

Let us call a QCA stack $\CX$ \emph{mock-proper} if the functor $(p_\CX)_\blacktriangle$ preserves
compactness. (Recall that $(p_\CX)_\blacktriangle|_{\Dmod(\CX)^c}=(p_\CX)_\bullet|_{\Dmod(\CX)^c}$,
so the above condition is equivalent to $(p_\CX)_\bullet$ preserving compactness.) 

\medskip

An example of a mock-proper stack will be given in \secref{ss:vector space}. Another set
of examples is supplied by \corref{c:cotrunk of BunG}.

\medskip

Note that from \corref{c:left adj via check}, we obtain that $\CX$ is mock-proper if and only if the functor
$(p_\CX)_!$, left adjoint to $p_\CX^!$, is defined. 

\begin{prop}  \label{p:mock-proper}
Let $\CX$ be mock-proper and smooth of dimension $n$. Then we have a canonical isomorphism of functors
$$(p_\CX)_\blacktriangle \simeq (p_\CX)_!\circ \psId_\CX[2n].$$
\end{prop} 

\begin{proof}

We apply \corref{c:stacks} to the functor $(p_\CX)_\blacktriangle$. The functor in question is given by the kernel
$\omega_\CX\in \Dmod(\CX)$. Since
$(p_\CX)_!\simeq ((p_\CX)_\blacktriangle)^{\on{op}}$, we obtain that the functor 
$(p_\CX)_!\circ \psId_\CX$ is given by the kernel $k_\CX$. Since $\CX$ is smooth
of dimension $n$, we obtain that $(p_\CX)_!\circ \psId_\CX[2n]$ is given by the kernel
$\omega_\CX$, i.e., the same as $(p_\CX)_\blacktriangle$.

\end{proof} 

\begin{rem}
Retracing the proof of \thmref{t:stacks} one can prove the following generalization of
\propref{p:mock-proper}. Let $\CX$ be mock-proper, but not necessarily smooth. Then
there is a canonical isomorphism
$$(p_\CX)_\blacktriangle \simeq (p_\CX)_!\circ \sF_{\CX\to \CX,(\Delta_\CX)_!(\omega_\CX)}.$$
\end{rem}

\sssec{}

Passing to dual functors in \propref{p:mock-proper}, and using \lemref{l:conjugate}, we obtain:

\begin{cor} \label{c:mock-proper}
Let $\CX$ be mock-proper and smooth of dimension $n$. Then we have a canonical isomorphism of functors
$$p_\CX^!\simeq \psId_\CX\circ ((p_\CX)_\blacktriangle)^R[2n].$$
\end{cor}

For a mock-proper stack, we shall denote by $\omega_{\CX,\on{mock}}$ the object
$$((p_\CX)_\blacktriangle)^R(k)\in  \Dmod(\CX).$$

We note that when $\CX$ is a proper scheme, $\omega_{\CX,\on{mock}}=\omega_\CX$. 

\medskip

Note that \corref{c:mock-proper} can be reformulated as saying that for $\CX$ smooth of dimension $n$ we have:
$$\psId_\CX(\omega_{\CX,\on{mock}})[2n]\simeq \omega_\CX.$$

\begin{rem}
Again, if $\CX$ is mock-proper, but not necessarily smooth, we have
$$\sF_{\CX\to \CX,(\Delta_\CX)_!(\omega_\CX)}(\omega_{\CX,\on{mock}})\simeq \omega_\CX.$$
\end{rem}

\ssec{Truncative and co-truncative substacks}  \label{ss:trunc}

\sssec{Co-truncative substacks}

Let $j:\CX_1\hookrightarrow \CX_2$ be an open embedding of QCA stacks. Recall that
according to \cite[Definition 3.1.5]{DrGa2}, $j$ is said to be \emph{co-truncative} if the partially defined
left adjoint to $j^!$, i.e., the functor $j_!$, is defined on all $\Dmod(\CX_1)$. 

\medskip

According to \corref{c:left adj via check}, this condition is equivalent to the functor $j_\bullet$ (which is the
same as $j_\blacktriangle$) preserving compactness.

\medskip

A typical example of a co-truncative open embedding will be considered in \secref{ss:vector space}.
Another series of examples is supplied in \cite{DrGa2}, where it is shown that the moduli stack
$\Bun_G$ of $G$-bundles on $X$ (here $G$ is a reductive group and $X$ a smooth complete curve)
can be written as a union of quasi-compact substacks under co-truncative open embeddings.

\begin{prop}[Drinfeld] \label{p:!and*}
Let $j$ be co-truncative. Then there is a canonical isomorphism of functors 
$$\psId_{\CX_2}\circ j_\bullet\simeq j_!\circ \psId_{\CX_1}.$$
\end{prop}

Just as an illustration, we will give a proof of \propref{p:!and*} using \thmref{t:stacks}. However,
one can give a more direct proof, see \lemref{l:on open}.

\begin{proof}

Consider the functor $j_\bullet$. It is given by the kernel
\begin{equation} \label{e:ker direct image}
\CQ:=(\on{id}_{\CX_1}\times j)_\bullet(\omega_{\CX_1})\in \Dmod(\CX_1\times \CX_2),
\end{equation}
where by a slight abuse of notation we denote by $\on{id}_{\CX_1}\times j$ the graph of the map $j$.

\medskip

Note that $j_!\simeq (j_\bullet)^{\on{op}}$ by \lemref{l:conjugate}. Hence, by \corref{c:stacks} applied to $j_\bullet$, the functor
$j_!\circ \psId_{\CX_1}$ is given by the kernel 
$$\BD_{\CX_1\times \CX_1}^{\on{Ve}}(\CQ).$$

\medskip

Consider now the functor $j^!$. It is also given by the kernel \eqref{e:ker direct image}. 
Since $j$ is an open embedding, we have $j^!\simeq j^\bullet$, and hence
$(j^!)^R\simeq j_\bullet$. Hence, by \thmref{t:stacks} applied
to $j^!$, we obtain that $\psId_{\CX_2}\circ j_\bullet$ is also given by
$$\BD_{\CX_1\times \CX_1}^{\on{Ve}}(\CQ),$$
as required. \footnote{Note that the above kernel is isomorphic to $(\on{id}_{\CX_1}\times j)_!(k_{\CX_1})$.}

\end{proof}

Passing to the dual functors, we obtain:

\begin{cor}  \label{c:!and*}
There is a canonical isomorphism of functors
$$\psId_{\CX_1}\circ j^?\simeq j^\bullet\circ \psId_{\CX_2},$$
where $j^?$ denotes the (continuous!) right adjoint of $j_\bullet$.
\end{cor}

\sssec{Truncative substacks}

Let $i:\CX_1\to \CX_2$ be a closed embedding. Recall (see \cite[Definition 3.1.5]{DrGa2}) that $i$ is said
to be \emph{truncative} if the partially defined left adjoint to $i_\bullet$, i.e., the functor 
$i^\bullet$, is defined on all of $\Dmod(\CX_2)$. 

\medskip

According to  \corref{c:left adj via check}, this is equivalent to the functor $i^!$ preserving compactness. 
Still, equivalently, $i$ is truncative if and only if the complementary open embedding is co-truncative;
see \cite[Sects. 3.1-3.3]{DrGa2} for a detailed discussion of the properties of truncativeness and
co-truncativeness.

\medskip

As in \propref{p:!and*} we show:

\begin{prop}   \label{p:trunc}
Let $i:\CX_1\to \CX_2$ be truncative. Then we have a canonical isomorphism
of functors
$$i^\bullet\circ \psId_{\CX_2}\simeq \psId_{\CX_1}\circ i^!.$$
\end{prop}

Passing to dual functors, one obtains:

\begin{cor}   \label{c:trunc}
There is a canonical isomorphism of functors
$$\psId_{\CX_2}\circ i_?\simeq i_\bullet\circ \psId_{\CX_1},$$
where $i_?$ is the (continuous!) right adjoint to $i^!$.
\end{cor}

\ssec{Miraculous stacks}

\sssec{}

Following \cite[Definition 4.5.2]{DrGa2}, we shall say that a QCA stack $\CX$ is \emph{miraculous}
if the category $\Dmod(\CX)$ is Gorenstein (see \secref{ss:Gorenstein}), i.e., if the functor
$\psId_\CX$ is an equivalence. 

\medskip

From \propref{p:when Gor} and \thmref{t:stacks}, we obtain:

\begin{cor}  \label{c:miraculous}
Let $\CX$ be a QCA stack for which the functor $\psId_\CX$ preserves compactness.
Then $\CX$ is miraculous.
\end{cor}

\sssec{Classifying space of a group}  \label{sss:classifying}

Let $G$ be an affine algebraic group, and consider the stack $\CX:=\on{pt}/G$. We claim
that it is both mock-proper and miraculous.

\medskip

Indeed, the category $\Dmod(\on{pt}/G)$ is compactly generated by one object, namely,
$\pi_\bullet(k)$, where $\pi:\on{pt}\to \on{pt}/G$. Now,
$$(p_{\on{pt}/G})_\bullet(\pi_\bullet(k))\simeq k\in \Vect^c.$$
Hence $\on{pt}/G$ is mock-proper. 

\medskip

Similarly, it is easy to see that
$$(\Delta_{\on{pt}/G})_!(k_{\on{pt}/G})\simeq (\Delta_{\on{pt}/G})_\bullet(k_{\on{pt}/G})[-d_G]\simeq 
(\Delta_{\on{pt}/G})_\bullet(\omega_{\on{pt}/G})[-d_G+2\dim(G)],$$
where 
$$d_G=
\begin{cases}
&2\dim(G) \text{ if $G$ is unipotent}; \\
&\dim(G) \text{ if $G$ is reductive} .
\end{cases}
$$

Hence, 
$$\psId_{\on{pt}/G}\simeq \on{Id}_{\Dmod(\on{pt}/G)}[-d_G+2\dim(G)]].$$

\ssec{An example of a miraculous stack} \label{ss:vector space}

The results of this and the next subsection were obtained jointly with A.~Beilinson and V.~Drinfeld. 

\sssec{}

Let $V$ be a vector space, considered as a scheme, and consider the stack 
$V/\BG_m$.  We will prove:

\begin{prop} \label{p:vector space}
The stack $V/\BG_m$ is miraculous and mock-proper. 
\end{prop}

\sssec{}

Let $i$ denote the closed embedding $\on{pt}/\BG_m\to V/\BG_m$, and let $j$ denote the complementary
open embedding
$$\BP(V)\simeq (V-0)/\BG_m\hookrightarrow V/\BG_m.$$
Let, in addition, $\pi$ denote the projection map $V/\BG_m\to \on{pt}/\BG_m$.

\medskip

According to \cite[Sect. 3.2.2]{DrGa2}, the closed embedding $i$ (resp., open embedding $j$) is
truncative (resp., co-truncative). Moreover, by \cite[Sect. 5.3]{DrGa2}, we have canonical isomorphisms of functors
$$i^\bullet\simeq \pi_\bullet,\quad i^!\simeq \pi_!,$$
and hence 
\begin{equation} \label{e:?!}
i_?\simeq \pi^!.
\end{equation} 

\sssec{}

The fact that $V/\BG_m$ is mock-proper follows from the fact that the functor $\pi_\bullet$ preserves
compactness (being the left adjoint of $i_\bullet$), combined with the fact that $\on{pt}/\BG_m$
is mock-proper. 

\sssec{}

Let us write down the isomorphisms of Propositions \ref{p:!and*} and \ref{p:trunc} and
Corollaries \ref{c:!and*} and \ref{c:trunc} in our case. 

\medskip

For that we note that the functor $\psId_{\BP(V)}$ identifies with $\on{Id}_{\Dmod(\BP(V))}[-2(\dim(V)-1)]$,
since $\BP(V)$ is a smooth separated scheme of dimension $\dim(V)-1$. 
By \secref{sss:classifying}, the functor $\psId_{\on{pt}/\BG_m}$ identifies with $\on{Id}_{\Dmod(\on{pt}/\BG_m)}[1]$.

\medskip

From \propref{p:!and*}, we obtain: 
\begin{equation} \label{e:psidV1}
\psId_{V/\BG_m}\circ j_\bullet\simeq j_![-2(\dim(V)-1)].
\end{equation}

From \corref{c:trunc} and \eqref{e:?!}, we obtain:
\begin{equation} \label{e:psidV4}
\psId_{V/\BG_m} \circ \pi^!\simeq i_\bullet[1].
\end{equation}

\medskip

From \propref{p:trunc} we obtain
$$i^\bullet \circ \psId_{V/\BG_m}\simeq i^![1]$$
and from \corref{c:!and*}: 
$$j^\bullet\circ \psId_{V/\BG_m}\simeq j^?[-2(\dim(V)-1)].$$

\sssec{}

In order to show that $V/\BG_m$ is miraculous, by \corref{c:miraculous}, it is sufficient to
show that the functor $\psId_{V/\BG_m}$ preserves compactness. 

\medskip

The category
$\Dmod(V/\BG_m)^c$ is generated by the essential images of $\Dmod((V-0)/\BG_m)^c$
and $\Dmod(\on{pt}/\BG_m)^c$ under the functors $j_\bullet$ and $\pi^!$, respectively. 
Hence, it is sufficient to show that the functors
$$\psId_{V/\BG_m}\circ j_\bullet \text{ and } \psId_{V/\BG_m}\circ \pi^!$$
preserve compactness.

\medskip

However, this follows from \eqref{e:psidV1} and \eqref{e:psidV4}, respectively. 

\begin{rem} 

To complete the picture, one can show that there is a canonical
isomorphism of functors
$$\psId_{V/\BG_m}\circ i_\bullet\simeq \pi^![-2(\dim(V))+1].$$

In particular, if $\dim(V)>1$, it is \emph{not} true that $\psId_{V/\BG_m}$ is an involution. 

\medskip

However, one can show that $\psId_{V/\BG_m}$ is an involution if $\dim(V)=1$. Indeed, for $V=k$,
one can show that $\psId_{V/\BG_m}$ is isomorphic to the functor of Fourier-Deligne transform. 

\end{rem}

\ssec{A non-example}

\sssec{}

Consider now the following stack $X:=(\BA^2-0)/\BG_m$, where we consider the {\it hyperbolic}
action of $\BG_m$ on $\BA^2$,
$$\lambda (x_1,x_2)=(\lambda\cdot x_1,\lambda^{-1}\cdot x_2).$$

In fact $X$ is a non-separated scheme, namely,
$\wt\BA^1$, i.e., $\BA^1$ with a double point. Let $i_1$ and $i_2$ denote 
the corresponding two closed embeddings $\on{pt}\to X$. 

\medskip

We claim that $X$ is \emph{not} miraculous. We will show that the functor $\psId_X$
fails to preserve compactness.

\sssec{}

Consider the canonical map
\begin{equation} \label{e:map of Deltas}
(\Delta_X)_!(k_X)\to (\Delta_X)_\bullet(k_X)\simeq  (\Delta_X)_\bullet(\omega_X)[-2].
\end{equation}

\begin{lem} \label{l:map of Deltas}
The cone of the map \eqref{e:map of Deltas} 
is isomorphic to 
the direct sum 
$$(i_1\times i_2)_\bullet(k\oplus k[-1])\oplus (i_2\times i_1)_\bullet(k\oplus k[-1]).$$
\end{lem}

\begin{proof} 

Let $U_1$ and $U_2$ be the two open charts of $X$, each isomorphic to $\BA^1$, so that
$$U_1\cap U_2\to U_1\times U_2$$
is the map
$$\BA^1-0\hookrightarrow \BA^1\overset{\Delta}\longrightarrow \BA^1\times \BA^1.$$

The assertion of the lemma follows by calculating the map \eqref{e:map of Deltas} on the
charts $U_1\times U_1$ and $U_2\times U_2$ (where it
also an isomorphism), and $U_1\times U_2$ and $U_2\times U_1$, each of which
contributes the corresponding direct summand.

\end{proof}

\sssec{}

By \lemref{l:map of Deltas}, it is sufficient to show that the functor
$\Dmod(X)\to \Dmod(X)$, given by the kernel $(i_1\times i_2)_\bullet(k)$ does not
preserve compactness. 

\medskip

However, the letter functor identifies with $(i_2)_\bullet\circ (i_1)^!$, which
sends $\on{D}_X$ to a non-compact object. 

\section{Artin stacks: the non-quasi compact case} \label{s:Artin gen}

\ssec{Truncatable stacks}

\sssec{}

Let $\CX$ be an algebraic stack, which is locally QCA, i.e., one that can be covered by
quasi-compact algebraic stacks that are QCA.

\medskip

Recall (see \cite[Lemma 2.3.2]{DrGa2}) that the category $\Dmod(\CX)$ is equivalent to
$$\underset{U\in \on{Open-qc}(\CX)^{\on{op}}}{lim}\, \Dmod(U),$$
where $\on{Open-qc}(\CX)$ is the poset of quasi-compact open substacks of $\CX$. 

\sssec{}  \label{sss:truncatable}

We shall say that an open substack $U\subset \CX$ is co-truncative
if for any quasi-compact open substack $U'\subset \CX$, the inclusion 
$$U\cap U'\hookrightarrow U'$$
is co-truncative. 

\medskip

Recall (see \cite[Definition 4.1.1]{DrGa2}) that $\CX$ is said to be \emph{truncatable} if 
$\CX$ can be covered by its quasi-compact open co-truncative  substacks.

\medskip

Our main example of a truncatable stack is $\Bun_G$, the moduli stack of $G$-bundles
on a smooth complete curve $X$, where $G$ is a reductie group. This fact is proved in 
\cite[Theorem 4.1.8]{DrGa2}. 

\sssec{}

We let $\on{Ctrnk}(\CX)$ denote the poset of quasi-compact open co-truncative  
subsets of $\CX$. Note that according to \cite[Lemma 3.8.4]{DrGa2}, the union of co-truncative
subsets is co-truncative. Hence, $\on{Ctrnk}(\CX)$ is filtered.  

\medskip

Furthermore, the condition of being truncatable is equivalent to the map of posets
$$\on{Ctrnk}(\CX)\to \on{Open-qc}(\CX)$$
being co-final. I.e., $\CX$ is truncatable if and only if every quasi-compact open substack of $\CX$ is contained in one which is co-truncative.

\medskip

Hence, for $\CX$ truncatable, we have
$$\Dmod(\CX)\simeq \underset{U\in \on{Ctrnk}(\CX)^{\on{op}}}{lim}\, \Dmod(U).$$

\sssec{}  \label{sss:co}

From now on we will assume that all our algebraic stacks are locally QCA and truncatable. 

\medskip

According to \cite[Proposition 4.1.6]{DrGa2}, the category $\Dmod(\CX)$
is compactly generated. The set of compact generators is provided by the objects
$$j_!(\CF), \quad j:U\hookrightarrow \CX,\quad U\in \on{Ctrnk}(\CX),\quad \CF\in \Dmod(U)^c.$$

\medskip

We introduce the DG category $\Dmod(\CX)_{\on{co}}$ as
$$\underset{\on{Ctrnk}(\CX)^{\on{op}}}{lim}\, \Dmod^?,$$
where the functor
$\Dmod^?:\on{Ctrnk}(\CX)^{\on{op}}\to \StinftyCat_{\on{cont}}$ sends
$$U\rightsquigarrow \Dmod(U) \text{ and } (U_1\overset{j_{1,2}}\hookrightarrow U_2) \rightsquigarrow j_{1,2}^?$$
(the functor $j_{1,2}^?$ is the \emph{right} adjoint of $(j_{1,2})_\bullet$, see \secref{ss:trunc}, and also
\secref{sss:strange functors} below).

\medskip

For $(U\overset{j}\hookrightarrow \CX)\in  \on{Ctrnk}(\CX)$, the tautological evaluation functor
$$j^?:\Dmod(\CX)_{\on{co}}\to \Dmod(U)$$
admits a \emph{left} adjoint, denoted $j_{\on{co},\bullet}$. The category $\Dmod(\CX)_{\on{co}}$ is compactly
generated by objects
$$j_{\on{co},\bullet}(\CF), \quad (j:U\hookrightarrow \CX)\in \on{Ctrnk}(\CX),\quad \CF\in \Dmod(U)^c.$$

\medskip

For $U\overset{j}\hookrightarrow \CX$ as above, the functor $j_{\on{co},\bullet}$ also admits a left adjoint, denoted
$j^\bullet_{\on{co}}$. 
 
\sssec{}

By \cite[Corollaries 4.3.2 and 4.3.5]{DrGa2}, there is a canonically defined equivaence
$$\bD_\CX^{\on{Ve}}:\Dmod(\CX)^\vee\to \Dmod(\CX)_{\on{co}}.$$

It is characterized by the property that the corresponding functor
$$\BD_\CX^{\on{Ve}}:(\Dmod(\CX)^c)^{\on{op}}\to (\Dmod(\CX)_{\on{co}})^c$$ 
acts as follows
$$\BD_\CX^{\on{Ve}}\left(j_!(\CF)\right)=j_{\on{co},\bullet}(\BD_U^{\on{Ve}}(\CF)), \quad 
\CF\in \Dmod(U)^c,\quad (U\overset{j}\hookrightarrow \CX)\in  \on{Ctrnk}(\CX).$$ 

\sssec{}   \label{sss:strange functors}

For a co-truncative quasi-compact $U\overset{j}\hookrightarrow \CX$ we have the following
isomorphisms:
$$(j_!)^{\on{op}}\simeq j_{\on{co},\bullet},\quad (j_!)^\vee\simeq j^?, \quad (j^\bullet)^\vee\simeq j_{\on{co},\bullet},$$
from which, using \lemref{l:conjugate}, we obtain 
$$(j_\bullet)^\vee \simeq j^\bullet_{\on{co}},\quad (j^\bullet)^{\on{op}}\simeq j^\bullet_{\on{co}}.$$

\ssec{Additional properties of $\Dmod(\CX)_{\on{co}}$}

The following several additional pieces of information regarding the categories $\Dmod(\CX)$ and $\Dmod(\CX)_{\on{co}}$ will be
used in the sequel.

\sssec{}

According to \cite[Corollaries 4.3.2 and 4.3.5]{DrGa2}, the functors
$$j_{\on{co},\bullet}:\Dmod(U)\to \Dmod(\CX)_{\on{co}},\quad (j:U\hookrightarrow \CX) \in \on{Ctrnk}(\CX)$$
have the property that the induced functor
\begin{equation} \label{e:as colimit}
\underset{\on{Ctrnk}(\CX)}{colim}\, \Dmod_{\bullet}\to \Dmod(\CX)_{\on{co}}
\end{equation}
is an equivalence, where the functor
$$\Dmod_{\bullet}:\on{Ctrnk}(\CX)\to \StinftyCat_{\on{cont}}$$ sends
$$U\rightsquigarrow \Dmod(U) \text{ and } 
(U_1\overset{j_{12}}\hookrightarrow U_2) \rightsquigarrow (j_{12})_\bullet.$$

\sssec{}  \label{sss:black nonqc}

For a truncatable stack $\CX$ we define a continuous functor
$$(p_\CX)_\blacktriangle:\Dmod(\CX)_{\on{co}}\to \Vect$$
to be the dual of the functor
$$(p_\CX)^!:\Vect\to \Dmod(\CX).$$

In terms of the equivalence 
$$\underset{U\in \on{Ctrnk}(\CX)}{colim}\, \Dmod(U)\to \Dmod(\CX)$$
of \eqref{e:as colimit}, the funtcor $(p_\CX)_\blacktriangle$ corresponds to the
family of functors $\Dmod(U)\to \Vect$, given by $(p_U)_\blacktriangle$, which
are naturally compatible under
$$(p_{U_1})_\blacktriangle \circ (j_{1,2})_\bullet \simeq (p_{U_2})_\blacktriangle,\quad U_1\overset{j_{1,2}}\hookrightarrow U_2
\overset{j_2}\hookrightarrow \CX.$$

\sssec{}  \label{sss:sotimes}

Next, we claim that the $\sotimes$ operation defines a canonical action of the monoidal category
$\Dmod(\CX)$ on $\Dmod(\CX)_{\on{co}}$. In terms of the equivalence \eqref{e:as colimit}, for
$\CF\in \Dmod(\CX)$ and 
$$\CF_U\in \Dmod(U), \quad (U\overset{j}\hookrightarrow \CX)\in \on{Ctrnk}(\CX),$$
we have
$$\CF\sotimes j_{\on{co},\bullet}(\CF_U):=j_{\on{co},\bullet}(j^\bullet(\CF)\sotimes \CF_U).$$

\medskip

The following assertion will be used in the sequel:

\begin{lem} \label{l:Hom as coh}
For $\CF\in (\Dmod(\CX)_{\on{co}})^c$ and $\CF_1\in \Dmod(\CX)_{\on{co}}$, there is a canonical isomorphism
$$\CMaps_{\Dmod(\CX)_{\on{co}}}(\CF,\CF')\simeq (p_\CX)_\blacktriangle(\BD^{\on{Ve}}_{\Bun(G)}(\CF)\sotimes \CF'),$$
where $\CF\mapsto \BD^{\on{Ve}}_{\Bun(G)}(\CF)$ is the equivalence
$$((\Dmod(\CX)_{\on{co}})^c)^{\on{op}}\to \Dmod(\CX)^c,$$
induced by $\bD^{\on{Ve}}_\CX$.
\end{lem}

\ssec{Kernels in the non-quasi compact situation}

\sssec{}

For a pair of truncatable stacks $\CX_1$ and $\CX_2$, let $\CQ$ be an object of the category
$$\Dmod(\CX_1)\otimes \Dmod(\CX_2)_{\on{co}}.$$

We shall say that $\CQ$ is \emph{coherent} if for any pair of quasi-compact open co-truncative substacks $U_1\overset{j_1}\hookrightarrow \CX_1$
and $U_2\overset{j_2}\hookrightarrow \CX_2$ we have
$$((j_1)^\bullet\otimes (j_2)_{\on{co}}^\bullet)(\CQ)\in \Dmod(U_1)\otimes \Dmod(U_2)\simeq \Dmod(U_1\times U_2)$$
is coherent.

\medskip

We claim that for any $\CQ$ which is coherent, there is a well-defined object, denoted
$$\BD_{\CX_1\times \CX_2}^{\on{Ve}}(\CQ)\in \Dmod(\CX_1\times \CX_2)_{\on{coh}}$$
(Note that the notion of coherence for an object of $\Dmod(\CX)$ makes sense for
not necessarily quasi-compact algebraic stacks.)
 
\medskip

Namely, we define $\BD_{\CX_1\times \CX_2}^{\on{Ve}}(\CQ)$ be requiring that for any quasi-compact open co-truncative 
$U_1\overset{j_1}\hookrightarrow \CX_1$ and $U_2\overset{j_2}\hookrightarrow \CX_2$, we have:
$$(j_1\times j_2)^\bullet\left(\BD_{\CX_1\times \CX_2}^{\on{Ve}}(\CQ)\right)\simeq
\BD_{U_1\times U_2}^{\on{Ve}}\left(((j_1)^\bullet\otimes (j_2)_{\on{co}}^\bullet)(\CQ)\right).$$

\sssec{}

Let us note that by \secref{sss:functors and kernels}, the category 
$$\Dmod(\CX_1)\otimes \Dmod(\CX_2)_{\on{co}}.$$
is equivalent to that of continuous functors $$\Dmod(\CX_1)_{\on{co}}\to \Dmod(\CX_2)_{\on{co}}.$$

\medskip

The category
$$\Dmod(\CX_1)\otimes \Dmod(\CX_2)\simeq \Dmod(\CX_1\times \CX_2)$$
is equivalent to that of continuous functors $$\Dmod(\CX_1)_{\on{co}}\to \Dmod(\CX_2).$$

\medskip

In both cases, we will denote this assignment by
$$\CQ\rightsquigarrow \sF_{\CX_1\to \CX_2,\CQ}.$$



\sssec{}

Note now that for a stack $\CX$, the object 
$$(\Delta_{\CX})_!(k_\CX)\in \Dmod(\CX\times \CX)$$
is well-defined.

\medskip

It has the property that for every quasi-compact open $j:U\hookrightarrow \CX$, we have
$$(j\times j)^\bullet\left((\Delta_{\CX})_!(k_\CX)\right)\simeq (\Delta_{U})_!(k_U).$$

Indeed, the functor $(j\times j)^\bullet\circ (\Delta_{\CX})_!$ is the partially defined left adjoint to
$$\Delta_{\CX}^!\circ (j\times j)_\bullet\simeq j_\bullet \circ \Delta_U^!,$$
as is $(\Delta_U)_!\circ j^\bullet$. 

\sssec{}  \label{sss:Ps}

We define the functor 
$$\psId_\CX: \Dmod(\CX)_{\on{op}}\to \Dmod(\CX)$$ to be given by the kernel 
$$(\Delta_{\CX})_!(k_\CX)\in \Dmod(\CX\times \CX)\simeq \Dmod(\CX)\otimes \Dmod(\CX)
\simeq (\Dmod(\CX)_{\on{op}})^\vee\otimes \Dmod(\CX).$$

\sssec{}   \label{sss:naive}

Finally, in addition to the functor $\psId_\CX:\Dmod(\CX)_{\on{co}}\to \Dmod(\CX)$, introduced above, 
there is another functor, denoted $$\psId_\CX^{\on{naive}}:\Dmod(\CX)_{\on{co}}\to \Dmod(\CX).$$
It is given by the kernel
$$(\Delta_\CX)_\bullet(\omega_\CX)\in  \Dmod(\CX\times \CX).$$

\medskip

In terms of the equivalence \eqref{e:as colimit}, the functor $\psId_\CX^{\on{naive}}$ corresponds to the family of functors
$$\Dmod(U)\to \Dmod(\CX), \quad (U\overset{j}\hookrightarrow \CX)\in \on{Ctrnk}(\CX),$$
given by $j_\bullet$, that are compatible under the isomorphisms
$$(j_1)_\bullet \simeq  (j_{12})_\bullet\circ  (j_2)_\bullet, \quad U_1\overset{j_{12}}\hookrightarrow U_2 \overset{j_2}\hookrightarrow \CX.$$

This functor is not an equivalence, unless the closure of any quasi-compact open substack of $\CX$ is quasi-compact, 
see \cite[Proposition 4.4.5]{DrGa2}.

\medskip

In \secref{ss:bizarre} we will describe a particular
object in the kernel of this functor for $\CX=\Bun_G$.

\ssec{The theorem for truncatable stacks}

\sssec{}

The following is an extension of \thmref{t:stacks} to the case of truncatable (but not necessarily quasi-compact)
stacks:

\begin{thm} \label{t:stacks nonqc}
Let $\CX_1$ and $\CX_2$ be truncatable stacks, and let 
$$\CQ\in \Dmod(\CX_1)\otimes \Dmod(\CX_2)_{\on{co}}$$
be coherent. Assume that the corresponding functor 
$$\sF_{\CX_1\to \CX_2,\CQ}:\Dmod(\CX_1)_{\on{co}}\to \Dmod(\CX_2)_{\on{co}}$$
preserves compactness. Then we have a canonical isomorphism
$$\psId_{\CX_1}\circ (\sF_{\CX_1\to \CX_2,\CQ})^R\simeq \sF_{\CX_2\to \CX_1,\BD_{\CX_1\times \CX_2}^{\on{Ve}}(\CQ)}$$
as functors 
$$\Dmod(\CX_2)_{\on{co}}\to \Dmod(\CX_1).$$
\end{thm}

\sssec{}

By passing to the dual functors, we obtain: 

\begin{cor} \label{c:stacks nonqc}
Let $\CX_1$ and $\CX_2$ be truncatable stacks, and let 
$$\CQ\in \Dmod(\CX_1)\otimes \Dmod(\CX_2)_{\on{co}}$$
be coherent. Assume that the corresponding functor 
$$\sF_{\CX_1\to \CX_2,\CQ}:\Dmod(\CX_1)_{\on{co}}\to \Dmod(\CX_2)_{\on{co}}$$
preserves compactness.  Then we have a canonical isomorphism
$$(\sF_{\CX_1\to \CX_2,\CQ})^{\on{op}}\circ \psId_{\CX_1}\simeq \sF_{\CX_1\to \CX_2,\BD_{\CX_1\times \CX_2}^{\on{Ve}}(\CQ)}$$
as functors
$$\Dmod(\CX_1)_{\on{co}}\to \Dmod(\CX_2).$$
\end{cor}

\begin{rem}
Note that \thmref{t:stacks nonqc} \emph{does not} fit into the paradigm of \thmref{t:general}. Indeed, we start with
$\bC_i=\Dmod(\CX_i)_{\on{co}}$, $i=1,2$ and a functor $\sF:\bC_1\to \bC_2$, and while \thmref{t:general} talks about
an isomorphism between two functors $\bC_2\to \bC_1$, in \thmref{t:stacks nonqc}, the target category
is no longer $\bC_2=\Dmod(\CX_2)_{\on{co}}$, but rather $\Dmod(\CX_2)$.
\end{rem}

\ssec{Proof of \thmref{t:stacks nonqc}}



\sssec{}    \label{sss:source qc}

We shall first consider the case when $\CX_1$ is quasi-compact. In this case we will not distinguish between $\Dmod(\CX_1)$ and 
$\Dmod(\CX_1)_{\on{co}}$. 

\medskip

Using the equivalence \eqref{e:as colimit}, in order to prove the theorem, it suffices to construct a compatible family of isomorphisms of functors
\begin{multline} \label{e:restr to U_2}
\psId_{\CX_1}\circ (\sF_{\CX_1\to \CX_2,\CQ})^R \circ (j_2)_{\on{co},\bullet}\simeq \\
\simeq \sF_{\CX_2\to \CX_1,\BD_{\CX_1\times \CX_2}^{\on{Ve}}(\CQ)} \circ (j_2)_{\on{co},\bullet},\quad
(j_2:U_2\hookrightarrow \CX_2) \in \on{Ctrnk}(\CX_2).
\end{multline}

We have: 
$$(\sF_{\CX_1\to \CX_2,\CQ})^R \circ (j_2)_{\on{co},\bullet}\simeq 
\left((j_2)_{\on{co}}^\bullet \circ \sF_{\CX_1\to \CX_2,\CQ}\right)^R\simeq
(\sF_{\CX_1\to U_2,\CQ_U})^R,$$
where
$$\CQ_U:=\left(\on{Id}_{\Dmod(\CX_1)} \otimes (j_2)_{\on{co}}^\bullet \right)(\CQ).$$

The functor
$$\sF_{\CX_2\to \CX_1,\BD_{\CX_1\times \CX_2}^{\on{Ve}}(\CQ)} \circ (j_2)_{\on{co},\bullet}$$
is given by the kernel
$$(\on{Id}_{\Dmod(\CX_1)} \otimes ((j_2)_{\on{co},\bullet})^\vee)(\BD_{\CX_1\times \CX_2}^{\on{Ve}}(\CQ))\simeq
(\on{Id}_{\Dmod(\CX_1)} \otimes j_2^\bullet)(\BD_{\CX_1\times \CX_2}^{\on{Ve}}(\CQ)),$$
which by the definition of $\BD_{\CX_1\times \CX_2}^{\on{Ve}}$ identifies with
$$\BD_{\CX_1\times U_2}^{\on{Ve}}(\CQ_U).$$

\medskip

Hence, both sides of \eqref{e:restr to U_2} identify with the corresponding functors when we replace
$\CX_2$ by $U_2$ and $\CQ$ by $\CQ_U$. In this case, the required isomorhism for  \eqref{e:restr to U_2} 
follows from \thmref{t:stacks}. Furthermore, this system of isomorphisms is compatible under the restrictions
for $U'_2\hookrightarrow U_2$.

\medskip

This establishes the isomorphism of the theorem in the case when $\CX_1$ is quasi-compact. 

\sssec{}

Let now $\CX_1$ be general truncatable. By the definition of the category $\Dmod(\CX_1)$, it is enough to show that for every
quasi-compact open co-truncative $U_1\overset{j_1}\hookrightarrow \CX_1$, there exists a canonical 
isomorphism of functors  
\begin{equation} \label{e:restr to U_1}
j_1^\bullet\circ \psId_{\CX_1}\circ (\sF_{\CX_1\to \CX_2,\CQ})^R\simeq 
j_1^\bullet \circ \sF_{\CX_2\to \CX_1,\BD_{\CX_1\times \CX_2}^{\on{Ve}}(\CQ)},
\end{equation}
compatible with the restriction maps under $U'_1\hookrightarrow U''_1$.

\begin{lem} \label{l:on open} 
For a truncatable stack $\CX$ and $(U\overset{j}\hookrightarrow \CX)\in \on{Ctrnk}$
there is a canonical isomorphism of functors
$$j^\bullet\circ \psId_{\CX}\simeq \psId_{U}\circ j^?,\quad
\Dmod(\CX)_{\on{co}}\to \Dmod(U).$$
\end{lem}

(Note that if $\CX$ were quasi-compact, the assertion of the lemma is a particular case of \corref{c:!and*}.)

\begin{proof}

The functor $j^\bullet\circ \psId_{\CX}$ is given by the kernel
$$(j\times \on{id}_{\CX})^\bullet\left((\Delta_{\CX})_!(k_{\CX})\right),$$
which by base change identifies with
$$((\on{id}_{U}\times j)\circ \Delta_{U})_!(k_{U}),$$
i.e., $(\on{id}_{U}\times j)_! \circ (\Delta_{U})_!(k_{U})$. 

\medskip

The functor $\psId_{U}\circ j^?$ is given by the kernel
$$(\on{Id}_{\Dmod(U)}\otimes (j^?)^\vee)\left((\Delta_{U})_!(k_{U})\right)\simeq
(\on{Id}_{\Dmod(U)}\otimes j_!)\left((\Delta_{U})_!(k_{U})\right).$$

We note that the functor $\on{Id}_{\Dmod(U)}\otimes j_!$ is the left adjoint of
$\on{Id}_{\Dmod(U)}\otimes j^\bullet$, and hence identifies with
$(\on{id}_{U}\times j)_!$. 

\end{proof}

\sssec{}

Hence, we obtain that the left-hand side in \eqref{e:restr to U_1} identifies canonically with
$$\psId_{U_1}\circ j_1^? \circ (\sF_{\CX_1\to \CX_2,\CQ})^R,$$
which we further rewrite as
$$\psId_{U_1}\circ (\sF_{\CX_1\to \CX_2,\CQ}\circ (j_1)_{\on{co},\bullet})^R.$$
Note that the functor 
$$\sF_{\CX_1\to \CX_2,\CQ}\circ (j_1)_{\on{co},\bullet}:\Dmod(U_1)\to \Dmod(\CX_2)_{\on{co}}$$
preserves compactness since $(j_1)_{\on{co},\bullet}$ does. The above functor 
is given by the kernel
$$\CQ_U:=(((j_1)_{\on{co},\bullet})^\vee\otimes \on{Id}_{\Dmod(\CX_2)_{\on{co}}})(\CQ)
\simeq 
(j_1^\bullet \otimes \on{Id}_{\Dmod(\CX_2)_{\on{co}}})(\CQ)\in \Dmod(U_1)\otimes \Dmod(\CX_2)_{\on{co}}.$$

\medskip

Now, the functor $$j_1^\bullet \circ \sF_{\CX_2\to \CX_1,\BD_{\CX_1\times \CX_2}^{\on{Ve}}(\CQ)},$$
appearing in the right-hand side of \eqref{e:restr to U_1}, is given by the kernel
$$(j_1^\bullet\otimes \on{Id}_{\Dmod(\CX_2)})\left(\BD_{\CX_1\times \CX_2}^{\on{Ve}}(\CQ)\right)
\simeq  \BD_{U_1\times \CX_2}^{\on{Ve}}(\CQ_U).$$

\medskip 

Hence, both sides in \eqref{e:restr to U_1} identify with the corresponding functors when we 
replace $\CX_1$ by $U_1$ and $\CQ$ by $\CQ_U$. We define the isomorphism in \eqref{e:restr to U_1} 
to be the isomorphism of \secref{sss:source qc} for the stack $U_1\times \CX_2$.
These identifications are
compatible under further restrictions for $U'_1\hookrightarrow U''_1$.

\medskip

This establishes the required isomorphism 
$$\psId_{\CX_1}\circ (\sF_{\CX_1\to \CX_2,\CQ})^R\simeq 
\sF_{\CX_2\to \CX_1,\BD_{\CX_1\times \CX_2}^{\on{Ve}}(\CQ)}.$$

\qed

\ssec{Applications}

\sssec{}

First, passing to dual functors in \lemref{l:on open}, we obtain:
\begin{cor} \label{c:on open}
For $(U\overset{j}\hookrightarrow \CX)\in \on{Ctrnk}(\CX)$
there is a canonical isomorphism of functors
$$\psId_{\CX}\circ j_{\on{co},\bullet}\simeq j_!\circ \psId_{U},\quad
\Dmod(U)\to \Dmod(\CX).$$
\end{cor}

\sssec{}

Let $\CX$ be a truncatable stack. We shall say that $\CX$ is miraculous if the functor
$\psId_\CX$ is an equivalence. 

\begin{prop}
For a stack $\CX$ the following conditions are equivalent:

\smallskip

\noindent{\em(a)} The $\CX$ is miraculous.

\smallskip

\noindent{\em(b)} Every co-truncatable quasi-compact open substack of $\CX$ is miraculous.

\smallskip

\noindent{\em(c)} There is a cofinal family in $\on{Ctrnk}(\CX)$ consisting of miraculous stacks.

\end{prop}

\begin{proof}

We reproduce the proof from \cite[Lemma 4.5.7]{DrGa2}. Assume that $\CX$ is miraculous, and let
$(j:U\to \CX)\in \on{Ctrnk}(\CX)$.  

\medskip

Let us first show that the functor $\psId_U$ has a left inverse.
For this it is enough to show that the composition $j_!\circ \psId_U$ has a left inverse.
Taking into account the isomorphism of \corref{c:on open},
it suffices to show that each of the functors $\psId_{\CX}$ and $j_{\on{co},\bullet}$ admits a left inverse. 
For $\psId_\CX$ this follows from the assumption that $\CX$ is miraculous. For $j_{\on{co},\bullet}$, the
left inverse is $j_{\on{co}}^\bullet$.

\medskip

Now, if $(\psId_U)^{-1,L}$ is the left inverse of $\psId_U$, passing to dual functors in
$$(\psId_U)^{-1,L}\circ \psId_U\simeq \on{Id}_{\Dmod(U)},$$
we obtain that $((\psId_U)^{-1,L})^\vee$ is the right inverse of $\psId_U$. Hence,
$\psId_U$ is an equivalence.

\medskip

The implication  (b) $\Rightarrow$ (c) is tautological. 

\medskip

The implication (c) $\Rightarrow$ (a) follows
from \lemref{l:on open}, since the functors $\psId_U$ define an equivalence between the limits
$$\underset{\on{Ctrnk}^{\on{op}}}{lim}\, \Dmod^?\to \underset{\on{Ctrnk}^{\on{op}}}{lim}\, \Dmod^\bullet.$$

\end{proof}

The proof of the following result is given in \cite{duality}:

\begin{thm}
The stack $\Bun_G$ of principal $G$-bundles on a complete smooth curve $X$, where $G$ is a reductive group,
is miraculous.
\end{thm}

\sssec{}

We shall say that a truncatable stack $\CX$ is mock-proper if the functor $(p_\CX)_\blacktriangle$
(defined in \secref{sss:black nonqc}) preserves compactness.  

\medskip

By \lemref{l:conjugate}, $\CX$ mock-proper if and only if the functor 
$$(p_\CX)_!:\Dmod(\CX)\to \on{pt},$$
right adjoint to $p_\CX^!$, is defined.

\medskip

The following assertion is proved in \cite[Corollary 4.3.2]{Contr}:

\begin{prop}
The stack $\Bun_G$ is mock-proper.
\end{prop}

\sssec{}  \label{sss:contrunk proper}

The following is immediate from the definitions:

\begin{lem}
For a stack $\CX$ the following conditions are equivalent:

\smallskip

\noindent{\em(a)} The $\CX$ is mock-proper.

\smallskip

\noindent{\em(b)} Every quasi-compact open co-truncative substack of $\CX$ is mock-proper.

\smallskip

\noindent{\em(c)} There is a cofinal family in $\on{Ctrnk}(\CX)$ consisting of mock-proper stacks.

\end{lem}

Hence, we obtain:

\begin{cor}  \label{c:cotrunk of BunG}
Every quasi-compact open co-truncative substack of $\Bun_G$ is mock-proper.
\end{cor} 

\sssec{}

The next assertion is proved in the same way as \propref{p:mock-proper}:

\begin{prop}
Let $\CX$ be mock-proper and smooth of dimension $n$. Then there exists a canonical isomorphism of functors
$$(p_\CX)_\blacktriangle \simeq (p_\CX)_!\circ \psId_\CX[2n].$$
\end{prop} 

Passing to dual functors, and using \lemref{l:conjugate}, we obtain:

\begin{cor} \label{c:mock-proper nonqc}
Let $\CX$ be mock-proper and smooth of dimension $n$. Then we have a canonical isomorphism of functors
$$p_\CX^!\simeq \psId_\CX\circ ((p_\CX)_\blacktriangle)^R[2n].$$
\end{cor}

For a mock-proper stack, we shall denote by $\omega_{\CX,\on{mock}}$ the object
$$((p_\CX)_\blacktriangle)^R(k)\in  \Dmod(X)_{\on{co}}.$$

Hence, \corref{c:mock-proper nonqc} can be reformulated as saying that for $\CX$ smooth of dimension $n$ we have:
$$\psId_\CX(\omega_{\CX,\on{mock}})[2n]\simeq \omega_\CX.$$

\ssec{A bizarre object in $\Dmod(\Bun_G)_{\on{co}}$}  \label{ss:bizarre}

\sssec{}

We consider the object 
$$\omega_{\Bun_G,\on{mock}}\in \Dmod(\Bun_G)_{\on{co}}.$$

The goal of this subsection and the next is to prove the following assertion:

\begin{thm} \label{t:weird in ker}
Let $G$ be a reductive group with a non-trivial semi-simple part (i.e., $G$ is not a torus). Then
the object $\omega_{\Bun_G,\on{mock}}$ belongs to the kernel of the functor
$$\psId_{\Bun_G}^{\on{naive}}:\Dmod(\Bun_G)_{\on{co}}\to \Dmod(\Bun_G).$$
\end{thm} 
(See \secref{sss:naive} where the functor $\psId_{\CX}^{\on{naive}}$ is introduced.)

\medskip

In order to prove this theorem we will use a description of the object $\omega_{\Bun_G,\on{mock}}$, which
is valid for any reductive group.

\sssec{}

Let us recall the setting of \cite[Sect. 4.1.1]{Contr}. We let $\Gr_{G,\Ran(X)}$ denote the prestack, which is
the Ran version of affine Grassmannian for $G$. Let $\pi$ denote the canonical map
$$\Gr_{G,\Ran(X)}\to \Bun_G.$$

The following is \cite[Theorem 4.1.6]{Contr}:

\begin{thm}  \label{t:contr}
The functor $\pi^!:\Dmod(\Bun_G)\to \Dmod(\Gr_{G,\Ran(X)})$ is fully faithful.
\end{thm}

\sssec{}

We recall that the pre-stack $\Gr_{G,\Ran(X)}$ is by definition the colimit 
\begin{equation} \label{e:Gr as colimit}
\underset{i\in I}{colim}\, Z_i,
\end{equation}
where $Z_i$ are proper schemes, and $I$ is some index category. In particular, for $\alpha:i\to j$, the
corresponding map $f_\alpha:Z_i\to Z_j$ is proper. We let $f_i$ denote the corresponding map $Z_i\to \Gr_{G,\Ran(X)}$. 

\medskip

The category $\Dmod(\Gr_{G,\Ran(X)})$ is the limit
\begin{equation} \label{e:on Gr as limit}
\underset{i\in I^{\on{op}}}{lim}\, \Dmod(Z_i),
\end{equation}
where for $(\alpha:i\to j)\in I$, the functor $\Dmod(Z_j)\to \Dmod(Z_i)$ is $f_\alpha^!$.
The corresponding evaluation functor $\Dmod(\Gr_{G,\Ran(X)})\to \Dmod(Z_i)$ is $f_i^!$. 

\medskip

Hence, by \cite[Proposition 1.7.5]{DrGa2}, we have a canonical equivalence
\begin{equation} \label{e:on Gr as colimit}
\Dmod(\Gr_{G,\Ran(X)})\simeq \underset{i\in I}{colim}\, \Dmod(Z_i),
\end{equation}
where for $(\alpha:i\to j)\in I$, the functor $\Dmod(Z_i)\to \Dmod(Z_j)$ is $(f_\alpha)_\bullet$.
For $i\in I$, the corresponding functor $\Dmod(Z_i)\to \Dmod(\Gr_{G,\Ran(X)})$ will be denoted $(f_i)_\bullet$,
and it is the left adjoint of $f_i^!$.

\medskip

In particular, by \cite[Proposition 1.8.3]{DrGa2}, the Verdier duality equivalences
$$\bD_{Z_i}^{\on{Ve}}:\Dmod(Z_i)^\vee\to \Dmod(Z_i)$$
give rise to an equivalence
$$\bD_{\Gr_{G,\Ran(X)}}^{\on{Ve}}:\Dmod(\Gr_{G,\Ran(X)})^\vee\simeq \Dmod(\Gr_{G,\Ran(X)});$$
under which we have:
$$(f_i^!)^\vee\simeq (f_i)_\bullet.$$

\sssec{}

Let 
$$\pi_\bullet: \Dmod(\Gr_{G,\Ran(X)})\to \Dmod(\Bun_G)_{\on{co}}$$
denote the functor dual to $\pi^!$ under the identifications $\bD_{\Gr_{G,\Ran(X)}}^{\on{Ve}}$ and $\bD^{\on{Ve}}_{\Bun_G}$.

\medskip

The functor $\pi_\bullet$ can be described more explicity as follows. By \eqref{e:on Gr as colimit},
the datum of $\pi_\bullet$ is equivalent to a compatible collection of functors
$$(\pi\circ f_i)_\bullet:\Dmod(Z_i)\to \Dmod(\Bun_G)_{\on{co}}.$$

\medskip

For each $i$, the category of factorizations of the map $\pi\circ f_i$ as
\begin{equation} \label{e:factorization}
Z_i\overset{f_{i,U}}\longrightarrow  U\overset{j}\hookrightarrow \Bun_G, \quad U\in \on{Ctrnk}(\Bun_G).
\end{equation}
is cofinal in $\on{Ctrnk}(\Bun_G)$, and hence, is contractible. 

\medskip

The sought-for functor $(\pi\circ f_i)_\bullet$ is 
$$j_{\on{co},\bullet}\circ (f_{i,U})_\bullet$$
for some/any factorization \eqref{e:factorization}. 

\medskip

In the sequel, we will use the following version of the projection formula, which follows immediately from the defintions:

\begin{lem} \label{l:proj formula}
For $\CF\in \Dmod(\Bun_G)$ and $\CF'\in \Dmod(\Gr_{G,\Ran(X)})$ there is a canonical isomorphism
$$\CF\sotimes \pi_\bullet(\CF')\simeq \pi_\bullet(\pi^!(\CF)\sotimes \CF'),$$
where $\sotimes$ in the left-hand side is understood in the sense of \secref{sss:sotimes}.
\end{lem}

\sssec{}

We claim:

\begin{thm} \label{t:weird}
There exists a canonical isomorphism
$$\omega_{\Bun_G,\on{mock}}\simeq \pi_\bullet(\omega_{\Gr_{G,\Ran(X)}}).$$
\end{thm}

\begin{proof}

We need to establish a functorial isomorphism
\begin{equation} \label{e:bizarre} 
\CMaps_{\Dmod(\Bun_G)_{\on{co}}}(\CF,\omega_{\Bun_G,\on{mock}})\simeq 
\CMaps_{\Dmod(\Bun_G)_{\on{co}}}\left(\CF,\pi_\bullet(\omega_{\Gr_{G,\Ran(X)}})\right).
\end{equation}
for $\CF\in (\Dmod(\Bun_G)_{\on{co}})^c$. 

\medskip

By definition, the left-hand side in \eqref{e:bizarre} can be rewritten as
$$\CMaps_{\Vect}\left((p_{\Bun_G})_\blacktriangle(\CF),k\right),$$
and further, by \lemref{l:conjugate}, as
\begin{equation} \label{e:bizarre1} 
(p_{\Bun_G})_!(\BD_{\Bun_G}^{\on{Ve}}(\CF)),
\end{equation}
where 
$$\BD_{\Bun_G}^{\on{Ve}}:(\Dmod(\Bun_G)_{\on{co}}^c)^{\on{op}}\simeq \Dmod(\Bun_G)^c$$
is the equivalence indiced by 
$$\bD_{\Bun_G}^{\on{Ve}}:(\Dmod(\Bun_G)_{\on{co}})^\vee\simeq \Dmod(\Bun_G).$$

\medskip

We rewrite the right-hand side of \eqref{e:bizarre} using \lemref{l:Hom as coh}
as
$$(p_{\Bun_G})_\blacktriangle\left(\BD_{\Bun_G}^{\on{Ve}}(\CF)\sotimes \pi_\bullet(\omega_{\Gr_{G,\Ran(X)}})\right).$$

Using \lemref{sss:sotimes}, we further rewrite it as
$$(p_{\Bun_G})_\blacktriangle\circ \pi_\bullet(\pi^!(\BD_{\Bun_G}^{\on{Ve}}(\CF))\sotimes \omega_{\Gr_{G,\Ran(X)}})\simeq 
(p_{\Bun_G})_\blacktriangle\circ \pi_\bullet(\pi^!(\BD_{\Bun_G}^{\on{Ve}}(\CF))),$$
and hence as
\begin{equation} \label{e:bizarre2} 
(p_{\Gr_{G,\Ran(X)}})_\bullet (\pi^!(\BD_{\Bun_G}^{\on{Ve}}(\CF))).
\end{equation}

Comparing \eqref{e:bizarre1} and \eqref{e:bizarre2}, the assertion of the theorem follows from the next lemma:

\begin{lem} \label{l:cohomology BunG}
For $\CF'\in \Dmod(\Bun_G)$ there is a canonical isomorphism 
$$(p_{\Bun_G})_!(\CF')\simeq (p_{\Gr_{G,\Ran(X)}})_\bullet(\pi^!(\CF')).$$
\end{lem}

\end{proof}

\begin{proof}[Proof of \lemref{l:cohomology BunG}]

It is enough to establish the isomorphism in question in the case when $\CF'\in \Dmod(\Bun_G)^c$.
We will show that
$$\CMaps_{\Vect}((p_{\Bun_G})_!(\CF'),V)\simeq \CMaps_{\Vect}((p_{\Gr_{G,\Ran(X)}})_\bullet(\pi^!(\CF')),V), \quad V\in \Vect.$$

We rewrite the left-hand side and the right-hand side as
$$\CMaps_{\Dmod(\Bun_G)}(\CF',p_{\Bun_G}^!(V)) \text{ and } 
\CMaps_{\Dmod(\Gr_{G,\Ran(X)})}(\pi^!(\CF'),p_{\Gr_{G,\Ran(X)}}^!(V)),$$
respectively, and the required assertion follows from  \thmref{t:contr}.

\end{proof}

\ssec{Proof of \thmref{t:weird in ker}}

\sssec{}

Taking into account \thmref{t:weird}, we need to show that the object
$$\psId_{\Bun_G}^{\on{naive}}\circ \pi_\bullet(\omega_{\Gr_{G,\Ran(X)}})\in \Dmod(\Bun_G)$$
is zero.

\medskip 

First, we recall that the prestack $\Gr_{G,\Ran(X)}$ is the colimit of ind-schemes, denoted $\Gr_{G,X^n}$, see \cite[Sect. 4.1.1]{Contr}.
We will show that for every $n$
\begin{equation} \label{e:req van}
\psId_{\Bun_G}^{\on{naive}}\circ (\pi_n)_\bullet(\omega_{\Gr_{G,X^n}})=0,
\end{equation}
where $\pi_n$ denotes the map $\Gr_{G,X^n}\to \Bun_G$.

\sssec{}

Recall that for an ind-scheme (of ind-finite type) $\CX$ the category $\Dmod(\CX)$ carries a canonical t-structure,
see \cite[Sect. 4.3]{Crys}. It is characterized by the property that if 
$$\CX\simeq \underset{i\in I}{colim}\, X_i,$$
where $f_i:X_i\to \CX$ are closed subschemes of $\CX$, the category $\Dmod(\CX)^{\leq 0}$ is generated under
colimits by the essential images of the categories $\Dmod(X_i)$ under the functors $(f_i)_\bullet$. 

\medskip

The assertion of \eqref{e:req van} follows from the combination of the following two statements:

\begin{prop} \label{p:Gr to BunG affine}
For a reductive group $G$, the functor 
$$\psId_{\Bun_G}^{\on{naive}}\circ (\pi_n)_\bullet:\Dmod(\Gr_{G,X^n})\to \Dmod(\Bun_G)$$
has cohomological amplitude bounded on the right by $n$. 
\end{prop}

\begin{prop} \label{p:omega on Gr}
If the semi-simple compnent of $G$ is non-trivial, the object
$$\omega_{\Gr_{G,X^n}}\in \Dmod(\Gr_{G,X^n})$$
is infinitely connective, i.e., belongs to $\Dmod(\Gr_{G,X^n})^{\leq -n}$ for any $n$.
\end{prop} 

\sssec{Proof of \propref{p:Gr to BunG affine}}

Let us write $\Gr_{G,X^n}$ as 
$$\underset{i\in I}{colim}\, Z_i,$$
where $Z_i$'s are closed subschemes of $\Gr_{G,X^n}$. 

\medskip

By the definition of the t-structure on $\Dmod(\Gr_{G,X^n})$, it is enough to show that each of functors
$$\psId_{\Bun_G}^{\on{naive}}\circ (\pi_n)_\bullet\circ (f_i)_\bullet:\Dmod(Z_i)\to \Dmod(\Bun_G)$$
has cohomological amplitude bounded on the right by $n$. 

\medskip

However, it follows from the definitions, that the above composed functor is the usual direct image functor
for the map 
$$(\pi_n\circ f_i):Z_i\to \Bun_G.$$
We factor the above map as a composition
$$Z_i\overset{(s_n\times f_i)\times (\pi_n\circ f_i)}\longrightarrow
X^n\times \Bun_G\to \Bun_G,$$
where $s_n$ is the natural projection $\Gr_{G,X^n}\to X^n$. 

\medskip

The required assertion follows from the fact that the map
$$(s_n\times f_i)\times (\pi_n\circ f_i):Z_i\to X^n\times \Bun_G$$
is schematic and affine. The latter follows from the fact that the map 
$$s_n\times \pi:\Gr_{G,X^n}\to X^n\times \Bun_G$$
is ind-affine.

\qed

\sssec{Proof of \propref{p:omega on Gr}}

Consider the diagonal stratification of $X^n$. It is easy to see that it is sufficient to show
that the !-restriction of $\omega_{\Gr_{G,X^n}}$ to the preimage of each stratum is
infinitely connective. 

\medskip

Using the factorization property of $\Gr_{G,X^n}$ over $X^n$, the assertion is
further reduced to the case when instead of $\Gr_{G,X^n}$ we consider $\Gr_{G,x}$,
i.e., its local version at some point $x\in X$.

\medskip

In the latter case we can assume that $X=\BP^1$ and $x=\infty\in \BP^1$. Denote the
corresponding ind-scheme simply by $\Gr_G$. We need to show that $\omega_{\Gr_G}$
is infinitely connective. 

\medskip

We have the following lemma, proved below:

\begin{lem} \label{l:t on ind}
For an ind-scheme $\CX$, the t-structure on $\Dmod(\CX)$ is local in the Zariski
topology, i.e., if $\CX=\underset{i}\cup\, U_i$, where $U_i\subset \CX$ are Zariski open
subschemes, then an object $\CF\in \Dmod(\CX)$ is connective/coconnective if and
only if its restrictions to $U_i$ have this property.
\end{lem}

Let
$$\Gr^0_G\subset \Gr_G$$
be the open Bruhat cell, i.e., the preimage of $\on{pt}/G\subset \Bun_G$ under the map $\pi$. The entire
ind-scheme $\Gr_G$ can be covered by translates of $\Gr^0_G$ by means of the loop group.
Hence, by \lemref{l:t on ind}, it is sufficient to show that $\omega_{\Gr^0_G}$ is infinitely connective.

\medskip

However, it is known that for a reductive group with a nontrivial semi-simple part, the ind-scheme 
$\Gr^0_G$ is isomorphic to 
$$\BA^\infty\simeq \underset{k\geq 0}{colim}\, \BA^k.$$

Now, for any $n$, we can write
$$\omega_{\BA^\infty}\simeq  \underset{m\geq n}{colim}\, (i_m)_\bullet(\omega_{\BA^m}),$$
where $i_m:\BA^m\to \BA^\infty$. The functors $(i_m)_\bullet$ are t-exact, and 
$$\omega_{\BA^m}\in \Dmod(\BA^m)^\heartsuit[m]\subset \Dmod(\BA^m)^{\leq -n},$$
since $m\geq n$.  Hence, 
$$\omega_{\BA^\infty}\in \Dmod(\BA^\infty)^{\leq -n},$$
as required. 

\qed 

\sssec{Proof of \lemref{l:t on ind}}

First, we note that the functor of restriction
$$\Dmod(\CX)\to \Dmod(U)$$
for an open embedding $U\hookrightarrow \CX$ is t-exact. 

\medskip

Let us show that the property of being coconnective is local in the Zariski topology. 
I.e.,
let $\CF\in \Dmod(\CX)$ be such that $\CF|_{U_i}\in \Dmod(U_i)^{> 0}$, and we need to
show that $\CF\in \Dmod(\CX)^{>0}$. 

\medskip

I.e., we need to show that for $\CF'\in \Dmod(\CX)^{\leq 0}$, we have
$$\CMaps_{\Dmod(\CX)}(\CF',\CF)\in \Vect^{>0}.$$

Let $U^\bullet$ be the \v{C}ech nerve of the cover $\underset{i}\cup\, U_i\to \CX$. 

The category $\Dmod(\CX)$ satisfies Zariski descent. Hence, $\CMaps_{\Dmod(\CX)}(\CF',\CF)$
is the totalization of a co-simplicial object of $\Vect$ whose $n$-th term is 
$$\CMaps_{\Dmod(U^n)}(\CF'|_{U^n},\CF|_{U^n}).$$

However, $\CF'|_{U^n}\in \Dmod(U^n)^{\leq 0}$ and $\CF|_{U^n}\in \Dmod(U^n)^{> 0}$, and the assertion
follows, as the functor of totalization is left t-exact.

\medskip

The proof in the connective case is similar.

\qed

\end{document}